\documentclass[11pt, leqno]{amsart}

\usepackage{graphicx}
\usepackage{amssymb,amsmath,amsthm,amscd,accents}
\usepackage{mathrsfs}
\usepackage{lscape}
\usepackage{enumerate}
\usepackage{upgreek}
\usepackage[usenames,dvipsnames]{color}
\usepackage[colorlinks=true, pdfstartview=FitV,
 linkcolor=blue,citecolor=blue,urlcolor=blue]{hyperref}
\usepackage{tikz}
\usetikzlibrary{arrows.meta,arrows}
\usepackage{bm}
\usepackage{marginnote}
\usepackage{verbatim}
\usepackage[normalem]{ulem}  

\usepackage[all]{xy}

\allowdisplaybreaks[3]

\newcommand{\arxiv}[1]{\href{http://arxiv.org/abs/#1}{\texttt{arXiv:#1}}}

\newcommand{\nc}{\newcommand}
\numberwithin{equation}{section}

\newenvironment{red}{\relax\color{red}}{\relax}
\newenvironment{blue}{\relax\color{blue}}{\hspace*{.5ex}\relax}
\newenvironment{purple}{\relax\color{blue}}{\hspace*{.5ex}\relax}

\newenvironment{magenta}{\relax\color{magenta}}{\hspace*{.5ex}\relax}

\newcommand{\bep}{\begin{purple}}
\newcommand{\eep}{\end{purple}}

\newcommand{\ber}{\begin{red}}
\newcommand{\er}{\end{red}}
\newcommand{\beb}{\begin{blue}}
\newcommand{\eb}{\end{blue}}
\newcommand{\bem}{\begin{magenta}}
\newcommand{\eem}{\end{magenta}}
\newcommand{\berm}[1]{\begin{red}{}\marginnote{\fbox{\scshape\lowercase{M}}}%
#1}  

\newcommand{\bero}[1]{\begin{red}{}\marginnote{\fbox{\scshape\lowercase{O}}}%
#1}

\nc{\hs}{\hspace*}
\nc{\ms}{\mspace}

\nc{\qR}[1]{\ttq_{\mspace{-2mu}\raisebox{-.8ex}{${\scriptstyle{#1}}$}}}

\theoremstyle{plain}
\newtheorem{lemma}{Lemma}[section]
\newtheorem{proposition}[lemma]{Proposition}
\newtheorem{theorem}[lemma]{Theorem}

\newtheorem{corollary}[lemma]{Corollary}
\newtheorem{conjecture}{Conjecture}

\newtheorem{Conj}{Conjecture}

\newtheorem{thmx}{Theorem}

\newtheorem{corx}{Corollary}

\theoremstyle{definition}
\newtheorem{remark}[lemma]{Remark}
\newtheorem{example}[lemma]{Example}

\newtheorem{definition}[lemma]{Definition}

\newcommand{\leN}{ \preceq_{\mspace{-2mu}\raisebox{-.5ex}{\scalebox{.6}{${\rm N}$}}} }
\newcommand{\lN}{ \prec_{\mspace{-2mu}\raisebox{-.5ex}{\scalebox{.6}{${\rm N}$}}} }

\renewcommand{\le}{\leqslant}
\renewcommand{\ge}{\geqslant}
\renewcommand{\preceq}{\preccurlyeq}

\newcommand{\st}{\mathop{\mbox{\normalsize$*$}}\limits}

\newcommand{\scup}{\mathop{\mbox{\normalsize$\cup$}}\limits}
\newcommand{\hconv}{\mathbin{\scalebox{.9}{$\nabla$}}}
\newcommand{\sconv}{\mathbin{\scalebox{.9}{$\Delta$}}}

\newcommand{\seteq}{\mathbin{:=}}
\newcommand{\conv}{\mathop{\mathbin{\mbox{\large $\circ$}}}}
\newcommand{\soplus}{\mathop{\mbox{\normalsize$\bigoplus$}}\limits}
\newcommand{\ev}{{\operatorname{ev}}}
\newcommand{\tens}{\mathop\otimes}

\newcommand{\lxi}{{  {}^\xi  } \hspace{-.2ex}}

\newcommand{\Mod}{\text{-}\mathrm{Mod}}
\newcommand{\gmod}{\text{-}\mathrm{gmod}}

\newcommand{\ex}{\mathrm{ex}}
\newcommand{\fr}{\mathrm{fr}}

\newcommand{\seed}{\scrS}

\newcommand{\g}{\mathfrak{g}}

\newcommand{\n}{\mathfrak{n}}

\newcommand{\C}{\mathbb{C}}
\newcommand{\Q}{\mathbb{Q}}
\newcommand{\Z}{\mathbb{Z}\ms{1mu}}
\newcommand{\N}{\mathbb{N}\ms{1mu}}

\newcommand{\al}{{\ms{1mu}\alpha}}
\newcommand{\ep}{\epsilon}
\newcommand{\la}{\lambda}
\newcommand{\be}{{\ms{1mu}\beta}}
\newcommand{\ga}{\gamma}

\newcommand{\La}{\Lambda}

\newcommand{\upal}{\upalpha}

\newcommand{\wt}{{\rm wt}}
\newcommand{\up}{{\rm up}}
\newcommand{\range}{{\rm range}}
\newcommand{\het}{{\rm ht}}
\newcommand{\mul}{{\rm mul}}
\newcommand{\de}{\mathfrak{d}}
\newcommand{\Dynkin}{\triangle}

\newcommand{\pair}[1]{  \langle \hspace{-.6ex} \langle #1 \rangle \hspace{-.6ex} \rangle  }

\newcommand{\cc}{ \textbf{\textit{c}}}
\newcommand{\ek}{ \textbf{\textit{e}}}
\newcommand{\ii}{ \textbf{\textit{i}}}
\newcommand{\TT}{ \textbf{\textit{T}}}
\newcommand{\sk}{ \textbf{\textit{s}}}
\newcommand{\pk}{ \textbf{\textit{p}}}

\newcommand*\circled[1]{ \fontsize{6}{6}\selectfont \tikz[baseline=(char.base)]{
  \node[shape=circle,draw,inner sep=0.4pt] (char) {#1};} \fontsize{12}{12}\selectfont }

\newcommand*\bcircled[1]{\fontsize{6}{6}\selectfont \tikz[baseline=(char.base)]{
    \node[shape=circle, fill= black, draw=black, text=white,  inner sep=0.4pt] (char) {#1};} \fontsize{12}{12}\selectfont}

\newcommand*\rcircled[1]{\fontsize{6}{6}\selectfont \tikz[baseline=(char.base)]{
    \node[shape=circle, fill= blue, draw=black, text=white, inner sep=0.4pt] (char) {#1};} \fontsize{12}{12}\selectfont}

\newcommand{\Hom}{\operatorname{Hom}}

\newcommand{\res}{\mathrm{res}}

\newcommand{\HOM}{\mathrm{H{\scriptstyle OM}}}

\newcommand{\hA}{\widehat{\calA}}

\newcommand{\tka}{\widetilde{\kappa}}
\newcommand{\tze}{\widetilde{\zeta}}

\newcommand{\tde}{\widetilde{\de}}
\newcommand{\tPsi}{\widetilde{\Psi}}

\newcommand{\tUptheta}{\widetilde{\Uptheta}}

\newcommand{\tii}{\widetilde{\ii}}
\newcommand{\tx}{\widetilde{x}}

\newcommand{\tB}{\widetilde{B}}

\newcommand{\tbfB}{\widetilde{\bfB}}
\newcommand{\tusfB}{\ms{1mu}\widetilde{\usfB}}
\newcommand{\tusfb}{\ms{1mu}\widetilde{\usfb}}

\newcommand{\tX}{\widetilde{X}}
\newcommand{\tF}{\widetilde{F}}
\newcommand{\tY}{\widetilde{Y}}

\newcommand{\tS}{\widetilde{S}}
\newcommand{\tG}{\widetilde{G}}

\newcommand{\tm}{\widetilde{m}}

\newcommand{\tLa}{\widetilde{\Lambda}}

\newcommand{\tD}{\widetilde{D}}

\newcommand{\hsfM}{\widehat{\sfM}}

\newcommand{\osfB}{\overline{\mathsf{B}}}
\newcommand{\usfB}{\underline{\mathsf{B}}}

\newcommand{\usfC}{\underline{\mathsf{C}}}
\newcommand{\uw}{{\underline{w}}}
\newcommand{\uup}{{\underline{\pk}}}

\newcommand{\ucalN}{\ms{1mu}\underline{\calN}}

\newcommand{\usfb}{{\underline{\sfb}}}
\newcommand{\usfc}{\underline{\sfc}}

\newcommand{\utm}{\underline{\tm}}
\newcommand{\us}{\underline{\sk}}

\newcommand{\um}{\underline{m}}
\newcommand{\ue}{\underline{\ek}}

\newcommand{\frakD}{\mathfrak{D}}

\newcommand{\frakK}{\mathfrak{K}}

\newcommand{\sfC}{\mathsf{C}}
\newcommand{\sfS}{\mathsf{S}}

\newcommand{\sfc}{\mathsf{c}}

\newcommand{\sfE}{\mathsf{E}}
\newcommand{\sfF}{\mathsf{F}}

\newcommand{\sfL}{\mathsf{L}}
\newcommand{\sfA}{\mathsf{A}}

\newcommand{\sfD}{\mathsf{D}}
\newcommand{\sfP}{\mathsf{P}}
\newcommand{\sfQ}{\mathsf{Q}}
\newcommand{\sfW}{\mathsf{W}}

\newcommand{\sfh}{\mathsf{h}}

\newcommand{\sfM}{\mathsf{M}}

\newcommand{\sfx}{\mathsf{x}}
\newcommand{\sfy}{\mathsf{y}}

\newcommand{\sfb}{\mathsf{b}}
\newcommand{\sfg}{\mathsf{g}}

\newcommand{\sfK}{\mathsf{K}}

\newcommand{\bbA}{\mathbb{A}}

\newcommand{\bbP}{\mathbb{P}}
\newcommand{\bbK}{\mathbb{K}}
\newcommand{\bbF}{\mathbb{F}}

\newcommand{\bfk}{\mathbf{k}}
\newcommand{\bfB}{\mathbf{B}}

\newcommand{\bfL}{\mathbf{L}}

\newcommand{\bfg}{\mathbf{g}}
\newcommand{\bfc}{\mathbf{c}}

\newcommand{\bfa}{\mathbf{a}}
\newcommand{\bfn}{\mathbf{n}}

\newcommand{\calQ}{\mathcal{Q}}

\newcommand{\calU}{\mathcal{U}}
\newcommand{\calA}{\mathcal{A}}
\newcommand{\calS}{\mathcal{S}}

\newcommand{\calK}{\mathcal{K}}
\newcommand{\calR}{\mathcal{R}}

\newcommand{\calX}{\mathcal{X}}
\newcommand{\calY}{\mathcal{Y}}
\newcommand{\calM}{\mathcal{M}}
\newcommand{\calN}{\mathcal{N}}

\newcommand{\calT}{\mathcal{T}}
\newcommand{\calW}{\mathcal{W}}
\newcommand{\calL}{\mathcal{L}}

\newcommand{\scrC}{\mathscr{C}}
\newcommand{\scrA}{\mathscr{A}}

\newcommand{\scrS}{\mathscr{S}}
\newcommand{\scrK}{\mathscr{K}}

\newcommand{\ttb}{\mathtt{b}}

\newcommand{\ttt}{\mathtt{t}}

\newcommand{\ttq}{{\ms{1mu}\mathtt{q}\ms{1mu}}}
\newcommand{\To}[1][{\hspace{2ex}}]{\xrightarrow{\,#1\,}}

\newlength{\mylength}
\newcommand*{\para}{%
  \rlap{\rotatebox{-30}{\rule[.05ex]{.4pt}{.77em}}}%
  \kern.04em%
  \rlap{\kern.36em\raisebox{0.649519052835em}{\rule{.6em}{.4pt}}}%
  \rule{.6em}{.4pt}\kern-.04em%
  \rotatebox{-30}{\rule[.05ex]{.4pt}{.77em}}}

\newcommand{\Rr}{\mathbf{r}}

\newcommand{\hDynkin}{{\widehat{\Dynkin}}}

\newcommand{\rmQ}{\mathrm{Q}}

\newcommand{\cm}{\sfC}
\newcommand{\wl}{\sfP}
\newcommand{\rl}{\sfQ}

\newcommand{\weyl}{\sfW}
\newcommand{\lan}{\langle}
\newcommand{\ran}{\rangle}

\newcommand{\isoto}[1][]{\mathop{\xrightarrow%
[{\raisebox{.3ex}[0ex][.3ex]{$\scriptstyle{#1}$}}]%
{{\raisebox{-.6ex}[0ex][-.6ex]{$\mspace{2mu}\sim\mspace{2mu}$}}}}}

\newcommand{\ee}{\end{enumerate}}
\newcommand{\bitem}{\begin{itemize}}
\newcommand{\eitem}{\end{itemize}}
\newcommand{\ben}{\begin{enumerate}[{\rm (1)}]}
\newcommand{\bnum}{\begin{enumerate}[{\rm (i)}]}
\newcommand{\bnump}{\begin{enumerate}[{\rm (i)$'$}]}
\newcommand{\bna}{\begin{enumerate}[{\rm (a)}]}
\newcommand{\bnA}{\begin{enumerate}[{\rm (A)}]}
\newcommand{\bc}{\begin{cases}}
\newcommand{\ec}{\end{cases}}

\newenvironment{myequation}
{\relax\setlength{\arraycolsep}{1pt}\begin{eqnarray}}
{\end{eqnarray}}
\newenvironment{myequationn}
{\relax\setlength{\arraycolsep}{1pt}\begin{eqnarray*}}
{\end{eqnarray*}}

\nc{\eq}{\begin{myequation}}
\nc{\eneq}{\end{myequation}}
\nc{\eqn}{\begin{myequationn}}
\nc{\eneqn}{\end{myequationn}}
\nc{\cl}{\colon}
\nc{\ake}[1][1ex]{\rule[-#1]{0ex}{1ex}}
\nc{\akew}[1][1ex]{\rule[-1ex]{#1}{0ex}}
\nc{\akeu}[1][1ex]{\rule[#1]{0ex}{1ex}}
\nc{\id}{\mathrm{id}}
\nc{\bl}{\bigl(}
\nc{\br}{\bigr)}
\nc{\qt}[1]{\quad\text{#1}}
\nc{\qtq}[1][{and}]{\quad\text{{#1}}\quad}
\nc{\eqs}[1]{\underset{\raisebox{.4ex}[.7ex][0ex]{$\scriptstyle{#1}$}}{=}}
\nc{\snoi}{\smallskip \noindent}
\nc{\mnoi}{\medskip \noindent}
\nc{\Mat}{\mathrm{Mat}}
\nc{\ol}{\overline}
\nc{\ul}{\underline}
\nc{\ang}[1]{\boldsymbol{\langle}{#1}\boldsymbol{\rangle}}
\nc{\rang}[1]{\boldsymbol{\langle}{#1}\boldsymbol{\rangle} \hspace{-.6ex} \boldsymbol{\rangle}}
\nc{\ba}{\begin{array}}
\nc{\ea}{\end{array}}
\nc{\noi}{\noindent}
\nc{\evq}{ \ev_{q=1}}

\nc{\yi}{x_{i}}
\nc{\yj}{x_{j}}
\nc{\yk}{x_{k}}

\nc{\yjm}{x_{j,m}}
\nc{\yjmp}{x_{j,m+1}}
\nc{\yjmm}{x_{j,m-1}}
\nc{\yim}{x_{i,m}}
\nc{\yip}{x_{i,p}}
\nc{\yipp}{x_{i,p+1}}
\nc{\yipm}{x_{i,p-1}}
\nc{\yjp}{x_{j,p}}
\nc{\yjpp}{x_{j,p+1}}
\nc{\yjpm}{x_{j,p-1}}
\nc{\yimp}{x_{i,m+1}}
\nc{\yimm}{x_{i,m-1}}
\nc{\ykm}{x_{k,m}}
\nc{\ynm}{x_{n,m}}
\nc{\ynmp}{x_{n,m+1}}
\nc{\ynmpp}{x_{n,m+2}}
\nc{\ynnm}{x_{n-1,m}}
\nc{\ynnmp}{x_{n-1,m+1}}
\nc{\ynnmpp}{x_{n-1,m+2}}
\nc{\ykmp}{x_{k,m+1}}
\nc{\ykmm}{x_{k,m-1}}
\nc{\ykp}{x_{k,p}}
\nc{\ykpp}{x_{k,p+1}}
\nc{\ykpm}{x_{k,p-1}}

\nc{\seq}[1]{ \boldsymbol{(} {#1} \boldsymbol{)}   }

\usepackage[colorinlistoftodos]{todonotes}

\newcommand{\Sejin}[1]{\todo[size=\tiny,inline,color=green!30]{#1    \\ \hfill --- Se-jin}}

\title[Braid group action on quantum virtual Grothendieck rings]{Braid group action on quantum virtual Grothendieck ring through constructing presentations}

\author[I.-S. Jang]{Il-Seung Jang$^\ddagger$}
\address[I.-S. Jang]{$^\ddagger$Department of Mathematics, Incheon National University, Incheon 22012, Korea}
\email{ilseungjang@inu.ac.kr}
\urladdr{https://sites.google.com/view/isjang/home/}
\thanks{$^\ddagger$ I.-S. Jang was supported by Incheon National University Research Grant in 2023.}

\author[K.-H. Lee]{Kyu-Hwan Lee$^{\star}$}
\thanks{$^{\star}$ K.-H. Lee was partially supported by a grant from the Simons Foundation (\#712100).}
\address[K.-H. Lee]{$^{\star}$Department of Mathematics, University of Connecticut, Storrs, CT 06269, U.S.A.}
\email{khlee@math.uconn.edu}
\urladdr{https://www.math.uconn.edu/~khlee/}

\author[S.-j.~Oh]{Se-jin Oh$^{\dagger}$}
\address[S.-j.~Oh]{$^{\dagger}$Department of Mathematics, Sungkyunkwan University, Suwon, South Korea}
\email{sejin092@gmail.com}
\urladdr{https://sites.google.com/site/mathsejinoh/}
\thanks{$^{\dagger}$ S.-j.\ Oh was supported by the Ministry of Education of the Republic of Korea and the National Research Foundation of Korea (NRF-2022R1A2C1004045).}

\date{\today}

\begin{document}

\begin{abstract}
As a continuation of \cite{JLO1}, we investigate the quantum virtual Grothendieck ring $\frakK_q(\g)$
associated with a finite dimensional simple Lie algebra $\g$, especially of non-simply-laced type. We establish an isomorphism $\Uppsi_Q$ between the heart subring  $\frakK_{q,Q}(\g)$
of $\frakK_q(\g)$ associated with a Dynkin quiver $Q$ of type $\g$ and the unipotent quantum coordinate algebra $\calA_q(\n)$ of type $\g$. This isomorphism and the categorification theory via quiver Hecke algebras  enable us to obtain a presentation of $\frakK_q(\g)$, which reveals that $\frakK_q(\g)$ can be understood as a boson-extension of $\calA_q(\n)$. Then we show that the automorphisms, arising from  the reflections on Dynkin quivers and the isomorphisms $\Uppsi_Q$,  preserve the canonical basis $\sfL_q$ of $\frakK_q(\g)$. Finally, we prove that such automorphisms produce a braid group $B_\g$ action on $\frakK_q(\g)$.
\end{abstract}

\setcounter{tocdepth}{1}

\maketitle
\tableofcontents

\section{Introduction}

The braid group action constructed by Lusztig \cite{L90} (see also \cite{Saito94, Kimura12}) provides essential symmetries for important constructions in the theory of quantum groups.
As we continue our study of the quantum virtual Grothendieck ring $\frakK_q(\g)$ which started in \cite{KO22, JLO1}, it becomes much desirable to establish a braid group action on $\frakK_q(\g)$ when we consider the relationships of $\frakK_q(\g)$ with the unipotent quantum coordinate algebra and representation theory of quiver Hecke algebras (see~\eqref{eq:important diagram} for more detail).
In this paper, we construct such a braid group action by means of presentations of $\frakK_q(\g)$ coming from  the categorification theory via quiver Hecke algebras and various isomorphisms  between the heart subrings of $\frakK_q(\g)$ and the unipotent quantum coordinate algebra $\calA_q(\n)$. In what follows, we explain the background and motivation of $\frakK_q(\g)$, summarize relevant results from \cite{JLO1} and present the main results of this paper.
\smallskip

Let $q$ be an indeterminate. For a complex finite-dimensional simple Lie algebra $\g$, the category $\scrC_\g$ of finite-dimensional modules over the quantum loop algebra $\calU_q(\calL \g)$ has been intensively
studied due to its rich structure. For instance, it is non-braided as a monoidal category, non-semisimple as an abelian category, and has rigidity.
In \cite{CP95,CP95A}, Chari and Pressley gave a complete classification of simple modules in $\scrC_\g$ in terms of Drinfeld polynomials.  With the motivation of quantizing $\calW$-algebras, Frenkel and Reshetikhin (\cite{FR99}) introduced the $q$-character for $\scrC_\g$, which can be thought as a quantum loop analogue of the natural characters for finite-dimensional $\g$-modules.
The $q$-character theory tells us that  the Grothendieck ring $K(\scrC_\g)$
can be understood as a \emph{commutative} polynomial ring generated by the $q$-characters of \emph{fundamental modules} $L(Y_{i,p})$'s $((i,p) \in I \times \bfk)$, where $I$ is the index set for simple roots of $\g$ and $\bfk =\overline{\Q(q)} \subset \bigsqcup_{m >0} \C( \hspace{-.3ex} ( q^{1/m}) \hspace{-.3ex} )$ is the algebraic closure of $\Q(q)$. Since every simple module $L$ in $\scrC_\g$ can be obtained as the head of a certain ordered tensor product $M$ of fundamental modules $L(Y_{i,p})$ in a unique way \cite{AK97,Kas02,VV02}, we can label $L$ as $L(m)$
and $M$ as $M(m)$, where $m$ is a product of $Y_{i,p}$'s.
Note that the $q$-character of $L(Y_{i,p})$ can be computed by a combinatorial algorithm, called {\rm Frenkel--Mukhin (FM)} algorithm \cite{FM01}.

Let $t$ be another indeterminate. The quantum Grothendieck ring $\calK_t(\scrC_\g)$ of $\scrC_\g$
is a non-commutative $t$-deformation of $K(\scrC_\g)$ in a quantum torus $\calY_t(\g)$. It was introduced in \cite{Nak04,VV03} for simply-laced types from a geometric motivation, and then in \cite{Her04} for all types in a uniform way using the inverse of
\emph{quantum Cartan matrix} and the $t$-deformed screening operators \cite{Her03}. The ring $\calK_t(\scrC_\g)$ has a \emph{canonical basis} $\bfL_t$ consisting of $(q,t)$-characters
corresponding to simple modules, denoted by $[L]_t$ for a simple module $L \in \scrC_\g$,  which has the following (conjectural) properties:
\begin{eqnarray} &&
\parbox{91ex}{
\bnum
\item The structure coefficients of $\bfL_t$ are positive.
\item $[L]_t$ has positive coefficients.
\item $[L]:=[L]_t |_{t=1}$ recovers the $q$-character of $L$.
\ee
}\label{eq: positivity properties}
\end{eqnarray}

Furthermore, the quantum Grothendieck ring admits a quantum loop analogue of Kazhdan--Lusztig theory with respect to $\bfL_t$.
More precisely, Nakajima \cite{Nak04}
established a remarkable  Kazhdan--Lusztig type algorithm for simply-laced $\g$ using the geometry of quiver varieties to compute the Jordan--H\"older multiplicity $P_{m,m'}$ of the simple module $L(m')$ occurring in $M(m)$. In particular, as a quantum loop analogue of Kazhdan--Lusztig polynomials, the polynomials
$P_{m,m'}(t)$ are constructed in the quantum Grothendieck ring, and it is  proved in \cite{Nak04} that $P_{m,m'}(1)=P_{m,m'}$. Since we have
$$
[M(m)]= [L(m)] + \sum_{m' \in \calM_+; \ m' \lN m} P_{m,m'} [L(m')]  \quad \text{ in $K(\scrC_\g)$},
$$
this result provides a way of computing all the simple $q$-characters (see~\eqref{eq: Nakajima order} for the definition of the partial order $\lN$).
Also, Nakajima \cite{Nak04} and Varagnolo--Vasserot \cite{VV03} proved that the positivity  in~\eqref{eq: positivity properties} holds for simply-laced types.
For a generalization of these results to non-simply-laced types,
Hernandez \cite{Her04} formulates $\calK_t(\scrC_\g)$ as the intersection of the  kernels of the $t$-quantized screening operators in \cite{Her03} (cf.~\cite{FR99,FM01}), and then established a (conjectural) Kazhdan--Lusztig type algorithm.
Then it was conjectured in \cite{Her04} that the natural properties in~\eqref{eq: positivity properties} also hold for all types.

Recently, Fujita--Hernandez--Oya and the third named author proved a large part of the conjectures for non-simply-laced types by showing the
coincidence of presentations of $\calK_t(\scrC_\bfg)$ and $\calK_t(\scrC_g)$ for a pair $(\bfg, g)$, where the types of $\bfg$ and $g$ are given by
\begin{align}\label{eq: g and bfg}
\text{$(g,\bfg)=(B_n,A_{2n-1})$, $(C_n,D_{n+1})$, $(F_4,E_{6})$  and $(G_2,D_{4})$.}
\end{align}

Let us explain this more precisely.
In $\scrC_\g$, there exists a skeleton subcategory $\scrC_\g^0$ in the following sense:
Any \emph{prime} simple module in $\scrC_\g$  can be obtained from a simple module in $\scrC_\g^0$ by twisting module actions via a $\bfk$-algebra automorphism of $\calU_q(\calL \g)$.
Thus, from the representation theoretic viewpoint, it is enough to consider $\scrC_\g^0$.  Compared with $\calK_t(\scrC_\g)$,
the quantum Grothendieck ring  $\calK_t(\scrC^0_\g)$ of $\scrC^0_\g$ is contained in the quantum torus $\calY^0_t(\g)$ generated by countably many generators, while $\calY_t(\g)$
is not. In \cite{FHOO}, it is proved that
\begin{align}\label{eq: coincidence}
\calK_t(\scrC^0_g) \simeq  \calK_t(\scrC^0_\bfg) \quad \text{as $\Z[t^{\pm \frac{1}{2}}]$-algebras},
\end{align}
and the presentation of $\bbK_t(\scrC^0_\bfg) \seteq \Q(t^{1/2}) \otimes_{\Z[t^{\pm1/2}]} \calK_t(\scrC^0_\bfg)$ is known in \cite{HL15}.
Hence, we conclude that $\bbK_t(\scrC^0_g)$ is a $\Q(t^{1/2})$-algebra generated by $\{ \sfx_{i,m} \ | \ i \in I_\bfg, \ m \in \Z \}$ subject to the relations
\begin{equation} \label{eq:presentation}
\begin{split}
&\sum_{s=0}^{1-\bfa_{i,j}} (-1)^s \left[ \begin{matrix} 1-\bfa_{i,j} \\ s  \end{matrix}\right]_{t} \sfx_{i,k}^{1-\bfa_{i,j}-s}  \sfx_{j,k} \sfx_{i,k}^s =0,  \\
&\sfx_{i,k}\sfx_{j,k+1}  = t^{-(\upal_i,\upal_j)} \sfx_{j,k+1}\sfx_{i,k} + (1-t^{-(\upal_i,\upal_i)})\delta_{i,j},  \\
&\sfx_{i,k}\sfx_{j,l} = t^{(-1)^{k+l}(\upal_i,\upal_j)}\sfx_{j,l}\sfx_{i,k}  \quad \text{ for $l>k+1$},
\end{split}
\end{equation}
where $(\bfa_{i,j})_{i,j \in I\bfg}$ is the Cartan matrix of $\bfg$ and $\upal_i$'s are the simple roots of $\bfg$. To obtain this presentation, we need to consider the notion of $\rmQ$-datum
introduced in \cite{FO21}, which can be understood as a generalization of a Dynkin quiver of simply-laced type.
A $\rmQ$-datum $Q=(\Dynkin,\sigma,\xi)$ for $\g$ is a triple consisting of (i)
the Dynkin diagram $\Dynkin(=\Dynkin_\bfg)$ of simply-laced type $\bfg$ (need not to be the same as $\Dynkin_\g$), (ii) the Dynkin diagram automorphism $\sigma$ on $\Dynkin$ yielding $\Dynkin_\g$ as the $\sigma$-orbits of $\Dynkin_\bfg$, and (iii) a height function $\xi: I_\bfg \to \Z$ satisfying certain conditions (\cite[Definition 3.5]{FO21}). For each $\rmQ$-datum $Q$
of $\g$, we can associate a heart monoidal subcategory $\scrC_\g^Q$ of $\scrC_\g^0$.

For any pair of $\rmQ$-data $Q_1$ of $\g_1$ and $Q_2$ of $\g_2$ whose Dynkin diagrams are the \emph{same as $\Dynkin_\bfg$},
it is proved in \cite[Theorem 8.6]{FHOO} (cf.~\cite{HL15, HO19}) that there exists an isomorphism of $\Z[q^{\pm 1/2}]$-algebras
\begin{align*}
 \calK_t\left(\scrC_{\g_k}^{Q_k} \right) \, \longrightarrow \, \calA_{\Z[q^{\pm 1/2}]}(\bfn) \qquad (k = 1, 2)
\end{align*}
preserving bar involutions on both sides,
where $\calA_{\Z[q^{\pm 1/2}]}(\bfn)$ denotes the integral form of the unipotent quantum coordinate algebra of $\bfg$.
Hence we have
\begin{align}   \label{eq: heart conincidence}
 \calK_t\left(\scrC_{\g_1}^{Q_1} \right)  \simeq  \calK_t\left(\scrC_{\g_2}^{Q_2} \right)
\end{align}
(see \cite[Corollary 10.5, Theorem 10.6]{FHOO} for more details).
Then using \eqref{eq: heart conincidence}, the rigidity of $\scrC_\g^0$, and singularities of $R$-matrices between fundamental modules,  they proved~\eqref{eq: coincidence} and established a large part of the conjectures for non-simply-laced $g$, which is referred to as the \emph{propagation of positivity phenomena} \cite{FHOO,FHOO2}.
At this point, for non-simply-laced $g$ in \eqref{eq: g and bfg}, we emphasize that the presentation \eqref{eq:presentation} of $\calK_t(\scrC^0_g)$ depends on the root system of $\bfg$, and the quantum cluster algebra structure \cite{BZ05} of $\calK_t(\scrC^0_g)$ is still of skew-symmetric type \cite{HL16,B21,KKOP2,FHOO2}.
\smallskip

To sum up, the representation theory of quantum loop algebras seems to have \emph{simply-laced nature},
even though $\g$ is of non-simply-laced.
Motivated by this observation, Kashiwara and the third named author introduced in \cite{KO22}
a new quantum torus $\calX_q(\g)$ and the quantum virtual
Grothendieck ring  $\frakK_q(\g)$ in $\calX_q(\g)$ for every $\g$. Note that, for simply-laced type $\g$, $\calX_q(\g)$ and $\frakK_q(\g)$ are isomorphic to $\calY^0_t(\g)$ and $\calK_t(\scrC^0_\g)$, respectively, but not for non-simply-laced types.
For the construction of $\frakK_q(\g)$ and $\calX_q(\g)$, they employed the inverses of \emph{$t$-quantized Cartan matrices} and \emph{Dynkin quivers of every finite type}, which seem to have not been well-investigated before. Then it is proved in \cite{KO22} that the $q$-commutativity of $\calX_q(\g)$ is controlled by the root system of $\g$, while $\calY^0_t(g)$ is controlled by that of $\bfg$ in~\eqref{eq: g and bfg} \cite{FO21}.
Thus the nature of  $\frakK_q(\g)$ \emph{does genuinely depend} on the type of $\g$.
Here we remark that (1) the $(q,t)$-Cartan matrix, introduced by Frenkel--Reshetikhin in~\cite{FR98},
is a two-parameter deformation of the Cartan matrix of $\g$, whose specialization at $t=1$
is the quantum Cartan matrix, and at $q=1$ is the $t$-quantized Cartan matrix, (2) $\frakK_q(\g)$ is a $q$-deformation of the commutative ring $\scrK^-_\ttt(\g)$, which is the specialization of
the refined ring $\overline{\scrK}_{\ttq,\ttt,\upal}(\g)$ of interpolating $(\ttq,\ttt)$-characters in \cite{FHR21} at $\ttq=1$ and $\upal=d$,  where $\upal$ is a factor to interpolate several characters (see \cite[Remark 6.2(1)]{FHR21}) and $d$ is the lacing number of $\g$.

In \cite{JLO1}, the authors of this paper initiated a study of the quantum virtual Grothendieck ring $\frakK_q(\g)$, especially for non-simply-laced types. The results in \cite{JLO1} are summarized as follows:
\bnA
\item We construct a basis $\sfF_q = \{F_q(m) \ | \ m \in \calM_+ \}$ of $\frakK_q(\g)$ in an algorithmic way, whose elements are parameterized by the set $\calM_+$ of dominant monomials in $\calX_q(\g)$. Using $\sfF_q$, we construct two other bases
$\sfE_q= \{E_q(m) \ | \ m \in \calM_+ \}$ and $\sfL_q = \{L_q(m) \ | \ m \in \calM_+ \}$ and show that the transition map between them exhibits the paradigm of Kazhdan--Lusztig theory:
$$
E_q(m) = L_q(m) + \sum_{m' \in \calM_+; \ m' \lN m} P_{m,m'}(q) L_q(m') \quad \text{ for some $P_{m,m'}(q) \in q\Z[q]$}.
$$
We call $\sfL_q$ the \emph{canonical basis} and $\sfE_q$  the \emph{standard basis} of $\frakK_q(\g)$, respectively.
\item  $\frakK_q(\g)$ has a quantum cluster algebra structure of \emph{skew-symmetrizable type} and  the quantum cluster of an initial quantum seed $\scrS$ is a $q$-commuting subset of $\sfF_q$:
\begin{align*}
\frakK_q(\g) \simeq
\scrA_{\Z[q^{\pm1/2}]}(\seed),
\end{align*}
where $\scrA_{\Z[q^{\pm1/2}]}(\seed)$ denotes the quantum cluster algebra with an initial quantum seed $\scrS$.
To prove this, we develop the \emph{quantum folded $T$-systems} among the \emph{Kirillov--Reshetikhin (KR) polynomials} $F_q(m)$ (see Section~\ref{subsec: T-sys KR poly}) and employ them as quantum exchange relations of $\scrA_{\Z[q^{\pm1/2}]}(\seed)$. Then we conclude that every KR-polynomial $F_q(m)$ appears as a  quantum cluster variable.
\ee
We also conjecture the following, which is in accordance with one of the ultimate goals of  quantum cluster algebra theory that every quantum  cluster monomial appears as a dual-canonical/upper-global basis element (see \cite[Introduction]{FZ02}):
\begin{Conj}[{\cite[Conjecture I]{JLO1}}]  \label{Conj: I}
For each KR-polynomial $F_q(m)$, we have
$$
F_q(m) = L_q(m).
$$
\end{Conj}

In this paper, we first focus on the heart subring $\frakK_{q,Q}(\g)$ of $\frakK_{q}(\g)$ associated with a Dynkin quiver $Q$ of $\g$.
Let $\tbfB^\up$ be the normalized dual-canonical/upper-global basis of $\calA_q(\n)$ by Lusztig \cite{L90} and Kashiwara \cite{K91,K93} (see also \cite{Kimura12}).
We analyze
\begin{enumerate}[(a)]
	\item the T-system among unipotent quantum minors in the unipotent quantum coordinate algebra $\calA_q(\n)$ of $\g$,
	\item the quantum folded T-system among KR-polynomials $F_q(m)$,
	\item the quantum cluster algebra structure of $\calA_q(\n)$ as a subalgebra of the quantum torus $\calX_q(\g)$,
\end{enumerate}
and obtain the following theorem:
\begin{thmx} \label{thmx: A} {\rm (Theorem~\ref{thm:main1})}
For any Dynkin quiver $Q$ of $\g$, we have an algebra isomorphism
\begin{align}\label{eq: PsiQ}
 \Psi_Q : \calA_{\Z[q^{\pm 1/2}]}(\n) \isoto \frakK_{q,Q}(\g),
\end{align}
which sends the normalized dual-canonical/upper-global basis $\tbfB^\up$ of $\calA_{\Z[q^{\pm 1/2}]}(\n)$ to the canonical basis $\sfL_{q,Q} \seteq \sfL_q \cap \frakK_{q,Q}(\g)$.
\end{thmx}

Theorem~\ref{thmx: A} gives an affirmative answer for the expectation in \cite[Remark 6.16]{KO22}, which may be regarded as a new realization of $\calA_{\Z[q^{\pm 1/2}]}(\n)$ in terms of non-commutative Laurent polynomials for non-simply-laced types.
In the proof of Theorem~\ref{thmx: A}, we also conclude that Conjecture~\ref{Conj: I} holds when $m$ is contained in a quantum subtorus $\calX_{q,Q}(\g)$
for some Dynkin quiver $Q$:
\begin{corx} {\rm (Corollary~\ref{cor: partial proof}:  Partial proof of Conjecture~\ref{Conj: I})}
Let  $m$ be a KR-monomial  contained in $\calX_{q,Q}(\g)$ for some Dynkin quiver $Q$. Then we have
$$
F_q(m) = L_q(m).
$$
\end{corx}

As another byproduct of Theorem \ref{thmx: A} combined with a recent work of Qin \cite{Qin20} (cf.~\cite{McNa21}), we have the following result:
\begin{corx} {\rm (Corollary \ref{thm: cluster monomials are in canonical basis})} \label{corx:cluster monomials are in canonical basis}
Every cluster monomial of $\frakK_{q,Q}(\g)$ is contained in the canonical basis $\sfL_{q,Q}$.
\end{corx}
\noindent
Since $\frakK_{q,Q}(\sfg)$ recovers the quantum Grothendieck ring of the heart subcategory $\scrC_{\sfg}^Q$ when $\sfg$ is of simply-laced finite type, the above result readily follows in that case from \cite{HL15} (see also \cite{HO19,FHOO})
with \cite{KKKO18,Qin17}, so Corollary \ref{corx:cluster monomials are in canonical basis} is an extension of previous results to skew-symmetrizable type cluster algebras.

On the other hand, Theorem~\ref{thmx: A} tells us that $\frakK_{q,Q}(\g)$ can be understood as a kind of \emph{character ring} of $R^\g\gmod$ from the viewpoint of categorification, where $R^\g\gmod$ is the category of finite-dimensional graded modules over the quiver Hecke algebra $R^\g$.
In \cite{KO22A}, Kashiwara and the third named author calculated the invariants of $R$-matrices between the cuspidal modules over $R^\g$, where the
cuspidal modules categorify the dual PBW-vectors associated with Dynkin quivers. Based on the results and Theorem~\ref{thmx: A}, we obtain the presentation of $\bbK_{q}(\g)\seteq \Q(q^{1/2})\otimes_{\Z[q^{\pm 1/2}]} \frakK_q(\g)$ as follows:

\begin{thmx} \label{thmx: B} {\rm (Theorem~\ref{thm: presentation})}
The algebra $\bbK_{q}(\g) $ is isomorphic to the $\Q(q^{1/2})$-algebra generated by
$\{ \sfy_{i.m} \ |  \ i \in I_\g, \ m \in \Z \}$ subject to the following relations:
\begin{equation} \label{eq:presen}
\begin{aligned}
&\sum_{s=0}^{1-\sfc_{i,j}} (-1)^s \left[ \begin{matrix} 1-\sfc_{i,j} \\ s  \end{matrix}\right]_{q_i} \sfy_{i,k}^{1-\sfc_{i,j}-s}  \sfy_{j,k} \sfy_{i,k}^s =0,   \\
&\sfy_{i,k}\sfy_{j,k+1}  = q^{-(\al_i,\al_j)} \sfy_{j,k+1}\sfy_{i,k} + (1-q^{-(\al_i,\al_i)})\delta_{i,j}, \\
&\sfy_{i,k}\sfy_{j,l} = q^{(-1)^{k+l}(\al_i,\al_j)}\sfy_{j,l}\sfy_{i,k}  \quad \text{ for $l>k+1$},
\end{aligned}
\end{equation}
where $\{ \al_i \ | \  i \in I_\g\}$ denotes the set of simple roots of $\g$,
$q_i \seteq q^{(\al_i,\al_i)/2}$ and $(\sfc_{i,j})_{i,j \in I_\g}$ denotes the Cartan matrix of $\g$.  We denote by
$\hA_{q}(\n)$ the $\Q(q^{1/2})$-algebra with presentation in~\eqref{eq:presen}.
\end{thmx}

Note that we can obtain a presentation with respect to any Dynkin quiver $Q$ of $\g$. Then using Dynkin quivers $Q$ and $Q'$, we can induce an automorphism $\tUptheta(Q',Q)$ on $\frakK_{q}(\g)$ and hence on $\hA_{q}(\n)$:

\begin{thmx} \label{thmx: C} {\rm (Theorem~\ref{thm: indeed frakKq})} For an arbitrary pair of Dynkin quivers $Q$ and $Q'$ of $\g$, we have an automorphism $\tUptheta(Q',Q)$ on $\frakK_{q}(\g)$, which makes the diagram below commutative and preserves
the canonical basis $\sfL_q$.
$$
\xymatrix@R=1ex@C=12ex{    & \frakK_{q,Q}(\g)  \ar@{^{(}->}[r]& \frakK_{q}(\g)  \ar[dd]^{\tUptheta(Q',Q)} \\
\calA_{\Z[q^{\pm 1/2}]}(\n)  \ar[ur]^{\Psi_Q}_\simeq\ar[dr]_{\Psi_{Q'}}^\simeq\\
& \frakK_{q,Q'}(\g)  \ar@{_{(}->}[r] & \frakK_{q}(\g)
}
$$
\end{thmx}

Let us choose a pair of Dynkin quivers, consisting of a Dynkin quiver $Q$ with a source at $i$ and its reflected Dynkin quiver $s_iQ$,
and set $\TT_i \seteq \tUptheta(s_iQ,Q)$. Concentrating on such a pair, we prove that
$\hA_{q}(\n)$ is a module over the Braid group $B_\g$:

\begin{thmx} \label{thmx: D} {\rm (Theorem~\ref{thm: braid main})} The automorphisms $\TT^{\pm}_i$ on $\hA_{q}(\n)$ can be described as follows:
\begin{align} \label{eq: br1}
\TT_{i}(\sfy_{j,m}) =
  \bc
\sfy_{j,m+\delta_{i,j}}  &  \text{if} \  \sfc_{i,j} \ge 0, \\[1ex]
\dfrac{ \left( \displaystyle \sum_{k=0}^{- \sfc_{i,j}} (-1)^{k} q_i^{ - \sfc_{i,j}/2 -k  } \sfy_{i,m}^{(k)}\sfy_{j,m}  \sfy_{i,m}^{(- \sfc_{i,j}-k)}  \right)}{ (q_i -q_i^{-1})^{- \sfc_{i,j}}  }  &\text{if} \  \sfc_{i,j} < 0,  \ec
 \end{align}
and
\begin{align}\label{eq: br2}
\TT_{i}^{-1}(\sfy_{j,m}) =
  \bc
\sfy_{j,m-\delta_{i,j}}  &  \text{if} \  \sfc_{i,j} \ge 0, \\[1ex]
\dfrac{ \left( \displaystyle \sum_{k=0}^{-\sfc_{i,j}} (-1)^{k} q_i^{ -\sfc_{i,j}/2 -k  } \sfy_{i,m}^{(-\sfc_{i,j}-k)} \sfy_{j,m}  \sfy_{i,m}^{(k)}  \right)}{ (q_i -q_i^{-1})^{-\sfc_{i,j} }  }  &\text{if} \  \sfc_{i,j} < 0,
\ec
 \end{align}
where
$$
\sfy_{i,m}^{(k)}  \seteq \dfrac{\sfy_{i,m}^{k} }{[k]_{q_i}!}.
$$
Furthermore,
$$\text{the set of automorphisms $\{ \TT^{\pm}_i \ | \ i \in I_\g \}$ induces a Braid group $B_\g$ action on $\hA_{q}(\n)$.}$$.
\end{thmx}
\vskip -2em
\noindent
Note that
\bna
\item Theorem~\ref{thmx: D} can be understood as a generalization the $B_\g$-action on the quantum group $U_q(\g)$ of Lusztig and
Saito~\cite{L902,Saito94} to $\hA_{q}(\n)$,
\item Theorem~\ref{thmx: D} and its proof for simply-laced $\g$ are announced in~\cite{KKOP21A}.
\ee

Let us explain the argument of the proof for Theorem~\ref{thmx: D}.
First of all, we obtain the formulae from Theorem~\ref{thmx: B}  and Theorem \ref{thmx: C} by comparing the presentations corresponding to $Q$ and $s_iQ$.
An attempt to prove Theorem~\ref{thmx: D} by directly applying the formulae in~\eqref{eq: br1} and~\eqref{eq: br2} brings out very complicated computations, which does not seem feasible by hand for non-simply-laced types, because we need to check that (i) $\{ \TT_{i}(\sfy_{j,m}) \}_{j \in I_\g}$ satisfy the relations in~\eqref{eq:presen} and (ii) $\{ \TT_i \}_{i \in I_\g}$ satisfy the braid relations, such as  $\TT_i\TT_j\TT_i(\sfy_{k,m})=\TT_j\TT_i\TT_j(\sfy_{k,m})$ and $\TT_i\TT_j\TT_i\TT_j(\sfy_{k,m})=\TT_j\TT_i\TT_j\TT_i(\sfy_{k,m})$ for all $(k,m) \in I_\g \times \Z$, as operators on  $\hA_{q}(\n)$. Instead, for each Dynkin quiver $Q$,
we understand $\sfy_{j,0}$ as an image of dual PBW vector of weight $\al_j$ in $\frakK_q(\g)$ through the isomorphism $\tUptheta_Q$
between $\bbK_q(\g)$ and $\hA_{q}(\n)$ (Theorem~\ref{thm: presentation}), and $\sfy_{j,m}$ as the \emph{$m$-th dual} of the image. Then using the isomorphisms $\tUptheta_Q$ and $\tUptheta_{s_iQ}$, the combinatorics of AR-quivers $\Gamma^Q$ and $\Gamma^{s_iQ}$, and relations among dual PBW vectors in $\calA_q(\n)$, we prove
Theorem~\ref{thmx: D} without complicated direct computations.

\medskip

This paper is organized as follows. In Section~\ref{sec: QVGR}, we review the definition  of the quantum virtual Grothendieck ring $\frakK_q(\g)$ and its properties. In particular, we emphasize the construction of $\calX_q(\g)$ and bases of $\frakK_q(\g)$. In Section~\ref{sec: Quivers and heart subring}, we first recall various quivers such as Dynkin quivers, AR-quivers,
and repetition quivers, which play important roles in the study of $\frakK_q(\g)$. Then we define subrings of $\frakK_q(\g)$, whose bases are compatible with the bases of  $\frakK_q(\g)$. We also investigate the multiplicative structure of $F_q(Y_{i,p})$ in terms of the notion of
$Q$-weight, which will be used in later sections. In Section~\ref{Sec: Aqn}, we review materials on the unipotent quantum coordinate algebra, especially their elements called the unipotent quantum minors,
the dual-canonical/upper-global basis, and the dual PBW basis. In Section~\ref{Sec: Iso}, we prove  Theorem~\ref{thmx: A} by applying the framework of \cite{HL15,HO19} due to
Hernadez--Leclerc--Oya. In Section~\ref{Sec: quiver Hecke}, we review the categorification theory of $\calA_q(\n)$ via quiver Hecke algebras
and investigate the interrelationship between the categorification and the isomorphism $\Psi_Q$ in Theorem~\ref{thmx: A}. In Section~\ref{Sec: Presentation}, based on the results in the previous sections, we prove Theorem~\ref{thmx: B}
(see also Proposition~\ref{prop: dual zero} and Corollary~\ref{cor: dual commute}). Then we prove Theorem~\ref{thmx: C} by adopting the methods in \cite{HL15,FHOO}. In Section~\ref{Sec: Braid}, we prove Theorem~\ref{thmx: D}, which is the culminating result of this paper. To prove this theorem, we use the equations between dual PBW vectors in $\calA_q(\n)$ and $F_q(Y_{i,p})$'s
in $\frakK_q(\g)$, circumventing direct calculations involving the expressions in Theorem~\ref{thmx: D}. Still, a large part of this paper is inevitably devoted to explicit computations for the proof of Theorem~\ref{thmx: D}.

\subsection*{Convention} Throughout this paper, we use the following convention.
\begin{enumerate}[\,\,$\bullet$]
\item For a statement $\mathtt{P}$, we set $\delta(\mathtt{P})$ to be $1$ or $0$
depending on whether $\mathtt{P}$ is true or not.
\item For $k,l \in \Z$ and $s \in \Z_{\ge1}$, we write $k \equiv_s l$ if $s$ divides $k-l$ and $k \not\equiv_s l$, otherwise.
\item For a totally ordered set $J = \{ \cdots < j_{-1} < j_{0} < j_{1} < j_{2} \cdots \}$, write
$$\prod_{j \in J}^\to A_j \seteq \cdots A_{j_{-1}}A_{j_{0}}A_{j_{1}}A_{j_{2}} \cdots.$$
\end{enumerate}

\section{Quantum virtual Grothendieck rings} \label{sec: QVGR}
In this section, we briefly review the definition of quantum virtual Grothendieck ring $\frakK_q(\g)$ by following \cite{JLO1}. We emphasize here that
$\frakK_q(\g)$ is isomorphic to the quantum Grothendieck ring of quantum loop algebra $\calU_q(\calL \g)$ when $\g$ is of simply-laced finite type.

\subsection{Cartan datum} Let $\g$ be a complex finite-dimensional simple Lie algebra of rank $n$.  We denote by $I= \{1,2,\ldots,n\}$ the index set, by  $\sfC=(\sfc_{i,j})_{1 \le i,j \le n}$
the Cartan matrix of $\g$, by $\Pi=\{ \al_i \}_{i \in I}$ the set of simple roots, by $\Pi^\vee=\{ h_i \}_{i \in I}$ the set of simple coroots and by $\wl$ the weight lattice of $\g$.

Note that there exists a diagonal matrix $\sfD={\rm diag}(d_i \in \Z_{\ge 1} \ | \ i \in I)$ such that
$$
\osfB = \sfD \sfC   \text{ and } \usfB = \sfC \sfD^{-1}  \text{ are  symmetric}.
$$
We take such $\sfD$ satisfying $\min(d_i \ | \ i \in I)=1$ throughout this paper. We have the symmetric bilinear form $( \cdot, \cdot)$ on $\wl$ such that
$$
( \al_i,\al_j ) = d_i \sfc_{i,j}=d_j\sfc_{j,i}  \quad \text{ and } \quad \lan h_i,\al_j \ran  = \dfrac{2 (\al_i,\al_j)}{(\al_i,\al_i)}.
$$
For $i,j \in I$, we set
\begin{align} \label{eq:def of hij}
h_{i,j} \seteq
\bc
2 & \text{ if } \sfc_{i,j}\sfc_{j,i}=0, \\
3 & \text{ if } \sfc_{i,j}\sfc_{j,i}=1, \\
4 & \text{ if } \sfc_{i,j}\sfc_{j,i}=2, \\
6 & \text{ if } \sfc_{i,j}\sfc_{j,i}=3.
\ec
\end{align}

We denote by $\Phi_{\pm}$ the \emph{set of positive $($resp. negative$)$ roots of $\g$.}
For $\al \in \Phi_+$, we set $d_{\al} \seteq (\al,\al)/2 \in \Z_{\ge 1}$.
For each $i \in I$, we write $\varpi_i$ the \emph{fundamental weight} of $\wl$ given by $\lan h_i,\varpi_i \ran =\delta_{i,j}$ and set $\rho \seteq \sum_{i\in I}\varpi_i$.
The free abelian group $\rl = \bigoplus_{i \in I}\Z \al_i$ is called the \emph{root lattice}.
We set $\rl^\pm  \seteq \pm\sum_{i \in I} \Z_{\ge 0} \al_i \subset \wl^\pm \seteq \pm\sum_{i \in I} \Z_{\ge 0} \varpi_i$. For $\mu,\zeta \in \wl$, we write $\mu \preceq \zeta$ if $\mu-\zeta \in \rl^+$.
For $\be = \sum_{i \in I} k_i\al_i \in \rl^+$, we set $\het(\be)\seteq\sum_{i \in I} k_i$ and
${\rm supp}(\be) \seteq\{  i \in I \ | \  k_i \ne 0 \}$.
For $\be  =\sum_{i \in I}n_i\al_i  \in \Phi_+$, the \emph{multiplicity} $\mul(\be)$  is defined by
$\mul(\be) \seteq \max\{n_i \,\mid \, i \in I\}$.
For $\al,\be \in \Phi$, we set
$$
p_{\be,\al} \seteq \max\{p\in \Z \ | \ \be -p \al \in \Phi \}
$$
(see~\eqref{eq: palbe} below).

Slightly modifying the usual notation of \emph{Dynkin diagram}, we use the following convention of Dynkin diagram $\Dynkin$,
consisting of the set of vertices $\Dynkin_0 =I$ and the set of edges $\Dynkin_1$,  as follows:
\begin{equation}\label{eq: finite}
\begin{aligned}
&A_n  \ \   \xymatrix@R=0.5ex@C=3.5ex{ *{\circled{2}}<3pt> \ar@{-}[r]_<{ 1 \ \  } & *{\circled{2}}<3pt> \ar@{-}[r]_<{ 2 \ \  } & *{\circled{2}}<3pt> \ar@{.}[r]_<{ 3  \ \  }
&*{\circled{2}}<3pt> \ar@{-}[r]_>{ \ \ n} &*{\circled{2}}<3pt> \ar@{}[l]^>{   n-1} }, \quad
 B_n  \ \   \xymatrix@R=0.5ex@C=3.5ex{ *{\bcircled{4}}<3pt> \ar@{-}[r]_<{ 1 \ \  } & *{\bcircled{4}}<3pt> \ar@{-}[r]_<{ 2 \ \  } & *{\bcircled{4}}<3pt> \ar@{.}[r]_<{ 3  \ \  }
&*{\bcircled{4}}<3pt> \ar@{-}[r]_>{ \ \ n} &*{\circled{2}}<3pt> \ar@{}[l]^>{   n-1} }, \quad
C_n  \ \   \xymatrix@R=0.5ex@C=3.5ex{ *{\circled{2}}<3pt> \ar@{-}[r]_<{ 1 \ \  } & *{\circled{2}}<3pt> \ar@{-}[r]_<{ 2 \ \  } & *{\circled{2}}<3pt> \ar@{.}[r]_<{ 3  \ \  }
&*{\circled{2}}<3pt> \ar@{-}[r]_>{ \ \ n} &*{\bcircled{4}}<3pt> \ar@{}[l]^>{   n-1} }, \  \allowdisplaybreaks\\
&D_n  \ \  \raisebox{1em}{ \xymatrix@R=2ex@C=3.5ex{  &&& *{\circled{2}}<3pt> \ar@{-}[d]_<{ n-1  }  \\
*{\circled{2}}<3pt> \ar@{-}[r]_<{ 1 \ \  } & *{\circled{2}}<3pt> \ar@{-}[r]_<{ 2 \ \  } & *{\circled{2}}<3pt> \ar@{.}[r]_<{ 3  \ \  }
&*{\circled{2}}<3pt> \ar@{-}[r]_>{ \ \ n} &*{\circled{2}}<3pt> \ar@{}[l]^>{   n-2} }},  \
E_{6}  \ \  \raisebox{1em}{   \xymatrix@R=2ex@C=3.5ex{  &&  *{\circled{2}}<3pt> \ar@{-}[d]^<{ 2\ \ }   \\
*{\circled{2}}<3pt>  \ar@{-}[r]_<{ 1 \ \  } & *{\circled{2}}<3pt> \ar@{-}[r]_<{ 3 \ \  } & *{\circled{2}}<3pt> \ar@{-}[r]_<{ 4 \ \  }
&*{\circled{2}}<3pt> \ar@{-}[r]_>{ \ \ 6} &*{\circled{2}}<3pt> \ar@{}[l]^>{   5} } },  \
E_{7}  \ \  \raisebox{1em}{    \xymatrix@R=2ex@C=3.3ex{  &&  *{\circled{2}}<3pt> \ar@{-}[d]^<{ 2\ \ }   \\
 *{\circled{2}}<3pt>  \ar@{-}[r]_<{ 1 \ \  } & *{\circled{2}}<3pt>  \ar@{-}[r]_<{ 3 \ \  } & *{\circled{2}}<3pt> \ar@{-}[r]_<{ 4 \ \  } & *{\circled{2}}<3pt> \ar@{-}[r]_<{ 5 \ \  }
&*{\circled{2}}<3pt> \ar@{-}[r]_>{ \ \ 7} &*{\circled{2}}<3pt> \ar@{}[l]^>{   6} } },  \allowdisplaybreaks \\
&  E_{8}  \ \  \raisebox{1em}{    \xymatrix@R=2ex@C=3.5ex{ &&  *{\circled{2}}<3pt> \ar@{-}[d]^<{ 2\ \ }   \\
 *{\circled{2}}<3pt>  \ar@{-}[r]_<{ 1 \ \  } & *{\circled{2}}<3pt>  \ar@{-}[r]_<{ 3 \ \  } & *{\circled{2}}<3pt> \ar@{-}[r]_<{ 4 \ \  } & *{\circled{2}}<3pt> \ar@{-}[r]_<{ 5 \ \  }   & *{\circled{2}}<3pt> \ar@{-}[r]_<{ 6 \ \  }
&*{\circled{2}}<3pt> \ar@{-}[r]_>{ \ \ 8} &*{\circled{2}}<3pt> \ar@{}[l]^>{   7} } }, \quad
F_{4}   \ \   \xymatrix@R=0.5ex@C=3.5ex{    *{\bcircled{4}}<3pt> \ar@{-}[r]_<{ 1 \ \  } & *{\bcircled{4}}<3pt> \ar@{-}[r]_<{ 2 \ \  }
&*{\circled{2}}<3pt> \ar@{-}[r]_>{ \ \ 4  } &*{\circled{2}}<3pt> \ar@{}[l]^>{   3} }, \quad
 G_2 \ \   \xymatrix@R=0.5ex@C=3.5ex{  *{\circled{2}}<3pt> \ar@{-}[r]_<{ 1 \ \  }  & *{\rcircled{6}}<3pt> \ar@{-}[l]^<{ \ \ 2  } }.
\end{aligned}
\end{equation}
Here $ \xymatrix@R=0.5ex@C=3.5ex{    *{\circled{t}}<3pt>}_k$ means that $(\al_k,\al_k)=t$. We write $i \sim j$ for $i,j \in I$ if $\sfc_{i,j} < 0$. We denote by
$d(i,j)$ the number of edges between $i$ and $j$ in $\Dynkin$. For instance, $d(i,j)=1$ if $i \sim j$.

\subsection{Weyl group $\weyl$, reduced sequences and convex orders}
Let us set up notations and briefly summarize some results associated with the Weyl group of $\g$, which can be found in \cite{Bou, Hum}.
We denote by $\weyl$ the Weyl group  of $\g$ and $s_i$ the simple reflection of $\weyl$ for each $i \in I$. A \emph{Coxeter element} of $\weyl$ is a product of the
form $s_{i_1}s_{i_2} \cdots s_{i_n}$ such that $\{ i_k \ | \ 1 \le k \le n \} =I$. Note that all Coxeter elements are conjugate and the order of Coxeter elements  is known as \emph{Coxeter number $\sfh$}.
We denote by $w_\circ$ the longest element in $\weyl$. Recall that $w_\circ$ induces an involution $^*: I \to I$ given by $w_\circ(\al_i) = -\al_{i^*}$.

For an element $w \in \weyl$, we write $\ell(w)$ for its length. A sequence $\ii=(i_1,\ldots,i_l) \in I^l$ is said to be \emph{reduced} if
$s_{i_1}s_{i_2}\cdots s_{i_l}$ is a reduced expression of an element $w\in \weyl$.
We denote by $I(w)$ the set of all reduced sequences of $w \in \weyl$.
Two reduced sequences $\ii=( i_1 , i_2, \ldots, i_l)$ and $\ii'=( j_1 , j_2, \ldots, j_l)$ are said to be \emph{commutation equivalent} if $\ii'$ can be obtained from
$\ii$ by transforming some adjacent components $(i,j)$ into $(j,i)$ when $i \not \sim j$.
In this case, we denote by $\ii \overset{c}{\sim} \ii'$.
 We denote $[\ii]$ the equivalent class of $\ii$
 with respect to $\overset{c}{\sim}$, and we often call it the \emph{commutation class} of $\ii$.

Let us consider a reduced sequence $\ii_\circ =( i_1, \ldots, i_\ell)$ in $I(w_\circ)$. Note that each $\ii_\circ$ induces a \emph{convex total order} on  $\Phi_+$ in the following way (cf.~\cite{Papi}). Setting  $\be_k^{\ii_\circ}\seteq  s_{i_1}s_{i_2} \cdots s_{i_{k-1}}(\al_{i_k})$,
we have $\{ \be_k^{\ii_\circ} \ | \ 1 \le k \le \ell \} =\Phi_+$ with $|\Phi_+|=\ell$ and
the total order defined by  $\be_k^{\ii_\circ} \le_{\ii_\circ} \be_l^{\ii_\circ} \iff k \le l$ satisfies the following property: For $\al,\be \in \Phi_+$ with
$\al+\be \in \Phi_+$, we have either $\al <_{\ii_\circ} \al+\be  <_{\ii_\circ} \be$ or
$\be \le_{\ii_\circ} \al+\be  \le_{\ii_\circ} \al$. One step further, for a commutation class $[\ii_\circ]$, we can define the \emph{convex partial order} $\prec_{[\ii_\circ]}$ on $\Phi_+$ by
$\al \prec_{[\ii_\circ]} \be$ if and only if $\al \le_{\ii'_\circ} \be$ for all $\ii'_\circ \in [\ii_\circ]$.

For a commutation class $[\ii_\circ]$ of $w_\circ$ and $\al \in \Phi_+$, we define $[\ii_\circ]$-residue of $\al$, denoted by
$\res^{[\ii_\circ]}(\al)$ to be $i_k \in I$ such that $\be^{\ii_\circ}_k =\al$. This is well-defined; that is, for any $\ii'_\circ=(j_1,\ldots,j_\ell) \in [\ii_\circ]$ with $\be^{\ii'_\circ}_l=\al$, we have $i_k=j_l$.
For a reduced sequence $\ii_\circ=(i_1, i_2, \ldots,i_\ell)$ of $w_\circ$, it is known that the expression $\ii_\circ'\seteq (i_2,\cdots,i_\ell, i_1^*)$ is also a reduced sequence
of  $w_\circ$. This operation is sometimes referred to as the \emph{combinatorial reflection functor} and we write $r_{i_1}\ii_\circ = \ii'_\circ$. Also it induces the operation on commutation classes of $w_\circ$ (i.e., $r_{i_1}[\ii_\circ] = [r_{i_1}\ii_\circ]$
is well-defined if there exists a reduced sequence $\ii_\circ' =(j_1,j_2,\ldots,j_\ell) \in [\ii_\circ]$ such that $j_1=i_1$).

\subsection{Exponents} \label{subsec: stat}
An element $\ue=\seq{\ek_\beta}$ of $\Z_{\ge 0}^{\Phi_+}$ is called an \emph{exponent}.
For an exponent $\ue$, we set $\wt(\ue) \seteq \sum_{\beta\in\Phi^+} \ek_\beta \be \in \rl^+$.

\begin{definition}[cf. \cite{McNa15,Oh18}] \label{def: bi-orders}
For a reduced sequence $\ii_\circ=(i_1,\ldots,i_{\ell})$ and a commutation class $[\ii_\circ]$,
we define the partial orders $<^\ttb_{\ii_\circ}$ and $\prec^\ttb_{[\ii_\circ]}$ on $\Z_{\ge 0}^{\Phi_+}$ as follows:
\begin{enumerate}[{\rm (i)}]
\item $<^\ttb_{\ii_\circ}$ is the bi-lexicographical partial order induced by $<_{\ii_\circ}$. Namely, $\ue<^\ttb_{\ii_\circ}\ue'$ if
\begin{itemize}
\item $\wt(\ue)=\wt(\ue')$,
\item there exists $\al\in\Phi^+$ such that
$\ek_\al<\ek'_\al$ and $\ek_\be=\ek'_\be$ for any $\be$ such that $\be<_{\ii_\circ}\al$,
\item there exists $\eta\in\Phi^+$ such that
$\ek_\eta<\ek'_\eta$ and $\ek_\zeta=\ek'_\zeta$ for any $\zeta$ such that $\eta<_{\ii_\circ}\zeta$.
\end{itemize}
\item \label{eq: crazy order} For exponents $\ue$ and $\ue'$, we define $\ue \prec^\ttb_{[\ii_\circ]} \ue'$ if the following conditions are satisfied:
$$\ue<^\ttb_{\ii'_\circ} \ue'\qt{for all $\ii'_\circ \in [\ii_\circ]$.}$$
\end{enumerate}
\end{definition}

\smallskip

We say that an exponent $\ue=\seq{\ek_\be} \in \Z_{\ge0}^{\Phi^+}$ is \emph{$[\ii_\circ]$-simple} if it is minimal with respect to the partial order $\prec^\ttb_{[\ii_\circ]}$. For a given $[\ii_\circ]$-simple exponent
$\us=\seq{\sk_\beta} \in \Z_{\ge0}^{\Phi^+}$, we call a cover\footnote{Recall that a \emph{cover} of $x$ in a poset $P$ with partial order $\prec$ is an element $y \in P$ such that $x \prec  y$ and there
does not exists $y' \in P$ such that $x \prec y' \prec y$.}
of $\us$ under $\prec^{\ttb}_{[\ii_\circ]}$ a \emph{$[\ii_\circ]$-minimal exponent for $\us$}.
A pair $\langle \hspace{-.6ex} \langle \al,\be \rangle \hspace{-.6ex} \rangle$ is called a {\em $[\ii_\circ]$-pair} if
$\al,\be\in\Phi^+$ satisfy $\be\not\preceq^\ttb_{[\ii_\circ]}\al$.
We regard a $[\ii_\circ]$-pair $\uup\seteq\pair{\al,\be}$ as an exponent
by $\pk_\al=\pk_\be=1$ and $\pk_\gamma=0$ for
$\gamma\not=\al,\be$.

Note that, for any $\ga \in \Phi_+ \setminus \Pi$ and $\ii_\circ \in I(w_\circ)$, there exist a $[\ii_\circ]$-minimal $[\ii_\circ]$-pair $\pair{\al,\be}$ such that $\al+\be =\ga$.

\subsection{Quantum virtual Grothendieck rings via $t$-quantized Cartan matrices} \label{subsec: Quantum virtual Grothendieck rings}
For $i,j \in I$, we set
$$  \usfc_{i,j}(t) \seteq   \delta(i=j)(t+t^{-1}) + \delta(i\ne j) \sfc_{i,j}$$
and call the matrix $\usfC(t) \seteq (\usfc_{i,j}(t))_{i,j \in I}$ the \emph{$t$-quantized Cartan matrix}. Then
$$  \usfB(t) \seteq \usfC(t) \sfD^{-1} \text{ is a symmetric matrix in ${\rm GL}_n(\Z[t^{\pm 1}])$}.$$
Regarding $\usfB(t)$ as a $\Q(q)$-valued matrix, the matrix $\usfB(t)$ is invertible and we denote its inverse by $\tusfB(t) = (\tusfB_{i,j}(t))_{i,j \in I}$. We write
$$ \tusfB_{i,j}(t) = \sum_{u \in \Z} \tusfb_{i,j}(u)t^u \in \Z(\!( q )\!) $$
for the Laurent expansion of the $(i,j)$-entry $\tusfB_{i,j}(t)$ at $t=0$. Note that $\tusfb_{i,j}(u)=\tusfb_{j,i}(u)$ for $i,j \in I$ and $u \in \Z$.

\begin{lemma}[\cite{HL15,FO21,FM21,KO22}]\label{lem: non-zero b}
For any $i,j \in I$ and $u \in \Z$, we have
\ben
\item \label{it: non-zero b} $\tusfb_{i,j}(u) =0$ if $u \le d(i,j)$ or $d(i,j) \equiv_2 u$,
\item$\tusfb_{i,j}(d(i,j)+1) =\max(d_i,d_j)$,
\item $\tusfb_{i,j}(u+\sfh) = \tusfb_{i,j^*}(u)$ and $\tusfb_{i,j}(u+2\sfh) = \tusfb_{i,j}(u)$ for $u \ge 0$,
\item $\tusfb_{i,j}(\sfh-u) = \tusfb_{i,j^*}(u)$  for $0 \le u \le \sfh$ and $\tusfb_{i,j}(2\sfh-u) = -\tusfb_{i,j}(u)$  for $0 \le u \le 2\sfh$,
\item $\tusfb_{i,j}(u) \ge 0$ for $0 \le u \le \sfh$ and $\tusfb_{i,j}(u) \le 0$ for $\sfh \le u \le 2\sfh$.
\ee
\end{lemma}

As seen in Lemma~\ref{lem: non-zero b}, $\tusfB_{i,j}(t)$ is periodic. Thus it is enough to compute
\begin{align} \label{eq:essential part of Laurent expansion of tQCM}
\tde_{i,j}(t) \seteq \sum_{u=0}^{\sfh-1} \tusfb_{i,j}(u) t^u \quad \text{ for each } i,j \in I.
\end{align}
Note that $\tde_{i,j}(t)$'s are completely calculated (see~\cite{HL15,FO21,KO22}). We exhibit  $\tusfB_{i,j}(t)$ of non-simply-laced types only:

\begin{proposition}  \label{prop:formulas for inverses of tQCMs}
For $i,j\in I$, the closed-form formula of $\tusfB_{i,j}(t)$ $(i \le j)$ is given as follows:
\ben
\item For $\Dynkin$ of type $B_n$ or $C_n$ with their Coxeter number $\sfh=2n$, we have
\begin{align} \label{eq: B_n C_n tusfB}
\tde_{i,j}(t) = \bc
\max(d_i,d_j) \sum_{s=1}^i  t^{n-i-1+2s}  & \text{ if } i \le j=n, \\
\max(d_i,d_j) \sum_{s=1}^i  (t^{j-i+2s-1} + t^{2n-j-i+2s-1})  & \text{ if } i \le j <n.
\ec
\end{align}
\item For $\Dynkin$ of type $F_4$ with its Coxeter number $\sfh=12$, we have
\begin{align*}
&\tde_{1,1}(t) = 2(t+t^{5}+t^{7}+t^{11} ),  && \tde_{1,2}(t)  = 2(t^{2}+t^{4}+2t^{6}+t^{8}+t^{10} ),  \allowdisplaybreaks\\
&\tde_{1,3}(t)  = 2(t^{3}+t^{5}+t^{7}+t^{9} ),  && \tde_{1,4}(t)  = 2(t^{4}+t^{8} ),  \allowdisplaybreaks\\
&\tde_{2,2}(t)  = 2(t+2t^{3}+3t^{5}+3t^{7}+2t^{9}+t^{11} ),   && \tde_{2,3}(t)  = 2(t^{2}+2t^{4}+2t^{6}+2t^{8}+t^{10} ),  \allowdisplaybreaks\\
&\tde_{2,4}(t)  = 2(t^{3}+t^{5}+t^{7}+t^{9} ),   && \tde_{3,3}(t)  = t+2t^{3}+3t^{5}+3t^{7}+2t^{9}+t^{11},  \allowdisplaybreaks\\
&\tde_{3,4}(t)  = t^{2}+t^{4}+2t^{6}+t^{8}+t^{10},   &&  \tde_{4,4}(t)  = t+t^{5}+t^{7}+t^{11}.
\end{align*}
\item For $\Dynkin$ of type $G_2$ with its Coxeter number $\sfh=6$, we have
$$
\tde_{1,1}(t) = t+2t^{3}+t,  \qquad \tde_{1,2}(t)  =3(t^{2}+t^{4}), \qquad \tde_{2,2}(t)  =3(t+2t^{3}+t).
$$
\ee
\end{proposition}

For $i,j \in I$ and $k\in \Z$, we set
\begin{align} \label{eq:def of tde ij k}
  \tde_{i,j}[k] \seteq  \tusfb_{i,j}(k-1).
\end{align}

Based on Lemma~\ref{lem: non-zero b}~\eqref{it: non-zero b}, we fix a parity function $\ep: I \to \{0,1\}$ such that $\ep_i \ne \ep_j$ for $i \sim j$. We define a subset $\hDynkin_0$
of $I \times \Z$ as follows:
\begin{align} \label{eq: tDynkin 0}
\hDynkin_0 = \{ (i,p) \in I \times \Z \ | \  p \equiv_2 \ep_i \}.
\end{align}

Let $q$ be an indeterminate with a formal square root $q^{1/2}$. For $i \in I$, we set $q_i \seteq q^{d_i}$.

\begin{definition}[\cite{Nak04,Her04,VV03,KO22}] \label{def:quantum torus Xqg}
We define a quantum torus $(\calX_q(\g),*)$ to be the $\Z[q^{\pm 1/2}]$-algebra generated by $\{ \tX_{i,p}^{\pm} \ | \ (i,p) \in \hDynkin_0\}$ subject to following relations:
$$   \tX_{i,p} * \tX_{i,p}^{-1} = \tX_{i,p}^{-1} * \tX_{i,p} =1 \quad \text{ and } \quad \tX_{i,p} *\tX_{j,s} = q^{\ucalN(i,p;j,s)}\tX_{j,s}*\tX_{i,p},$$
where
\begin{align*}
\ucalN(i,p;j,s) = \tusfb_{i,j}(p-s-1)-\tusfb_{i,j}(s-p-1)-\tusfb_{i,j}(p-s+1)+\tusfb_{i,j}(s-p+1).
\end{align*}
\end{definition}
\noindent
If there is no danger of confusion, we write $\calX_q$ instead of $\calX_q(\g)$.
Note that $\tX_{i,p}*\tX_{j,p} = \tX_{j,p}*\tX_{i,p}$ for $(i,p),(j,p) \in \hDynkin_0$ by Lemma~\ref{lem: non-zero b}.

\smallskip

There exists a $\Z$-algebra anti-involution $\overline{( \cdot )}$ on $\calX_q$ (\cite{Her04,KO22}) given by
\begin{align} \label{eq:bar-inv on Xq}
    q^{1/2} \longmapsto  q^{-1/2} \quad \text{ and } \quad  \tX_{i,p}  \longmapsto   q_i\tX_{i,p}.
\end{align}

We say that $\tm \in \calX_q$ is a monomial if it is a product of generators $\tX_{i,p}^{\pm}$ and $q^{\pm 1/2}$; i,e;
\begin{align}  \label{eq: uip}
\tm = q^a\prod_{(i,p) \in \hDynkin_0}^{\to} \tX_{i,p}^{u_{i,p}(\tm)}  \quad \text{ for some  $a \in \tfrac{1}{2}\Z$ and $u_{i,p}(\tm) \in \Z$.}
\end{align}
 A monomial $\tm$ is called \emph{dominant} if
it is a product of generators $\tX_{i,p}$ (not $\tX_{i,p}^{-1}$) and $q^{\pm 1/2}$; equivalently $u_{i,p}(\tm) \in \Z_{\ge 0}$ in ~\eqref{eq: uip} for all $(i,p) \in \hDynkin_0$.
A monomial $\tm \in \calX_q$ is said to be
\emph{commutative} if $\tm= \overline{\tm}$.
We denote by
$\calX$ the commutative Laurent polynomial ring obtained by specializing $q$ at $1$ to $\calX_q$. We write the map $\ev_{q=1}: \calX_q \to \calX$ and set $X_{i,p} \seteq \ev_{q=1}(q^a\tX_{i,p})$ for any $a \in \tfrac{1}{2}\Z$.
Then we have
$$\calX \simeq \Z[X_{i,p}^{\pm 1} \ | \ (i,p) \in \hDynkin_0 ].$$

\begin{remark} To a monomial $m \in \calX$, we can associate the commutative monomial $\um \in \calX_q$
such that $\ev_{q=1}(\um)=m$. Hereafter, we use $\um$ for commutative monomials in $\calX_q$.
\end{remark}

We set $\calM$ the set of all monomials in $\calX$ and $\calM_+$ the subset of dominant monomials in $\calM$.
For a monomial $m \in \calM$, we set $$\range(m) \seteq [a,b], $$ where $$\text{$a = \min( p \in \Z \ |  \ X_{i,p}^{\pm1} \text{ is a factor of $m$})$ and
$b = \max( p \in \Z \ |  \ X_{i,p}^{\pm 1} \text{ is a factor of $m$})$}.$$ For $m \in \calX$ with $\range(m)=[a,b]$, the monomial $m$  is said to be \emph{right-negative} if $X_{i,b}$ has negative power for every factor  $X_{i,b}$ of $m$ ($i \in I$).

For an element $x \in \calX$, we denote by $\calM(x)$ the set of all monomials appearing in $x$. We set
$$  \range(x) \seteq \bigcup_{m \in \calM(x)} \range(m).$$
Finally, for a monomial $\tm \in \calX_q$ and an element $\tx \in \calX_q$, we set
$$ \range(\tm) \seteq \range(\ev_{q=1}(\tm))  \quad \text{ and } \quad \range(\tx) \seteq \displaystyle \scup_{\tm :\text{monomial in $\tx$}} \range( \tm  ).$$

For monomials $\tm,\tm' \in \calX_q$, we set
$$    \ucalN (\tm,\tm') \seteq \sum_{(i,p),(j,s) \in \hDynkin_0} u_{i,p}(\tm)u_{j,s}(\tm') \ucalN(i,p;j,s).$$
For commutative monomials $\um,\um'$, the monomial
$$   \um \cdot \um' \seteq q^{- \ucalN(\um,\um')/2}  \utm* \utm' \text{ is commutative again.}$$

For each $(i,p) \in I \times \Z$ with $(i,p-1) \in \hDynkin_0$, we define the commutative monomial $\tB_{i,p} \seteq \underline{B_{i,p}}$ where
$$    B_{i,p} = X_{i,p-1} X_{i,p+1}  \prod_{j \sim i}  X_{j,p}^{\sfc_{j,i}} \in \calX.$$
For each $i \in I$, we denote by $\frakK_{i,q}(\g)$ the $\Z[q^{\pm 1/2}]$-algebra of $\calX_q$ generated by
$$    \tX_{i,l} * (1+ q_i^{-1}\tB^{-1}_{i,l+1})   \text{ and } \tX_{j,s}^{\pm 1} \quad \text{ for } j \in I \setminus \{ 0 \} \text{ and } (i,p),(j,s) \in \hDynkin_0.$$

\begin{definition}[\cite{Her04,JLO1}]
We set
$$
\frakK_q(\g) \seteq \bigcap_{i \in I} \frakK_{i,q}(\g)
$$
and call it the  \emph{quantum virtual Grothendieck ring associated to $\usfC(t)$}. In particular, it is known as
\emph{quantum  Grothendieck ring}
when $\g$ is of simply-laced type.
\end{definition}
If there is no danger of confusion, we write $\frakK_q$ instead of $\frakK_q(\g)$.

\subsection{Bases} For monomials $m,m' \in \calX$, we write
\begin{align} \label{eq: Nakajima order}
m \leN m' \text{ if and only if } m^{-1}m' \text{ is a product of elements in } \{ B_{i,p} \ | \  (i,p+1) \in \hDynkin_0 \}.
\end{align}
This defines the \emph{Nakajima partial order} on $\calM$.
For monomials $\tm,\tm' \in \calX_q$, we write $\tm \leN \tm'$ if $\ev_{q=1}(\tm) \leN \ev_{q=1}(\tm')$.

\begin{theorem} [\cite{Her04,JLO1}, cf.~\cite{FM01}] \label{thm: F_q}
For each dominant monomial $m \in \calM_+$, there exists a unique element $F_q(m)$ of $\frakK_q$ such that
$$\text{$\um$ is the unique dominant monomial appearing in $F_q(m)$
and $\overline{F_q(m)}=F_q(m)$.}$$ Furthermore,
\bna
\item a monomial appearing in $F_q(m)-\um$ is strictly less than $\um$ with respect to $\leN$;
\item $\sfF_q \seteq \{ F_q(m) \ | \  m \in \calM_+ \}$ forms a $\Z[q^{\pm 1/2}]$-basis of $\frakK_q$;
\item $\frakK_q$ is generated by the set $\{ F_q(X_{i,p}) \ | \ (i,p) \in \hDynkin_0 \}$ as a $\Z[q^{\pm 1/2}]$-algebra;
\item \label{eq: range}
all $\calX_q$-monomials of $F_q(X_{i,p})- \underline{X_{i,p}}- \underline{X_{i^*,p+\sfh}}$  are product of $\tX_{j,u}^{\pm 1}$ with $p<u<p+\sfh$.
\ee
\end{theorem}

For $m \in \calM_+$, we set
$$ E_q(m) \seteq q^b \left( \st^{\to}_{p \in \Z} \left(   \st_{i; \in I, (i,p) \in \hDynkin_0}  F_q(X_{i,p})^{u_{i,p}(m)} \right) \right),$$
where $b$ is an element in $\tfrac{1}{2}\Z$ such that $\um$ appears in $E_q(m)$ with coefficient $1$.
We remark here that, for a fixed $p \in \Z$, the set  $\{ F_q(X_{i,p}) \ | \ (i,p) \in \hDynkin_0\}$  forms a commuting family so that $E_q(m) $ is well-defined.
Then the set $\sfE_q\seteq \{ E_q(m) \ | \  m \in \calM \}$ also forms a $\Z[q^{\pm 1/2}]$-basis of $\frakK_q$ and satisfies
\begin{align} \label{eq: uni E F}
E_q(m) = F_q(m) + \sum_{m'\lN m} C_{m,m'}F_q(m')
\end{align}
for some $C_{m,m'} \in \Z[q^{\pm 1/2}]$.
\smallskip

Using the basis $\sfF_q$ and $\sfE_q$, we construct a basis $\sfL_q$, called the \emph{canonical basis} of $\frakK_q$:

\begin{theorem}[\cite{Nak04,Her04,JLO1}]\label{thm: L_q}
For each $m \in \calM_+$, there exists a unique element $L_q(m)$ in $\frakK_q$ such that
\begin{eqnarray} &&
\parbox{90ex}{
\bnum
\item $\overline{L_q(m)} = L_q(m)$ and
\item \label{it: Lq(m)} $L_q(m)  = E_q(m) + \sum_{m' \lN m}  Q_{m,m'}(t) E_q(m') $ for some $ Q_{m,m'}(t) \in q\Z[q]$.
\ee
}\label{eq: L_q}
\end{eqnarray}
Furthermore,
\bna
\item a monomial appearing in $L_q(m)-\um$ is strictly less than $\um$ with respect to $\leN$;
\item $\sfL_q \seteq \{ L_q(m) \ | \  m \in \calM \}$ forms a $\Z[q^{\pm 1/2}]$-basis of $\frakK_q$;
\item $L_q(X_{i,p}) = F_q(X_{i,p})$ for all $(i,p) \in \hDynkin_0$.
\ee

\end{theorem}

Note that, for $r \in2\Z$, we have a $\Z[q^{\pm 1/2}]$-algebra endomorphism on $\calX_q$
\begin{equation} \label{eq: shift of spectral parameters}
\scrS_r:   \calX_q \to   \calX_q  \quad\text{ given by } \quad \tX_{i,p} \longmapsto \tX_{i,p+r}
\end{equation}
which induces a $\Z[q^{\pm 1/2}]$-algebra endomorphism on $\frakK_q$ such that
$$ \scrS_r(F_q(X_{i,p})) = F_q(X_{i,p+r}).$$
We call it the \emph{parameter shift by $r$}.

\section{Quivers and subrings} \label{sec: Quivers and heart subring}

In this section, we introduce various quivers such as Dynkin quivers, repetition quivers, and Auslander--Reiten (AR) quivers, which will be used for defining subrings of $\frakK_q$.
Then we review the notion of $Q$-weights and investigate the multiplicative relations among $F_q(Y_{i,p})$'s by using their $Q$-weights.
At the end of this section, we also review the quantum folded $T$-system in terms of $F_q(m)$ with KR-monomials $m$ and the truncation homomorphisms.

\subsection{Dynkin quivers} A Dynkin quiver $Q$ is a pair $(\Dynkin,\xi)$ consisting of (i) a Dynkin diagram $\Dynkin$ and (ii) a height function $\xi:\Dynkin_0 \to \Z$ satisfying
$| \xi_i -\xi_j | =1$ if $i \sim j$.  For a given Dynkin quiver $Q=(\Dynkin,\xi)$, we can associate an oriented graph whose underlying graph is $\Dynkin$ and arrows between $i$ and $j$ are assigned as follows:
$$
i \to j \quad \text{ if } i \sim j \text{ and } \xi_i =\xi_j+1.
$$
Conversely, for a given oriented graph on $\Dynkin$, we can associate a Dynkin quiver whose height function is well-defined up to integers (see \cite{KO22} for more details).

For a Dynkin quiver $Q=(\Dynkin,\xi)$, we call $i \in \Dynkin_0$ a \emph{source} (resp. \emph{sink}) of $Q$ if $\xi_i > \xi_j$ (resp. $\xi_i < \xi_j$) for all $j \in \Dynkin_0$ with $j \sim i$.
For a Dynkin quiver $Q$ and its source $i$, we denote by $s_iQ=(\Dynkin,s_i\xi)$ a Dynkin quiver on $\Dynkin$ whose height function $s_i\xi$ is defined as follows:
$$
(s_i\xi)_j \seteq \xi_j - \delta(i=j)2.
$$
Note that $i$ is a sink of the Dynkin quiver $s_iQ$.

\begin{remark} Recall that we fixed the parity function $\ep$ in Section~\ref{subsec: Quantum virtual Grothendieck rings}. Throughout this paper, we only consider Dynkin quivers $Q=(\Dynkin,\xi)$ such that
$ \xi_i \equiv_2 \ep_i$ for all $i \in \Dynkin_0$.
\end{remark}

For a sequence $\ii = (i_1,\ldots,i_l) \in I^l$, we say $\ii$  is \emph{adapted to $Q=(\Dynkin,\xi)$} if
$$
\text{ $i_k$ is a source of $s_{i_{k-1}}s_{i_{k-2}} \cdots s_{i_{1}}Q$ for all $1 \le k \le l$}.
$$
For instance, a sequence $(i_1,i_2,\ldots,i_n)$, for which $\xi_{i_1} \ge \xi_{i_2} \ge  \cdots \ge \xi_{i_n}$ and $\{ i_k \ | \ 1 \le k \le n \}=I$ is a reduced sequence of a Coxeter element $\tau$, is adapted to $Q=(\Dynkin,\xi)$ (see \cite{FO21,KO22}).
It is known that for each
Dynkin quiver $Q$, there exists a unique Coxeter element $\tau_Q$ all of whose reduced sequences are adapted to $Q$ and form a single commutation equivalent class. Conversely, for any Coxeter element $\tau$, there exists a unique Dynkin quiver $Q$ such that $\tau=\tau_Q$.

For each Dynkin quiver $Q$, it is also known that there exists a reduced sequence $\ii$ of $w_\circ$, which is adapted to $Q$. Furthermore, the set of all
reduced sequences adapted to $Q$ makes a single commutation class of $w_\circ$, denoted by $[Q]$. Note that $[Q] \ne [Q']$
unless there exists $k \in \Z$ such that $\xi_i - \xi'_i =k$ for all $i \in \Dynkin_0$, where $\xi$ and $\xi'$
are the height functions of $Q$ and $Q'$, respectively.  For a Dynkin quiver $Q$ with a source $i$, we have
\begin{align}  \label{eq: si ri}
[s_iQ] = r_i[Q].
\end{align}

\subsection{Repetition quiver and $Q$-coordinates} Let us recall the repetition quiver and related notions:
\begin{definition}  (\cite{FHOO,KO22,JLO1}) \label{def: convex subset} \hfill
\ben
\item \label{def: hDynkin}
We define the \emph{repetition quiver} $\hDynkin = (\hDynkin_0,\hDynkin_1)$ associated to $\Dynkin$ as follows:
\begin{align*}
&\text{$\hDynkin_0$ is as in \eqref{eq: tDynkin 0} and} \\
&\hDynkin_1 \seteq \{ (i,p)\To[{\;-\lan h_{i},\al_j \ran\;}] (j,p+1)   \mid  (i,p) \in \hDynkin_0, \ d(i,j)=1  \}.
\end{align*}
Here $(i,p) \To[{\;-\lan h_{i},\al_j \ran\;}](j,p+1)$ denotes $(-\lan h_{i},\al_j \ran)$-many arrows from $(i,p)$ to $(j,p+1)$.
\item \label{def: convexity} A subset $\calR \subset  \hDynkin_0$ is said to be \emph{convex} if it satisfies the following condition: For any oriented path $(x_1 \to x_2 \to \cdots \to x_l)$ consisting of arrows
in $\hDynkin$, we have $\{ x_1,x_2,\ldots,x_l\} \subset \calR$ if and only if $\{ x_1,x_l\} \subset \calR$.
\item For a height function $\xi$ on $\Dynkin$, set
\begin{align}\label{eq: Dynkinxi0}
\lxi\hDynkin_0 \seteq \{ (i,p) \in \hDynkin_0 \ | \  p \le \xi_i \}.
\end{align}
Note that $\lxi\hDynkin_0$ is a convex subset of $\hDynkin$ for any height function $\xi$.
\ee
\end{definition}

\begin{example}  \label{ex: repetition quivers}
Some repetition quivers $\hDynkin$'s are depicted as follows:
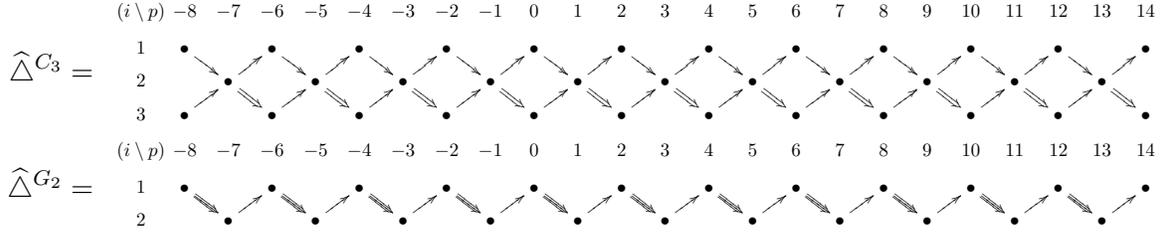
\begin{align*}
\hDynkin^{C_3} & =  \raisebox{9mm}{
\scalebox{0.65}{\xymatrix@!C=0.5mm@R=2mm{
(i\setminus p) & -8 & -7 & -6 &-5&-4 &-3& -2 &-1& 0 & 1& 2 & 3& 4&  5
& 6 & 7 & 8 & 9 & 10 & 11 & 12 & 13 & 14   \\
1&\bullet \ar@{->}[dr]&& \bullet \ar@{->}[dr]  &&\bullet\ar@{->}[dr]
&&\bullet \ar@{->}[dr]  && \bullet \ar@{->}[dr]  &&\bullet \ar@{->}[dr]  &&  \bullet \ar@{->}[dr]
&&\bullet \ar@{->}[dr]  && \bullet \ar@{->}[dr]  &&\bullet \ar@{->}[dr]   && \bullet\ar@{->}[dr]  &&
\bullet   \\
2&&\bullet \ar@{=>}[dr]\ar@{->}[ur] && \bullet \ar@{=>}[dr]\ar@{->}[ur]  &&\bullet \ar@{=>}[dr]\ar@{->}[ur]
&& \bullet \ar@{=>}[dr]\ar@{->}[ur] && \bullet\ar@{=>}[dr] \ar@{->}[ur] && \bullet \ar@{=>}[dr]\ar@{->}[ur] &&\bullet \ar@{=>}[dr]\ar@{->}[ur] &
&\bullet \ar@{=>}[dr]\ar@{->}[ur] &&\bullet\ar@{=>}[dr] \ar@{->}[ur] && \bullet \ar@{=>}[dr]\ar@{->}[ur]
&&\bullet \ar@{=>}[dr]\ar@{->}[ur]  \\
3&\bullet  \ar@{->}[ur]&& \bullet  \ar@{->}[ur]  &&\bullet \ar@{->}[ur]
&&\bullet  \ar@{->}[ur]  && \bullet \ar@{->}[ur]  &&\bullet  \ar@{->}[ur] &&  \bullet  \ar@{->}[ur]
&&\bullet  \ar@{->}[ur]  && \bullet  \ar@{->}[ur]  &&\bullet  \ar@{->}[ur]  && \bullet \ar@{->}[ur]  &&
\bullet     }}}
\\
\hDynkin^{G_2} & =  \raisebox{6mm}{
\scalebox{0.65}{\xymatrix@!C=0.5mm@R=2mm{
(i\setminus p) & -8 & -7 & -6 &-5&-4 &-3& -2 &-1& 0 & 1& 2 & 3& 4&  5
& 6 & 7 & 8 & 9 & 10 & 11 & 12 & 13 & 14  \\
1&\bullet \ar@{=>}[dr]&& \bullet \ar@{=>}[dr]  &&\bullet\ar@{=>}[dr]
&&\bullet \ar@{=>}[dr]  && \bullet \ar@{=>}[dr]  &&\bullet \ar@{=>}[dr]  &&  \bullet \ar@{=>}[dr]
&&\bullet \ar@{=>}[dr]  && \bullet \ar@{=>}[dr]  &&\bullet \ar@{=>}[dr]   && \bullet\ar@{=>}[dr]  && \bullet
  \\
2&& \bullet \ar@{-}[ul] \ar@{->}[ur] &&\bullet \ar@{-}[ul]\ar@{->}[ur]  &&\bullet \ar@{-}[ul]\ar@{->}[ur]   &&\bullet\ar@{-}[ul] \ar@{->}[ur] && \bullet \ar@{-}[ul]\ar@{->}[ur]
&&\bullet \ar@{-}[ul]\ar@{->}[ur]  && \bullet \ar@{-}[ul]\ar@{->}[ur]   &&\bullet\ar@{-}[ul] \ar@{->}[ur]   &&\bullet \ar@{-}[ul]\ar@{->}[ur]   &&
\bullet \ar@{-}[ul]\ar@{->}[ur] &&\bullet\ar@{-}[ul]\ar@{->}[ur]  && }}}
\end{align*}
\end{example}

In what follows, we assign a coordinate system to $\hDynkin$ arising from a Dynkin quiver $Q=(\Dynkin,\xi)$. For each $i \in \Dynkin_0$, define
$$\ga_i^Q \seteq (1-\tau_Q)\varpi_i \in \Phi_+.$$
There exists a bijection $\phi_Q: \hDynkin_0 \to \Phi_+ \times \Z$ defined recursively as follows \cite{HL15,FO21,KO22}:
\bnum
\item $\phi_Q(i,\xi_i) = (\ga_i^Q,0)$ for each $i \in I$.
\item If $\phi_Q(i,p) =(\al,k)$, we have
$$
\phi_Q(i,p \pm 2) = \bc
(\tau_Q^{\mp} \al,k) & \text{ if } \tau_Q^{\mp} \al \in \Phi_+,\\
(-\tau_Q^{\mp} \al,k \pm 1) & \text{ if } \tau_Q^{\mp} \al \in \Phi_-.
\ec
$$
\ee
We call $\phi_Q$ the \emph{$Q$-coordinate of $\hDynkin_0$.}
For $(i,p)\in \hDynkin_0$ with $\phi_Q(i,p)= (\al,k)$, we set
$$\pi_Q(i,p) \seteq (-1)^{k}\al \in \Phi.$$

\begin{lemma} {\rm (\cite[(3.18)]{FO21}, \cite[Lemma 3.11]{KO22})\bf{.}} For a Dynkin quiver $Q$ and $(i,p) \in \hDynkin_0$, we have
$$  \pi_Q(i,p)  =  \tau_Q^{(\xi_i-p)/2}  \ga_i^Q.$$
\end{lemma}

\begin{theorem} [\cite{HL15,FO21,KO22}] \label{thm: N and wt}
Let $(i,p),(j,s) \in \hDynkin_0$ and $Q$ be any Dynkin quiver. Then we have
\begin{align*}
\ucalN(i,p;j,s) &= (-1)^{k+l+\delta(p\ge s)} \delta((i,p)\ne(j,s)) (\al,\be),
\end{align*}
where $\phi_Q(i,p)=(\al,k)$ and $\phi_Q(j,s)=(\be,l)$.
\end{theorem}

\subsection{$Q$-weights, block decomposition of $\frakK_q$ and AR-quivers}
Based on Theorem~\ref{thm: N and wt},  for an $\calX_q$-monomial $\tm$ (resp.~$\calX$-monomial $m$), we define
$$   \wt_Q(\tm) = \sum_{(i,p) \in \hDynkin_0}  u_{i,p}(\tm) \pi_Q(i,p)  \quad \text{ (resp.~}  \wt_Q(m) = \sum_{(i,p) \in \hDynkin_0}  u_{i,p}(m) \pi_Q(i,p)) $$
and call it\emph{ the $Q$-weight of $\tm$ $($resp.~$m)$}.
Note that $\wt_Q(\tm*\tm') = \wt_Q(\tm) + \wt_Q(\tm')$ for $\calX_q$-monomials $\tm$ and $\tm'$.

\begin{corollary} {\rm (\cite[(5.3)]{FO21}, \cite[Corollary 5.6]{KO22})\bf{.}} \label{cor: B=0}
For any $(i,p) \in I \times \Z$ with $(i,p-1) \in \hDynkin_0$ and a Dynkin quiver $Q$, we have
$$    \wt_Q(\tB_{i,p})=   \wt_Q(B_{i,p})  =0.$$
\end{corollary}

An element $\tx \in \calX_q$ is said to be\emph{ homogeneous} if
$$
\wt_Q(\tm) = \wt_Q(\tm')  \quad \text{ for any monomials $\tm,\tm'$ in $\tx$ and any Dynkin quiver $Q$.}
$$
In this case, we set $\wt_Q(\tx) \seteq \wt_Q(\tm) $ for any monomial $\tm$ in $\tx$.
By Theorem~\ref{thm: F_q}, Theorem~\ref{thm: L_q}, and Corollary~\ref{cor: B=0}, the elements $F_q(m)$ and $L_q(m)$ are homogeneous and their $Q$-weights are the same as $\wt_Q(m)$. The proposition below is proved in~\cite{CM05,KKOP22,FO21} when $Q$ is of simply-laced type, in the context of blocks of category.

\begin{proposition}
For any Dynkin quiver $Q$, we have a natural decomposition of $\frakK_q(\g)$ with respect to $\rl$:
$$
\frakK_q  = \soplus_{\be \in  \rl} \frakK^Q_q[\be],
$$
where
$$
\frakK^Q_q[\be] = \soplus_{ \substack{m \in \calM_+, \\ \wt_Q(m)=\be } } \Z[q^{\pm 1/2}] F_q(m)= \soplus_{ \substack{m \in \calM_+, \\ \wt_Q(m)=\be } } \Z[q^{\pm 1/2}] L_q(m)
= \soplus_{ \substack{m \in \calM_+, \\ \wt_Q(m)=\be } } \Z[q^{\pm 1/2}] E_q(m).
$$
\end{proposition}

\begin{proposition} \label{prop: wt pairing}
Let $\tx,\tx' \in \calX_q$ be homogeneous elements whose ranges are $[a,b]$ and $[a',b']$, respectively. Assume that either $a < b < a'< b'$ or $a' < b' < a< b$. Then we have
$$
\tx * \tx' = (-1)^{\delta(b'<a)} q^{(\wt_Q(\tx),\wt_Q(\tx'))} \tx'*\tx
$$
for any Dynkin quiver $Q$.
\end{proposition}

\begin{proof}
This is a direct consequence of Theorem~\ref{thm: N and wt}.
\end{proof}

For a Dynkin quiver $Q=(\Dynkin,\xi)$, let
$$ \Gamma^Q_0 \seteq  \phi_{Q}^{-1}(\Phi_+ \times \{ 0 \}).$$
Then $\Gamma^Q_0 $ is a convex subset of $\hDynkin_0$. We set $$\Gamma^Q \seteq {}^{\Gamma^Q_0}\hDynkin,$$ where we understand $\Gamma^Q$ as a sub-quiver of $\hDynkin$ with $\Gamma^Q_0$ as its vertex set and the arrows induced from $\hDynkin$,  and call it \emph{the $($combinatorial$)$ AR-quiver of $Q$}.
Note that the vertices of $\Gamma^Q$ can be labeled by $\Phi_+$ via $\phi_Q$. Thus, we often identify $\Gamma^Q_0 $ with $\Phi_+$.

\begin{example} \label{ex: C3 AR}
For a Dynkin quiver $ Q=  \xymatrix@R=0.5ex@C=6ex{    *{\circled{2}}<3pt> \ar@{->}[r]_<{1 \ \ } & *{\circled{2}}<3pt> \ar@{->}[r]_<{2 \ \ } &*{\circled{4}}<3pt>
\ar@{-}[l]^<{\ \ 3   }  }$ of $C_3$ with $\xi_1=3$, $\xi_2=2$ and $\xi_3=1$,
$\Gamma^Q$ can be depicted as follows:
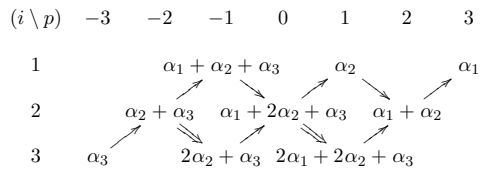
\begin{align*}
 \raisebox{3.2em}{ \scalebox{0.7}{\xymatrix@!C=2ex@R=2ex{
(i\setminus p) & -3&  -2 & -1 & 0 & 1 & 2  & 3\\
1&&& \al_1+\al_2+\al_3\ar@{->}[dr]&& \al_2 \ar@{->}[dr] && \al_1   \\
2&& \al_2+\al_3 \ar@{=>}[dr]\ar@{->}[ur] && \al_1+2\al_2+\al_3 \ar@{=>}[dr] \ar@{->}[ur]  && \al_1+\al_2 \ar@{->}[ur] \\
3& \al_3  \ar@{->}[ur] && 2\al_2+\al_3 \ar@{->}[ur]  && 2\al_1+2\al_2+\al_3 \ar@{->}[ur]  }}}
\end{align*}
\end{example}

For a source $i$ of $Q$, the quiver $\Gamma^{s_iQ}$ can be obtained from $\Gamma^Q$ in the following way:
\begin{eqnarray} &&
\parbox{87ex}{
\bnum
\item Each $\be \in \Phi^+ \setminus \{ \al_i \}$ is located at $(j,p)$ in $\Gamma^{s_iQ}$ if $s_i(\be)$ is at $(j,p)$ in $\Gamma^Q$;
\item \label{it: al_i poisition} $\al_i$ is located at $(i,\xi_i-\sfh)$  in $\Gamma^{s_iQ}$, while $\al_i$ is at $(i,\xi_i)$ in $\Gamma^Q$.
\ee
}\label{eq: Qd properties}
\end{eqnarray}
We refer the operation in ~\eqref{eq: Qd properties}  the \emph{reflection} from $\Gamma^Q $ to $ \Gamma^{s_iQ}$. In particular, for $i \ne j$ with $\phi_Q(s_i(\al_j))=(i,p)$, we have
\begin{align} \label{eq: j to j}
\phi_{Q}^{-1}(s_i(\al_j),0) =(i,p) = \phi_{s_i(Q)}^{-1}(\al_j,0).
\end{align}

\begin{example} \label{ex: s1 C3 AR}
In Example~\ref{ex: C3 AR}, vertex $1$ is a source of $Q$. Thus $\Gamma^{s_1Q}$ can be obtained from~\eqref{eq: Qd properties}:
\begin{align*}
 \raisebox{3.2em}{ \scalebox{0.7}{\xymatrix@!C=4ex@R=2ex{
(i\setminus p) & -3&  -2 & -1 & 0 & 1 & 2   \\
1&\al_1 \ar@{->}[dr] && \al_2+\al_3\ar@{->}[dr]&& \al_1+\al_2 \ar@{->}[dr]    \\
2&& \al_1+\al_2+\al_3 \ar@{=>}[dr]\ar@{->}[ur] && \al_1+2\al_2+\al_3 \ar@{=>}[dr] \ar@{->}[ur]  && \al_2   \\
3& \al_3  \ar@{->}[ur] && 2\al_1+2\al_2+\al_3 \ar@{->}[ur]  &&   2\al_2+\al_3 \ar@{->}[ur]  }}}
\end{align*}

\end{example}

We set
\begin{align} \label{eq:LqQ al m}
L_q^Q(\al,m) \seteq L_q(X_{i,p}) \quad \text{ where } \phi_Q(i,p) = (\al,m) \in \Phi_+ \times \Z.
\end{align}

\begin{proposition}[\cite{HL15,FO21,KO22}]
The quiver $\Gamma^Q$ realizes the convex partial order $\prec_{[Q]}$ in the following sense$:$
$\al \prec_{[Q]} \be$ if and only if there exists a path in $\Gamma^Q $ starting from $\phi_Q^{-1}(\be,0)$ to
$\phi_Q^{-1}(\al,0)$.
\end{proposition}

We say that a total ordering $\boldsymbol{(} \be_1,\ldots \be_{\ell}\boldsymbol{)}$ of $\Phi_+$ is a \emph{compatible reading} of $\Gamma^Q$ if we have $k<l$ whenever there is a path from $\beta_l$ to $\beta_k$
in $\Gamma^Q$.

\begin{proposition}[\cite{OS19a}]\label{prop: reading}
$\ii_\circ=(i_1,\ldots ,i_\ell)$ is a reduced sequence of $w_\circ$ in $[Q]$ if and only if there exists a compatible reading $\Phi_+=\boldsymbol{(} \be_1,\ldots \be_{\ell}\boldsymbol{)}$ of $\Gamma^Q$
such that $i_k=\res^{[Q]}(\be_k)$ for all $1 \le k \le \ell$.
\end{proposition}

\begin{proposition}[\cite{HL15,FO21,KO22}]\label{prop: Gamma Q 0}
Let $Q=(\Dynkin,\xi)$ be a Dynkin quiver.  Then we have
\bna
\item $\Gamma^Q_0 = \{ (i,p) \in \hDynkin_0 \ | \  \xi_{i^*}- \sfh < p \le \xi_{i} \}$,
\item $\phi_Q(i^*,p \pm \sfh)=(\al,k \mp 1)$ if $\phi_Q(i,p)=(\al,k)$,
\item for $\be\in\Phi_+$ with $\phi_Q^{-1}(\be,0)=(i,p)$, we have $\res^{[Q]}(\be)=i$.
\ee
\end{proposition}

By Theorem~\ref{thm: N and wt} and Proposition~\ref{prop: Gamma Q 0}, we have an automorphism $\frakD_q$ on $\calX_q$ given by
\begin{align} \label{eq:dual map Dq}
    \frakD_q(\tX_{i,p})  = \tX_{i^*,p+\sfh}.
\end{align}
Since $\frakD_q(\tB_{i,p})= \tB_{i^*,p+\sfh}$, it induces an automorphism on $\frakK_q$ also.

We denote by $\frakD$ the automorphism on $\calX$ satisfying
$ \frakD \circ \ev_{q=1} = \ev_{q=1}  \circ   \frakD_q$. Then we have $\frakD_q(F_q(m)) = F_q( \frakD(m))$ for any $m \in \calM_+$.

\begin{definition}{(cf.~\cite[Proposition 6.18]{FM01})\bf{.}}
    For $k\in \Z$, we call
$\frakD^k_q(F_q(m))$ the \emph{$k$-dual} of $F_q(m)$.
\end{definition}

\smallskip

Recall that
$$L_q(X_{i,p})=E_q(X_{i,p})=F_q(X_{i,p}) \qquad \text{ for all } (i,p) \in \hDynkin_0,$$
by definition.

\begin{proposition} \label{prop: quantum Boson1}
Let $(i,p) \in \hDynkin_0$. Then we have
\begin{align}\label{eq: quantum Boson1}
F_q(X_{i,p}) *  F_q(X_{i^*,p+\sfh}) =  q_i^{-2} F_q(X_{i^*,p+\sfh}) * F_q(X_{i,p}) + (1-q_i^{-2}).
\end{align}
\end{proposition}

\begin{proof}
Take the Dynkin quiver $Q=(\Dynkin,\xi)$
such that $i$ is a source of $Q$ with $\xi_i=p$. Then we have $\phi_Q(i,p)=(\al_i,0)$, $\phi_Q(i^*,p+\sfh)=(\al_i,-1)$ and hence $\wt_Q(F_q(X_{i,p})) = \al_i$, $\wt_Q(F_q(X_{i^*,p+\sfh})) = -\al_i$.

Rewriting the RHS of~\eqref{eq: quantum Boson1} as
\begin{align*}
(F_q(X_{i,p}) - \underline{X_{i^*,p+\sfh}^{-1}} + \underline{X_{i^*,p+\sfh}^{-1}} ) * ( F_q(X_{i^*,p+\sfh}) - \underline{X_{i^*,p+\sfh}} + \underline{X_{i^*,p+\sfh}}),
\end{align*}
our assertion follows from Theorem~\ref{thm: F_q}~\eqref{eq: range} and Proposition~\ref{prop: wt pairing}.
\end{proof}

\begin{proposition} \label{prop: far enough}
Let $(i,p),(j,s) \in \hDynkin_0$ with $p-s>\sfh$. Then we have
\begin{align}\label{eq: quantum Boson2}
F_q(X_{i,p}) *  F_q(X_{j,s}) =  q^{-(\wt_Q(X_{i,p}),\wt_Q(X_{j,s}) )}  F_q(X_{j,s}) * F_q(X_{i,p}),
\end{align}
for any Dynkin quiver $Q$.
\end{proposition}

\begin{proof}
By Theorem~\ref{thm: F_q}~\eqref{eq: range}  and Proposition~\ref{prop: wt pairing}, $F_q(X_{i,p})$ and  $F_q(X_{j,s})$ $q$-commute as asserted.
\end{proof}

\subsection{Heart subrings}

\begin{proposition}[\cite{FHOO,JLO1}]\label{prop: convex subrings}
Let $\sfS$ be a convex subset of $\hDynkin_0$, and
$\calM^\sfS$ be the set of all monomials $m$ satisfying the following condition:
\begin{align}\label{eq: M S}
\text{$X_{i,p}^{\pm 1}$ is a factor of $m \in \calM^\sfS$ if and only if $(i,p) \in \sfS$.}
\end{align}
Then we have
\bna
\item $\frakK_{q,\sfS}  \hspace{-.5ex} \seteq \hspace{-1.4ex} \soplus_{m \in \calM_+^\sfS} \hspace{-1ex} \Z[q^{\pm 1/2}] F_q(m) =\hspace{-1.4ex} \soplus_{m \in \calM_+^\sfS} \hspace{-1ex} \Z[q^{\pm 1/2}] E_q(m) = \hspace{-1.4ex} \soplus_{m \in \calM_+^\sfS} \hspace{-1ex} \Z[q^{\pm 1/2}] L_q(m) $, where $\calM^\sfS_+ \seteq \calM^\sfS \cap \calM_+$,
\item  $\frakK_{q,\sfS}$ is a subring of $\frakK_{q}$.
\ee
\end{proposition}

For a Dynkin quiver $Q$ of $\g$, let
$$\frakK_{q,Q} \seteq \frakK_{q,\Gamma_0^Q}.$$
By the above proposition, as a $\Z[q^{\pm 1/2}]$-algebra, $\frakK_{q,Q}$
is generated by $F_q(X_{i,p})$'s such that $(i,p) \in \Gamma^Q_0$. It is known that, for any $(j,s) \in\hDynkin_0$, there exists $k \in \Z$ and $(i,p) \in \Gamma_0^Q$ such that
\begin{align}\label{eq: heart}
\frakD^{k}(X_{i,p}) = X_{j,s}.
\end{align}
Since $\frakK_{q}$ is generated by
$$\{ F_q(X_{i,p}) \ | \  (i,p) \in \hDynkin_0  \} = \{ F_q( \frakD^k(X_{i,p})) \ | \  (i,p) \in \Gamma^Q_0, \; k \in \Z \},$$ we call  $\frakK_{q,Q}$ the \emph{heart subring of $\frakK_q$ associated to $Q$}.

Using the unitriangular maps between bases of $\frakK_q$ and the characterization of $\sfL_q$, we obtain the following lemma by applying the arguments in \cite[Lemma 3.11]{FHOO}:
\begin{lemma}  \label{lem: D_q}
For $m \in \calM_+$, we have
$$ \frakD_q(L_q(m)) = L_q(\frakD(m)).$$
\end{lemma}

\subsection{Quantum folded $T$-system among KR-polynomials} \label{subsec: T-sys KR poly}
For $(i,p) \in \hDynkin_0$ and $k \in \N  \seteq \Z_{\ge 1}$, we define a dominant monomial in  $\calX$ as follows:
$$
m_{k,p}^{(i)}   \seteq  m^{(i)}[p,s]\seteq X_{i,p}X_{i,p+2} \cdots X_{i,p+2k-2}, \qquad \text{ where } s \seteq p+2k-2.
$$
We call such monomials \emph{KR-monomials}. For a KR-monomial $m$, we call $F_q(m)$ the \emph{KR-polynomial} associated with $m$.

\smallskip

Note that there exists a functional relation among KR-polynomials, called \emph{quantum folded $T$-system}, as follows:
\begin{theorem}[\cite{HL15,JLO1}] \label{thm: quantum T-system via usf}
For $(i,p) \in \hDynkin_0$ with $k  \in \N$, we have
\begin{align}\label{eq: T-system via C(t)}
F_q(m_{k,p}^{(i)})*F_q(m_{k,p+2}^{(i)}) = q^{\kappa}  F_q(m_{k+1,p}^{(i)})*F_q(m_{k-1,p}^{(i)}) + q^{\zeta} \prod_{j \sim i} F_q(m_{k,p+1}^{(j)})^{-\sfc_{j,i}},
\end{align}
for some $\kappa,\zeta \in \frac{1}{2}\Z$.
\end{theorem}

We remark that the half-integers $\kappa$ and $\zeta$ could be determined explicitly (e.g.~see \cite[Theorem 6.8]{JLO1}).
The conjecture below for simply-laced type $\g$ is proved by Nakajima \cite{Nak04}:

\begin{conjecture}[{\cite[Conjecture 2]{JLO1}}]  \label{conj: JLO1 sencond} For every KR-polynomial $F_q(m)$, we have
$$  F_q(m) = L_q(m).$$
\end{conjecture}

\subsection{Truncation} Let $\xi$ be a height function on $\Dynkin$ and recall the convex subset $\lxi\hDynkin_0$ of $\hDynkin_0$. By Proposition~\ref{prop: convex subrings}, we have the subring
$\frakK_{q,\le \xi} \seteq \frakK_{q,\lxi\hDynkin_0}$. We denote by $\calX_{q,\le \xi}$ the quantum subtorus of $\calX_q$ generated by $\tX_{i,p}$ for $(i,p) \in \lxi\hDynkin_0$.

For an element $x \in \calX_q$, we denote by $x_{\le \xi}$ the element of $\calX_{q,\le \xi}$ obtained from $x$ by discarding all the monomials containing the factors $\tX_{i,p}^{\pm 1}$ with $(i,p) \in \hDynkin_0 \setminus \lxi\hDynkin_0$.
We define a $\Z[q^{\pm 1/2}]$-linear map $(\cdot)_{\le \xi}$ by
\begin{align*}
    (\cdot)_{\le \xi} :
    \xymatrix@R=0.5ex@C=5ex{
    \calX_q \ar@{->}[r] & \calX_{q,\le \xi}, \\
    x \ar@{|->}[r] & x_{\le \xi}.
    }
\end{align*}

\begin{proposition} $($\cite[Proposition 5.11]{FHOO},  \cite[Proposition 6.3]{JLO1}$)$ For a height function $\xi$ on $\Dynkin$, the  $\Z[q^{\pm 1/2}]$-linear map $(\cdot)_{\le \xi}$ induces \emph{the injective $\Z[q^{\pm 1/2}]$-algebra homomorphism}
$$  (\cdot)_{\le \xi} :  \frakK_{q,\le \xi} \hookrightarrow \calX_{q,\le \xi}.$$
\end{proposition}

\begin{definition} \label{def: Q quantum-torus}
For each Dynkin quiver $Q=(\Dynkin,\xi)$, we denote by $\calX_{q,Q}$ the quantum subtorus of $\calX_q$ generated by $\tX_{i,p}$ for $(i,p) \in \Gamma_0^Q$.
\end{definition}

\begin{corollary} Let $Q=(\Dynkin,\xi)$ be a Dynkin quiver.
The $\Z[q^{\pm 1/2}]$-linear  map $(\cdot)_{\le \xi}$ induces the injective $\Z[q^{\pm 1/2}]$-algebra homomorphism
$$  \frakK_{q,Q} \hookrightarrow \calX_{q,Q}. $$
\end{corollary}

\begin{proof}
Since $ \frakK_{q,Q}$ is generated by $F_q(X_{i,p})$ for $(i,p) \in \Gamma^Q_0$, it is enough to show that
$(F_q(X_{i,p}))_{\le \xi}$ is contained in $\calX_{q,Q}$. Then the assertion follows from Theorem~\ref{thm: F_q}~\eqref{eq: range}.
\end{proof}

\begin{remark} \label{rmk: characterization}
Through the injective homomorphism $( \cdot)_{\le \xi}$, the basis $\sfL_{q,Q}=\{ L_q(m) \ | \ m \in \calM_+^Q \}$ of $\frakK_{q,Q}$
transforms to the basis
\begin{align} \label{eq: L_q^T}
\text{$\sfL_{q,Q}^T \seteq \{ L_q(m)_{\le \xi} \seteq (L_q(m))_{\le \xi} | \ m \in \calM_+^Q \}$ of $ \frakK_{q,Q}^T \seteq (\frakK_{q,Q})_{\le \xi} \subset  \calX_{q,Q}$,}
\end{align}
 which can be characterized as follows:
\bnum
\item $\overline{L_q(m)_{\le \xi}} = L_q(m)_{\le \xi}$ and
\item $L_q(m)_{\le \xi} - E_q(m)_{\le \xi} \in \displaystyle\sum_{m' \lN m \in \calM_+^Q} q\Z[q]E_q(m')_{\le \xi}$,
where $E_q(m')_{\le \xi} \seteq (E_q(m'))_{\le \xi}$.
\ee
\end{remark}

\section{Unipotent quantum coordinate algebra } \label{Sec: Aqn}

In this section, we review the theory of the unipotent quantum coordinate algebra, especially their elements called unipotent quantum minors, its quantum cluster algebra structure,
the dual-canonical/upper-global basis and dual PBW bases.

\subsection{Unipotent quantum coordinate algebra} In this subsection, we briefly recall the quantum unipotent coordinate algebra. We mainly refer to \cite{KKKO18,HO19} and sometimes skip details available in the references.

For $n \ge m \in \Z_{\ge 1}$ and $i \in I$, we set
$$q_i \seteq q^{d_i}, \ \ [n]_i \seteq \dfrac{q_i^n-q_i^{-n}}{q_i - q_i^{-1}}, \ \ [n]_i! =\prod_{k=1}^n [k]_i \ \text{ and  } \
\left[ \begin{matrix} n \\ m \end{matrix} \right]_i \seteq \dfrac{[n]_i!}{[m]_i![n-m]_i!}.$$

We denote by $\calU_q(\g)$ the quantum group associated to a finite
Cartan datum $(\sfC,\wl,\Pi,\wl^\vee,\Pi^\vee)$, which is the
associative  algebra over $\mathbb Q(q^{1/2})$ generated by $e_i,f_i$ $(i
\in I)$ and $q^{h}$ ($h\in\wl^\vee$). We set $\calU_q^+(\g)$ (resp. $\calU_q^-(\g)$) be
the subalgebra generated by $e_i$ (resp. $f_i$) for $i \in I$.  Note
that  $e_i$'s (resp. $f_i$'s) satisfy the quantum Serre relations
\begin{align}
\sum_{k=0}^{1-\sfc_{i,j}} (-1)^k \left[ \begin{matrix} 1-\sfc_{i,j} \\ k \end{matrix} \right]_i f_i^{1-\sfc_{i,j}-k} f_j f_i^{k} =
\sum_{k=0}^{1-\sfc_{i,j}} (-1)^k \left[ \begin{matrix} 1-\sfc_{i,j} \\ k \end{matrix} \right]_i e_i^{1-\sfc_{i,j}-k} e_j e_i^{k} =0,
\end{align}
and
$\calU_q(\g)$ admits the weight space decomposition:
$$  \calU_q(\g) = \soplus_{\be \in \rl} \calU_q(\g)_\be. $$

For any $i \in I$, there exist $\Q(q^{1/2})$-linear endomorphisms ${}_ie'$ and $e_ i'$ of $\calU_q^-(\g)$ uniquely determined by
\begin{equation}\label{eq: deri}
\begin{aligned}
& {}_ ie'(f_j)=\delta_{i,j}, && {}_ ie'(xy) = {}_ ie'(x) \; y+   q^{(\al_i,\be)}x \;  {}_ ie' (y), \\
& e_ i'(f_j)=\delta_{i,j}, && e_i'(xy) = q^{(\al_i,\eta)}e_i'(x) \; y+    x \; e'_i(y),
\end{aligned}
\end{equation}
where $x \in \calU_q^-(\g)_\be, \  y \in \calU_q^-(\g)_\eta$.
Then it is known \cite{K91} that there exists a unique non-degenerate symmetric $\Q(q^{1/2})$-bilinear form $( , )_K$ on $\calU_q^-(\g)$ such that
$$
(  \mathbf{1},\mathbf{1})_K=1, \quad (f_iu,v)_K = (u,{}_ie'v)_K \text{ for $i \in I$, $u,v \in \calU_q^-(\g)$}.
$$

We set $\bbA=\Z[q^{\pm 1/2}]$ and let
$\calU_{\bbA}^{+}(\g)$ be the $\bbA$-subalgebra of
$\calU_q^{+}(\g)$ generated by $e_i^{(n)}\seteq e_i^n/[n]_i!$
for $i \in I$ and $n\in\Z_{>0}$.

Let $\Delta_\n$ be the algebra homomorphism $\calU_q^+(\g) \to \calU_q^+(\g) \tens \calU_q^+(\g)$ given by  $ \Delta_\n(e_i) = e_i
\tens 1 + 1 \tens e_i$,
where the algebra structure on $\calU_q^+(\g)
\tens \calU_q^+(\g)$ is defined by
$$(x_1 \tens x_2) \cdot (y_1 \tens y_2) = q^{-(\wt(x_2),\wt(y_1))}(x_1y_1 \tens x_2y_2).$$

Set
$$ \calA_q(\n) = \soplus_{\beta \in  \rl^-} \calA_q(\n)_\beta \quad \text{ where } \calA_q(\n)_\beta \seteq \Hom_{\Q(q^{1/2})}(\calU^+_q(\g)_{-\beta}, \Q(q^{1/2})).$$
Then  $\calA_q(\n)$ is an algebra with the multiplication given by
$(\psi \cdot \theta)(x)= \theta(x_{(1)})\psi(x_{(2)})$, when
$\Delta_\n(x)=x_{(1)} \tens x_{(2)}$ in  Sweedler's notation.
Let us denote by $\calA_\bbA(\n)$ the $\bbA$-submodule
of $\calA_q(\n)$
 consisting of $ \uppsi \in \calA_q(\n)$ such that
$ \uppsi \left( \calU_q^{+}(\g)_{\bbA} \right) \subset\bbA$. Then
it is an $\bbA$-subalgebra of $\calA_q(\n)$.

It is known  \cite{GLS13, K91, LusztigBook, Kimura12} that
\begin{equation} \label{eq:An=Uq- and An=CN at q=1}
 \calA_{\bbA}(\n) \simeq \calU_{\bbA}^-(\g)  \quad \text{ and } \quad
\C \otimes_{\Z[q^{\pm 1/2}]} \calA_{\bbA}(\n) \simeq \C[N],
\end{equation}
where $\C[N]$ denotes the coordinate (commutative) algebra of the complex unipotent algebraic group $N$ whose Lie algebra is $\n$.  Thus we have a natural algebra (surjective) homomorphism $\ev_{q=1}: \calA_{\bbA}(\n) \to \C[N]$.
We call $\calA_{q}(\n)$ the \emph{unipotent quantum coordinate algebra}.

For a dominant weight $\la \in \wl^+$, let $V(\la)$ be the
irreducible highest weight $\calU_q(\g)$-module with highest weight
vector $u_\la$ of weight $\la$. Let $( \ , \ )_\la$ be the
non-degenerate symmetric bilinear form on $V(\la)$ such that
$(u_\la,u_\la)_\la=1$ and $(xu,v)_\la=(u,\upvarphi(x)v)_\la$ for $u,v
\in V(\la)$ and $x\in \calU_q(\g)$, where $\upvarphi$ is  the algebra
antiautomorphism on $\calU_q(\g)$  defined by $\upvarphi(e_i)=f_i$,
$\upvarphi(f_i)=e_i$  and $\upvarphi(q^h)=q^h$. For each $\mu, \zeta \in
\sfW_\g \la$, the \emph{unipotent quantum minor} $D(\mu, \zeta)$ is
an element in $\calA_\bbA(\n)$ given by
 $(D(\mu,\zeta),x)_{K}=(x u_\mu, u_\zeta)_\la$
for $x\in \calU_q^{+}(\g)_{\bbA}$, where $u_\mu$ and $u_{\zeta}$ are the
extremal weight vectors in $V(\la)$ of weight $\mu$ and $\zeta$,
respectively.
Note that
$\wt( D(\mu,\zeta)) = \mu -\zeta$ if $\mu-\zeta \in \rl^-$ and $D(\mu, \zeta)$ vanishes otherwise.

\begin{definition} We set
\begin{align} \label{eq: tD}
\tD(\mu,\zeta) \seteq q^{-(\mu-\zeta,\mu-\zeta)/4+(\mu-\zeta,\rho)/2} D(\mu,\zeta)
\end{align}
and call it a \emph{normalized quantum unipotent minor}.
\end{definition}

\begin{proposition}[\cite{BZ05,GLS13,KKKO18}] \hfill
\ben
\item For $\la,\mu \in \wl^+$ and $s,t,s',t' \in \weyl$ such that
\bnum
\item $\ell(ss')=\ell(s)+\ell(s')$ and $\ell(tt')=\ell(t)+\ell(t')$,
\item  $s's\la \preceq t'\la$ and $s'\mu \preceq t't\mu $,
\ee
we have $\theta \in \tfrac{1}{2}\Z$ such that
$$  \tD(s'\mu,t't\mu)  \tD(s's\la,t'\la) = q^{\theta}\tD(s's\la,t'\la) D(s'\mu,t't\mu).  $$
\item For $u,v \in \weyl$ and $i \in I$ satisfying $u < us_i$ and $v < vs_i$, we have $\tka,\tze \in \tfrac{1}{2}\Z$ such that
\begin{equation} \label{eq: determinatial id}
\begin{aligned}
& \tD(u\varpi_i,v\varpi_i)  \tD(us_i\varpi_i,vs_i\varpi_i)  = q^{\tka} \tD(us_i\varpi_i,v\varpi_i)\tD(u\varpi_i,vs_i\varpi_i)        + q^{\tze} \prod_{j \sim i}^{\to} \tD(u\varpi_j, v\varpi_j)^{-\sfc_{j,i}}.
\end{aligned}
\end{equation}
\ee
\end{proposition}
\noindent
We call~\eqref{eq: determinatial id} the \emph{determinantal identity} among unipotent quantum minors.

\subsection{Quantum cluster algebra structure}  In this subsection, we briefly recall the quantum cluster algebra structure of  $\calA_\bbA(\n)$ investigated in \cite{GLS13,GY17}.

Let $r \in \N \sqcup \{\infty\}$.
For a sequence of indices $\ii=(i_1,i_2,\ldots,i_r) \in I^r$, $1 \le k \le r$  and $j \in I$, we use the following notations:
\begin{equation}
\begin{aligned}
w^\ii_{\le k} &\seteq s_{i_1}\cdots s_{i_k}, \quad \la^\ii_k \seteq w^\ii_{\le k}\varpi_{i_k},  \quad w^\ii_{\le 0}=\mathrm{id}  ,\\
k_\ii^+(j)&\seteq \min(\{ k<u \le r \ | \ i_u=j \} \cup \{r+1\}), \\
k_\ii^-(j)&\seteq \max(\{ 1\le u < k \ | \ i_u=j \} \cup \{0\}),  \\
k_\ii^\pm &\seteq k_\ii^\pm(j) \ \  \text{ if } i_k =j.
\end{aligned}
\end{equation}
If there is no danger of confusion, we drop the scripts $_\ii$ or $^\ii$ for notational simplicity.

\smallskip

Let $\ii = (i_u)_{u \in \N}$ be a sequence of $I$ satisfying the following condition:
\begin{align} \label{eq: infinity condition}
\text{ for any $i \in I$, we have $|\{   u \in \N \ | \ i_u =i \}| = \infty.$}
\end{align}
For instance, when we take $\ii_\circ = (i_1,\ldots,i_\ell) \in I(w_\circ)$, we can  extend it to $\widetilde{\ii}_\circ \in I^\N$ satisfying
~\eqref{eq: infinity condition} as follows:
\begin{align} \label{eq: extension}
  i_{s+\ell} = i^*_{s}  \text{ for } s \in \N.
\end{align}
In particular, in this case, one can easily see that any successive subsequence of length $\le \ell$ of $\tii_\circ$ is reduced.

For $\ii = (i_u)_{u \in \N}$ satisfying~\eqref{eq: infinity condition}, define an $\N \times \N$-matrix $\tB_{\ii}=(b_{ab})_{a,b\in \N}$ and a symmetric $\N \times \N$-matrix $\La_{\ii}=(\la_{st})_{s,t\in \N}$ as follows:
\begin{align}
& b_{ab} = \bc
\pm 1 & \text{ if } b=a^\pm, \\
\sfc_{i_a,i_b}& \text{ if }  a<b<a^+<b^+, \\
-\sfc_{i_a,i_b}& \text{ if }  b<a<b^+<a^+, \\
0 & \text{ otherwise},
\ec \ \
\text{ and }  \ \
\la_{st}= (\varpi_{i_s}-w_{\le s}\varpi_{i_s}, \varpi_{i_t}+w_{\le t}\varpi_{i_t}  ) \ \  \text{ for } s<t.
\end{align}

For $n \in \N$, we set
$$  \sfK \seteq \{ k \in \Z \ | \ 1 \le k \le n\}, \quad \sfK_\fr = \{ k \in \sfK \ | \ k^+ >n \},   \quad \sfK_\ex \seteq \sfK \setminus \sfK_\fr,$$
and define $\tB^n_\ii \seteq (b_{u,v})_{u \in \sfK, v \in \sfK_\ex}$ and $\La^n_\ii \seteq (\la_{u,v})_{u ,v\in \sfK}$.
We understand $\tB^\infty_\ii=\tB_\ii$ and $\La^\infty_\ii=\La_{\ii}$ with $\sfK_{\fr}=\emptyset$.

\begin{definition}
Let $\sfK$ be a countable index set with a decomposition $\sfK=\sfK_\ex \sqcup \sfK_\fr$. The indices in $\sfK_\ex$ (resp,  $\sfK_\fr$) are called \emph{exchangeable} (resp. \emph{frozen}) indices.
A $\Z$-valued $\sfK \times \sfK_\ex$-matrix $\tB=(b_{i,j})_{i \in \sfK,j\in \sfK_\ex}$ is called an \emph{exchange matrix} if
\bna
\item for each $j \in \sfK_\ex$, there are finite many $i \in \sfK$ such that $b_{i,j} \ne 0$,
\item its $\sfK_\ex \times \sfK_\ex$-submatrix $B \seteq (b_{i,j})_{i,j \in \sfK_\ex}$ is \emph{skew-symmetrizable}; i.e., there exists a sequence $S =(t_i \ | \ i \in \sfK_\ex, \ t_i \in \Z_{\ge 1})$ such that $t_i b_{i,j}=t_j b_{j,i}$ for all $i,j \in \sfK_\ex$.
\ee
\end{definition}

\begin{proposition}[\cite{FHOO2}] For any $n \in \N \sqcup \{ \infty \}$, the matrix  $\tB^n_\ii$ is an exchange matrix and
$$    \sum_{k \in \sfK} b_{k,u}\La_{k,v}  = 2d_{i_u}\delta_{u,v}  \quad \text{ for } u \in \sfK_\ex, \  v \in \sfK.$$
In this case, the pair $(\La^n_\ii,\tB^n_\ii)$ is called compatible.
\end{proposition}

Recall the notation  $\be^{\ii_\circ}_{k} = w^{\ii_\circ}_{\le k-1}(\al_{i_k})$  for $\ii_\circ \in I(w_\circ)$ and $1 \le k \le \ell$.  For $ \be  \in \Phi_+$ with $\be^{\ii_\circ}_k = \be$ $(1 \le k \le \ell)$, we define
\begin{align*}
\be^+_{\ii_\circ} &\seteq \bc \be^{\ii_\circ}_{k_{\ii_\circ}^+} & \text{ if } k_{\ii_\circ}^+ \le \ell, \\ 0  & \text{ if } k_{\ii_\circ}^+= \ell+1, \ec \qquad
\be^-_{\ii_\circ} \seteq  \bc \be^{\ii_\circ}_{k_{\ii_\circ}^-}  & \text{ if } k_{\ii_\circ}^->0, \\ 0 & \text{ if } k_{\ii_\circ}^-=0,  \ec \qquad
\la^{\ii_\circ}_{\be} \seteq w^{\ii_\circ}_{\le k}\varpi_{i_k}.
\end{align*}
We also define   $\be^{(\pm 2)}_{\ii_\circ}  =  ( \be^\pm_{\ii_\circ} )^\pm_{\ii_\circ} $, $\be^{(\pm 3)}_{\ii_\circ}  = ( ( \be^\pm_{\ii_\circ} )^\pm_{\ii_\circ})^\pm_{\ii_\circ} $, $\dots$; that is, $\be^{(\pm k)}_{\ii_\circ}=
\hspace{-3ex}\underbrace{( \cdots (}_{\text{ $(k-1)$-braces}}\mkern-10mu
  \hspace{-2.2ex} \be^\pm_{\ii_\circ} )^\pm_{\ii_\circ} \cdots )^\pm_{\ii_\circ} $.
If there is no danger of confusion, we also drop the scripts $_{\ii_\circ}$ or $^{\ii_\circ}$ for notational simplicity.

\begin{remark} For $\ii_\circ \in I(w_\circ)$ and its extension $\ii$ via~\eqref{eq: extension}, let us set $(\La_{\ii_\circ},\tB_{\ii_\circ}) \seteq (\La^{\ell}_{\ii},\tB^{\ell}_{\ii})$.
As $\res^{[\ii_\circ]}(\be)$ $(\be \in \Phi_+)$ is well-defined, the assignment $\be\mapsto
\la_\be$ does depend only on $[\ii_\circ]$. Thus, for $\ii'_\circ\overset{c}{\sim}\ii_\circ$,
we can obtain  $(\La_{\ii'_\circ},\tB_{\ii'_\circ})$ by permuting the indices of $(\La_{\ii_\circ},\tB_{\ii_\circ})$. Hence we can define a compatible pair $(\La_{[\ii_\circ]},\tB_{[\ii_\circ]})$ as follows:
Set $\sfK=\Phi_+$, $\sfK_\fr = \{ \al \in \Phi_+ \ | \  \al^+ = 0 \}$ and   $\sfK_\ex =  \sfK \setminus  \sfK_\fr$. Define an $\sfK\times \sfK_{\ex}$-matrix $\tB_{[\ii_\circ]}=(b_{\al\be})_{\al \in \sfK,\be \in\sfK_\ex}$ and
a symmetric $\sfK\times \sfK$-matrix $\La_{[\ii_\circ]}=(\la_{\eta\theta})_{\eta,\theta \in \sfK}$ by setting
\begin{align*}
& b_{\al\be} = \bc
\pm 1 & \text{ if }\be=\al^\pm, \\
\sfc_{i_a,i_b}& \text{ if }  \al \prec_{[\ii_\circ]} \be  \prec_{[\ii_\circ]}  \al^+  \prec_{[\ii_\circ]}  \be^+ \text{ and } i_a \sim i_b, \\
-\sfc_{i_a,i_b}& \text{ if }  \be \prec_{[\ii_\circ]} \al \prec_{[\ii_\circ]} \be^+ \prec_{[\ii_\circ]} \al^+ \text{ and } i_a \sim i_b, \\
0 & \text{ otherwise},
\ec \quad \text{ where $\al=\be^{\ii_\circ}_a$ and $\be=\be^{\ii_\circ}_b$,}
\end{align*}
and
\begin{align*}
\Lambda_{\eta\theta}= (\varpi_{i}-\la_\eta, \varpi_{j}+\la_\theta  )
\quad \text{ for } \ \theta \not\preceq_{[\ii_\circ]} \eta \   \text{ with }  \ i=\res^{[\ii_\circ]}(\eta) \text{ and } j=\res^{[\ii_\circ]}(\theta)
\end{align*}
satisfying
$$    \sum_{\ga \in \sfK}  b_{\ga\al}\La_{\ga\be}=2d_\al\delta_{\al,\be} \quad \text{ for all $\al \in \sfK_\ex$ and $\be \in \sfK$.}$$
\end{remark}

We define the quantum torus $(\calT_{q,[\ii_\circ]},\star)$ as the $\Z[q^{\pm1/2}]$-algebra generated by the set of generators $\{ \tY^{\pm 1}_\be \ | \ \be \in \Phi_+\}$ satisfying the relations
\begin{align} \label{eq: Y torus1}
\tY_\be \star \tY_\be^{-1} = \tY_\be^{-1} \star \tY_\be =1 \quad \text{ and } \quad \tY_\al  \star \tY_\be = q^{\La_{\al\be}} \tY_\be \star \tY_\al.
\end{align}
We define $\wt(\tY_\al) \seteq -\sum_{k \in \Z_{\ge 0}}\al^{(-k)}$.
Note that there exists a $\Z$-algebra anti-involution on $\calT_{q,[\ii_\circ]}$ sending
\begin{align} \label{eq: bar-inv on Tq}
  q^{\pm 1/2}\longmapsto  q^{\mp 1/2} \quad \text{ and } \quad  \tY_\al \longmapsto \tY_\al \quad \text{ for all }\al \in \Phi_+
\end{align}
(cf.~\eqref{eq:bar-inv on Xq}).

Let $\scrA(\La_{[\ii_\circ]},\tB_{[\ii_\circ]})$ be the quantum cluster algebra associated with $(\La_{[\ii_\circ]},\tB_{[\ii_\circ]})$. It is a $\Z[q^{\pm 1/2}]$-subalgebra of $\calT_{q,[\ii_\circ]}$
generated by quantum cluster variables (we refer to~\cite[Appendix A]{FHOO2} for the quantum cluster algebra and related notions).

For $\al,\be \in \Phi_+$ with $\res^{[\ii_\circ]}(\al)=\res^{[\ii_\circ]}(\be)= i$, we write
$$
\tD_{[\ii_\circ]}(\al,\be) \seteq \tD(\la^{\ii_\circ}_\al,\la^{\ii_\circ}_\be), \quad \tD_{[\ii_\circ]}(\al,0) \seteq \tD(\la^{\ii_\circ}_\al,\varpi_i).
$$
We understand $\tD_{[\ii_\circ]}(\al,\be)=0$ unless $\al \prec_{[\ii_\circ]} \be$ and  $\res^{[\ii_\circ]}(\al) = \res^{[\ii_\circ]}(\be)$,
and set  $\tD_{[\ii_\circ]}(0,0)=1$.

\smallskip

The following theorem is proved in \cite{GLS13,KKKO18} for simply-laced types and in \cite{GY17,Qin20} for symmetrizable types.

\begin{theorem}[\cite{GLS13,GY17,KKKO18,Qin20}]\label{thm:UQCR isom to CA}
For a reduced sequence  $\ii_\circ$ of $w_\circ$, there exists an $\bbA$-algebra isomorphism
$${\rm CL}:  \scrA(\La_{[\ii_\circ]},\tB_{[\ii_\circ]}) \to   \calA_{\bbA}(\n)$$
given by
$$\tY_\al \longmapsto\tD_{[\ii_\circ]}(\al,0) \quad \text{ for  } \al\in\Phi_+.$$
\end{theorem}

\subsection{Dual PBW and dual-canonical/upper-global bases} Throughout this subsection, we fix $[\ii_\circ]$ of $w_\circ$.  For $\be \in \Phi_+$, we set
$$
F^{\up}_{[\ii_\circ]}(\al) \seteq D_{[\ii_\circ]}(\al,\al^-)  \quad \left(\text{resp. } \tF^{\up}_{[\ii_\circ]}(\al) \seteq \tD_{[\ii_\circ]}(\al,\al^-)  \right)
$$
and call $$\left\{ \left. \tF^{\up}_{[\ii_\circ]}(\al) \ \right| \  \al \in \Phi_+\right\} \quad
\left( \text{resp. } \left\{ \left. \tF^{\up}_{[\ii_\circ]}(\al) \ \right| \  \al \in \Phi_+\right\}\right)$$ \emph{the set of $($resp.~normalized$)$ dual PBW-vectors associated with $[\ii_\circ]$} (cf.~\cite[Proposition 7.4]{GLS13}).

For an exponent $\ue = (\ek_\be)_{\be \in \Phi_+} \in \Z_{\ge 0}^{\Phi_+}$, we set
$$
\tF^\up_{[\ii_\circ]} (\ue) \seteq q^{ -\sum_{  k<l }  \ek_{\be_k}\ek_{\be_l}(\be_k,\be_l)}  \tF^{\up}_{[\ii_\circ]}(\be_\ell)^{\ek_{\be_\ell}} \cdots \tF^{\up}_{[\ii_\circ]}(\be_2)^{\ek_{\be_2}}\tF^{\up}_{[\ii_\circ]}(\be_1)^{\ek_{\be_1}},
$$
where $\be_k \seteq \be_k^{\ii_\circ}$.

\begin{proposition} $($See, for instance, \cite[Chapter 40, 41]{LusztigBook}$.)$ \label{prop: PBW} The set  $\{ \tF^\up_{[\ii_\circ]}(\ue) \ | \ \ue \in \Z_{\ge 0}^{\Phi_+} \}$ forms a $\Z[q^{\pm 1/2}]$-basis of $\calA_{\bbA}(\n)$.
\end{proposition}
The basis  $\{ \tF^\up_{[\ii_\circ]}(\ue) \ | \ \ue \in \Z_{\ge 0}^\ell \}$ is referred to as a \emph{normalized dual Poincar\'e--Birkhoff--Witt type basis}.

\begin{theorem} [\cite{BKM12}] \label{thm: minimal pair dual pbw}
For an $[\ii_\circ]$-minimal $[\ii_\circ]$-pair $\uup=\pair{\al,\be}$ for
$\ga = \al+\be \in \Phi_+ \setminus \Pi$ $($defined in Section \ref{subsec: stat}$)$, we have
\begin{align}
& F^\up_{[\ii_\circ]}(\al)F^\up_{[\ii_\circ]}(\be) - q^{-(\al,\be)}F^\up_{[\ii_\circ]}(\be)F^\up_{[\ii_\circ]}(\al)
= q^{-p_{\be,\al}}(1-q^{2(p_{\be,\al}-(\al,\be))} )F^\up_{[\ii_\circ]}(\ga)  \label{eq: BKMc1} \\
& \iff \tF^\up_{[\ii_\circ]}(\al)\tF^\up_{[\ii_\circ]}(\be) - q^{-(\al,\be)}\tF^\up_{[\ii_\circ]}(\be)\tF^\up_{[\ii_\circ]}(\al)
= q^{-p_{\be,\al}+(\al,\be)/2 }(1-q^{2(p_{\be,\al}-(\al,\be))} )  \tF^\up_{[\ii_\circ]}(\ga).  \label{eq: BKMc2}
\end{align}
In particular, when $p_{\be,\al}=0$, we have
\begin{align}\label{eq: BKMcp}
\dfrac{ q^{-(\al,\be)/2}\tF^\up_{[\ii_\circ]}(\be)\tF^\up_{[\ii_\circ]}(\al) -q^{(\al,\be)/2 }\tF^\up_{[\ii_\circ]}(\al)\tF^\up_{[\ii_\circ]}(\be)}{(q^{-(\al,\be)}-q^{(\al,\be)} )}
=   \tF^\up_{[\ii_\circ]}(\ga).
\end{align}
\end{theorem}

  For a positive root $\al$, recall $d_\al = (\al,\al)/2 \in \N$.
Let  $\al,\be$ be positive roots such that $\ga=\al+\be \in \Phi_+$. Then we have
\begin{align} \label{eq: palbe}
p_{\be,\al} = \bc
2 & \text{ if } d_\ga=3 \text{ and } d_{\al}= d_{\be} =1,\\
1 & \text{ if } d_\ga=2 \text{ and } d_{\al}= d_{\be} =1,\\
1 & \text{ if $\g$ is of type $G_2$ and }  d_\al=d_{\be}= d_{\ga} =1,\\
0 & \text{otherwise.}
\ec
\end{align}

Note that there exists the $\Z$-algebra anti-involution, denoted also by $\overline{( \cdot )}$, on $\calA_{\bbA}(\n)$ defined by
$$  q^{\pm 1/2}\longmapsto  q^{\mp 1/2} \quad \text{ and } \quad  \tF^{\up}_{[\ii_\circ]}(\al) \longmapsto \tF^{\up}_{[\ii_\circ]}(\al) \quad \text{ for all }\al \in \Phi_+,$$
which satisfies (see  \cite[(8.3)]{HO19} and \cite[\S 7.4]{FHOO})
\begin{align}\label{eq: bar-compatibility}
   \overline{( \cdot )} \circ {\rm CL}  = {\rm CL} \circ \overline{( \cdot )}.
\end{align}

Lusztig \cite{L90,L91} and Kashiwara \cite{K91} have constructed a specific $\Z[q^{\pm 1/2}]$-basis  $\bfB^\up$  of $\calA_{\bbA}(\n)$, which is called the
\emph{dual-canonical/upper-global basis}.
In this paper, we consider the normalized one $\tbfB^\up$ of $\bfB^\up$.
Note that $\tD(\mu,\zeta) \in \tbfB^\up$  for $\mu \preceq \zeta$.

\begin{theorem}[{\cite[Theorem 4.29]{Kimura12}, see also \cite[Proposition 4.1.1]{KKK2}}]  \label{thm: dual can} For each $[\ii_\circ]$ of $w_\circ$, we have
$$
\tbfB^\up = \{ \tG_{[\ii_\circ]}(\ue) \ | \  \ue \in \Z_{\ge 0}^{\Phi_+} \},
$$
where $\tG_{[\ii_\circ]}(\ue)$ is the unique element of $\calA_{\bbA}(\n)$ satisfying the following properties for each $\ue \in \Z_{\ge 0}^{\Phi_+}$:
\begin{eqnarray} &&
\parbox{85ex}{
\bna
\item $\overline{\tG_{[\ii_\circ]}(\ue)} = \tG_{[\ii_\circ]}(\ue)$,
\item \label{it: tGc}$\tG_{[\ii_\circ]}(\ue) - \tF^\up_{[\ii_\circ]}(\ue) \in \sum_{\ue' \widetilde{<}_{[\ii_\circ]} \ue}  q\Z[q] \tF^\up_{[\ii_\circ]}(\ue')$.
\ee
}\label{eq: dual canonical}
\end{eqnarray}
Here $\ue'=(\ek'_{\al})_{\al \in \Phi_+} \widetilde{<}_{[\ii_\circ]} \ue=(\ek_{\al})_{\al \in \Phi_+}$ means that there exists $\be \in \Phi^+$ such that $\ek'_{\be}< \ek_{\be}$ and $\ek'_{\al}=\ek_{\al}$ for all $\be \not\preceq_{[\ii_\circ]} \al$.
\end{theorem}

\section{HLO isomorphisms} \label{Sec: Iso}
In this section, we shall prove that there exists an algebra isomorphism from $\calA_\bbA(\n)$ to $ \frakK_{q,Q}$
sending $\tbfB^\up$ to $\sfL_{q,Q}$ for a Dynkin quiver $Q$ of type $\g$. This isomorphism is sometimes referred to as the \emph{HLO isomorphism} since it is established by Hernandez--Leclerc--Oya in \cite{HL15,HO19} (see also \cite{FHOO,FHOO2}).

\smallskip

Recall the quantum torus $\calX_{q,Q}$ associated with a Dynkin quiver $Q$ in Definition~\ref{def: Q quantum-torus}.

\begin{theorem} $($\cite[Proposition 3.8]{HL15}, \cite[Theorem 6.14]{KO22}$)$ \label{thm: tori iso} For each Dynkin quiver $Q=(\Dynkin,\xi)$, there exists an isomorphism between quantum tori
\begin{align} \label{eq: Uppsi}
\begin{split}
\Uppsi_Q :
\xymatrix@R=1ex@C=9ex{
    \calT_{q,[Q]} \ar@{->}[r] & \calX_{q,Q} \\
    \tY_{\al} \ar@{|->}[r] & \underline{m^{(i)}[p,\xi_i]}
}
\end{split}
\end{align}
for each $\al \in \Phi_+$ with $\phi_Q(i,p)=(\al,0)$.
\end{theorem}

Recall Remark \ref{rmk: characterization} and that
$$\calA_\bbA(\n) \subset \calT_{q,[Q]} \quad \text{ and } \quad \frakK_{q,Q} \simeq \frakK_{q,Q}^T \subset  \calX_{q,Q}.$$

\smallskip

The theorem below for simply-laced $\g$ is proved in \cite{HL15} (see also \cite{FHOO}) and can be proved in our case through an argument similar to that of
\cite[Theorem 8.6]{FHOO}. Here we repeat the proof in order to show how to use the quantum folded $T$-system.

\begin{theorem} \label{thm:main1}
The isomorphism $\Uppsi_Q$ in~\eqref{eq: Uppsi} induces an isomorphism between $\calA_\bbA(\n) $ and $\frakK_{q,Q}^T$, which sends the  $\Z[q^{\pm 1/2}]$-basis
$\tbfB^\up $ of $\calA_{\bbA}(\n)$ to the $\Z[q^{\pm 1/2}]$-basis $\sfL_{q,Q}^T$ of $\frakK_{q,Q}^T$. More precisely, for each $\ue \in \seq{\ek_\al}_{\al \in \Phi_+} \in \Z_{\ge 0}^{\Phi^+}$, we have
$$ \Uppsi_Q\left( \tG_{[Q]}(\ue) \right)  = L_q(m(\ue))_{\le \xi}, $$
where
$m(\ue) \in \calM_+^Q$ is given by
$$
u_{i,p}(m(\ue)) = \ek_\al \quad \text{ when } \phi_Q(i,p)=(\al,0) \text{ for each $(i,p) \in \Gamma_0^Q$.}
$$
\end{theorem}

\begin{proof}
Let
$$
\calS_Q \seteq \{ (i; p,u) \in I \times \Z^2 \ | \  (i,p), (i,u) \in \hDynkin_0, \ \xi_{i^*}- \sfh < p \le u \le \xi_i+2 \}
$$
and choose a total ordering $\prec$ on $\calS_Q$ satisfying
$$
(i;p,u)  \prec (i';p',u') \quad \text{ if either {\rm (i)} }  u>u', \text{ or  {\rm (ii)} } u=u' \text{ and } p>p'.
$$

For each $(i;p,u) \in \calS_Q$, we set
\begin{align} \label{eq:quantum minors to Q}
\tD(i;p,u) \seteq \bc
\tD_{[Q]}(\al,\be) & \text{ if } u \le \xi_i, \\
\tD_{[Q]}(\al,0) & \text{ if } u = \xi_i+2, \\
1& \text{ if } p=u,
\ec
\end{align}
where $\phi_Q(i,p)=(\al,0)$ and  $\phi_Q(i,u)=(\be,0)$.
Note that ~\eqref{eq: determinatial id} can be expressed as
\begin{align}\label{eq: tD ips det}
\tD(i;p,u)\tD(i;p^+,u^+)= q^{\tka} \tD(i;p,u^+)\tD(i;p^+,u) + q^{\tze} \prod_{j \sim i} \tD(j; p^+(j),u^+(j))^{-\sfc_{j,i}},
\end{align}
where $t^\pm = t \pm 2$ and $t^+(j) = \min\{ t' \in \Z \ | \ (j,t') \in \hDynkin_0, \; t <t' \}$ for $t \in \Z$.

To show $\Uppsi_Q(\calA_\bbA(\n)) \simeq \frakK_{q,Q}^T$, it is enough to show that
\begin{align} \label{eq: D to F}
\Uppsi_Q(\tD(i;p,u)) = F_q(m^{(i)}[p,u-2])_{\le \xi}
\end{align}
for each $(i;p,u) \in \calS_Q$, by Proposition~\ref{prop: PBW}. Let us apply an induction with respect to $\prec$ on $\calS_Q$. When $u = \xi_i+2$, we have
\begin{align*}
\Uppsi_Q(\tD(i;p,u)) = \Uppsi_Q(\tD_{[Q]}(\al,0)) & =  \Uppsi_Q(\tY_\al)
 = \underline{m^{(i)}[p,\xi_i]}  = F_q(m^{(i)}[p,\xi_i])_{\le \xi}.
\end{align*}
Since the case $p=u$ is trivial, we consider the case  when $p<u \le \xi_i$.

Denoting by $\bbF(\calT_{q,[Q]})$ (resp. $\bbF(\calX_{q,Q})$) the skew-field of fractions of $\calT_{q,[Q]}$ (resp. $\calX_{q,Q}$), one can extend the $\Z$-algebra anti-involution $\overline{( \cdot )}$ of $\calT_{q,[Q]}$ (resp. $\calX_{q,Q}$) to the
$\Q$-algebra anti-involution on $\bbF(\calT_{q,[Q]})$ (resp. $\bbF(\calX_{q,Q})$). Then we have $\Q$-algebra isomorphism $\widetilde{\Uppsi}_Q:  \bbF(\calT_{q,[Q]}) \to \bbF(\calX_{q,Q})$ satisfying
$$
\widetilde{\Uppsi}_Q \circ \overline{( \cdot )} = \overline{( \cdot )} \circ \widetilde{\Uppsi}_Q.
$$
From~\eqref{eq: tD ips det}, $\tD(i;p,u)$ is an element
in $\bbF(\calT_{q,[Q]})$ that is $\overline{( \cdot )}$-invariant and of the form
\begin{align} \label{eq: main step 1}
\dfrac{  q^{\tka} \tD(i;p,u^+)\tD(i;p^+,u) + q^{\tze} \prod_{j \sim i} \tD(j; p^+(j),u^+(j))^{-\sfc_{j,i}}}{ \tD(i;p^+,u^+) }.
\end{align}

By applying the induction hypothesis, we see that
\begin{align*}
& \widetilde{\Uppsi}_Q(\tD(i;p,u)) \\
& \ \ = \dfrac{\left(  q^{\tka} F_q(m^{(i)}[p,u])_{\le \xi}* F_q(m^{(i)}[p+2,u-2])_{\le \xi} + q^{\tze} \prod_{j \sim i} F_q(m^{(j)}[p+1,u+1])_{\le \xi}^{-\sfc_{j,i}}\right)}{F_q(m^{(i)}[p+2,u])_{\le \xi}}
\end{align*}
is a $\overline{( \cdot )}$-invariant element. On the other hand, Theorem~\ref{thm: quantum T-system via usf} tells us that
$F_q(m^{(i)}[p,u-2])_{\le \xi}$ is the $\overline{( \cdot )}$-invariant element of the form
\begin{align*}
\dfrac{\left(  q^{\kappa} F_q(m^{(i)}[p,u])_{\le \xi}* F_q(m^{(i)}[p+2,u-2])_{\le \xi} + q^{\zeta} \prod_{j \sim i} F_q(m^{(j)}[p+1,u+1])_{\le \xi}^{-\sfc_{j,i}}\right)}{F_q(m^{(i)}[p+2,u])_{\le \xi}},
\end{align*}
obtained from applying $( \cdot)_{\le \xi}$ on~\eqref{eq: T-system via C(t)}.
Since the $\overline{( \cdot )}$-invariance determines the values of $\kappa$ and $\xi$ uniquely,
the first assertion follows.

Now let us consider the second assertion. Since $\Uppsi_Q$ commutes with $\overline{(\cdot)}$, the element $\Uppsi(\tG_{[Q]}(\ue))$ is $\overline{(\cdot)}$-invariant. By Theorem~\ref{thm: L_q} and Theorem~\ref{thm: dual can} (see also Remark~\ref{rmk: characterization}),
it is enough to show that
$$  \Uppsi_Q( \tF^\up_{[Q]}(\ue) ) = E_q(m(\ue))_{\le \xi}$$
for every $\cc \in (\Z_{\ge 0})^\ell$. Note that we have already shown that
$$
 \Uppsi_Q( \tF^\up_{[Q]}(\al) ) = \Uppsi_Q(\tD_{[Q]}(\al,\al^-)) = F_q(X_{i,p})_{\le \xi}
$$
as the special case of the first assertion, where $\phi_Q(i,p)=(\al,0)$. Since $\Uppsi_Q$ is a $\Z[q^{\pm1/2}]$-algebra homomorphism, we have
$$
\Uppsi_Q( \tF^\up_{[Q]}(\ue) ) =q^{a(\ue)} E_q(m(\ue))_{\le \xi} \quad \text{ for some } a(\ue) \in \tfrac{1}{2}\Z.
$$
To complete a proof of the second assertion, we need to show $a(\ue)=0$. Since $L_q(m(\ue))$ and $\tG_{[Q]}(\ue)$ are $\overline{(\cdot)}$-invariant, the equations~\eqref{it: Lq(m)} in~\eqref{eq: L_q} and~\eqref{it: tGc} in~\eqref{eq: dual canonical}
can be rephrased in terms of $E_q(m(\ue))$ and $\tF^\up_{[Q]}(\ue)$ as follows:
\begin{align*}
\overline{E_q(m(\ue))}   \in E_q(m(\ue)) + \sum_{m' \lN m(\ue)} \Z[q^{\pm 1}] E_q(m')
 \ \ \text{and} \ \
\overline{\tF^\up_{[Q]}(\ue)}  \in \tF^\up_{[Q]}(\ue) + \sum_{\ue' \widetilde{<}_{[Q]} \ue} \Z[q^{\pm 1}] \tF^\up_{[Q]}(\ue').
\end{align*}
Therefore, we have
\begin{align*}
\overline{\Uppsi_Q(  \tF^\up_{[Q]}(\ue) )} & =  \Uppsi_Q\left(\,\overline{ \tF^\up_{[Q]}(\ue)}   \,\right) = q^{a(\ue)} E_q(m(\ue))_{\le \xi} + \sum_{m' \ne m(\ue)} \Z[q^{\pm 1/2}] E_q(m')_{\le \xi}, \\
\overline{\Uppsi_Q(  \tF^\up_{[Q]}(\ue)    )} & =   \overline{q^{a(\cc)} E_q(m(\ue))_{\le \xi} }     = q^{-a(\ue)} E_q(m(\ue))_{\le \xi} + \sum_{m' \ne  m(\ue)} \Z[q^{\pm 1/2}] E_q(m')_{\le \xi}.
\end{align*}
Thus $a(\ue)=0$ and the second assertion also follows.
\end{proof}

\begin{corollary} \label{cor: HLO iso}
$($HLO isomorphism$)$ For each Dynkin quiver $Q=(\Dynkin,\xi)$, there exists a $\Z$-algebra isomorphism
$$  \Psi_Q: \calA_{\bbA}(\n)  \longrightarrow \frakK_{q,Q} $$
such that $( \cdot)_{\le \xi} \circ \Psi_Q = \Uppsi_Q$.  Furthermore, this isomorphism sends $\tbfB^\up$ to $\sfL_{q,Q}$.
\end{corollary}

Now we can give an affirmative answer to Conjecture~\ref{conj: JLO1 sencond} when $m \in \calM_+^Q$.
\begin{corollary} \label{cor: partial proof}
For a KR-monomial $m^{(i)}_{k,p}$ with $m^{(i)}_{k,p} \in \calM_+^Q$ for some Dynkin quiver $Q$, we have
$$  F_q(m^{(i)}_{k,p}) = L_q(m^{(i)}_{k,p}).$$
\end{corollary}

\begin{proof}
By~\eqref{eq: D to F}, $ F_q(m^{(i)}_{k,p}) $ corresponds to a normalized unipotent quantum minor, so it is contained in $\tbfB^\up$. Thus it follows from Corollary~\ref{cor: HLO iso} that $ F_q(m^{(i)}_{k,p}) = L_q(m)$ for some $m \in \calM_+$.
By Theorem \ref{thm: F_q} and Theorem \ref{thm: L_q}, $L_q(m)$ should have $\um^{(i)}_{k,p}$ as its maximal $\calX_q$-monomial in $L_q(m)$ with respect to $\leN$. Hence we conclude that $m=m^{(i)}_{k,p}$, which implies our assertion.
\end{proof}

The following corollary can be obtained from combining HLO isomorphism in Corollary~\ref{cor: HLO iso} and the main result of Qin \cite[Theorem 1.2.2]{Qin20} (see also~\cite{McNa21}), which proves that every cluster monomial of $\calA_\bbA(\n)$ is contained in $\tbfB^\up$.
Notice that this corollary also recovers Corollary~\ref{cor: partial proof}.

\begin{corollary} \label{thm: cluster monomials are in canonical basis}
Every cluster monomial of $\frakK_{q,Q}$ is contained in the canonical basis $\sfL_{q,Q}$.
\end{corollary}

\begin{example} \label{ex:for main1 in C3}
   Let us illustrate Theorem \ref{thm:main1} in type $C_3$.
   We continue Example \ref{ex: C3 AR}, where the Dynkin quiver is $$ Q=  \xymatrix@R=0.5ex@C=6ex{    *{\circled{2}}<3pt> \ar@{->}[r]_<{1 \ \ } & *{\circled{2}}<3pt> \ar@{->}[r]_<{2 \ \ } &*{\circled{4}}<3pt>
\ar@{-}[l]^<{\ \ 3   }  }, \qquad \xi_1=3,\,\, \xi_2=2,\,\, \xi_3=1. $$
    The Coxeter element $\tau_Q$ adapted to $Q$ is given by $\tau_Q=s_1s_2s_3$, so the corresponding reduced expression $\ii_\circ$ of $w_\circ$ is $\ii_\circ=(1,2,3,1,2,3,1,2,3)$.
    The positive roots $\beta_k = \beta^{\ii_\circ}_k$ with respect to $\ii_\circ$ are given by
    \begin{align} \label{eq:beta k in C3}
        \raisebox{3.2em}{ \scalebox{0.7}{\xymatrix@!C=2ex@R=2ex{
        (i\setminus p) & -3&  -2 & -1 & 0 & 1 & 2  & 3\\
        1&&& \be_7 \ar@{->}[dr]&& \be_4 \ar@{->}[dr] && \be_1   \\
        2&& \be_8 \ar@{=>}[dr]\ar@{->}[ur] && \be_5 \ar@{=>}[dr] \ar@{->}[ur]  && \be_2 \ar@{->}[ur] \\
        3& \be_9  \ar@{->}[ur] && \be_6 \ar@{->}[ur]  && \be_3 \ar@{->}[ur]  }}}
        \quad = \quad
        \raisebox{3.2em}{ \scalebox{0.7}{\xymatrix@!C=2ex@R=2ex{
        (i\setminus p) & -3&  -2 & -1 & 0 & 1 & 2  & 3\\
        1&&& \al_1+\al_2+\al_3\ar@{->}[dr]&& \al_2 \ar@{->}[dr] && \al_1   \\
        2&& \al_2+\al_3 \ar@{=>}[dr]\ar@{->}[ur] && \al_1+2\al_2+\al_3 \ar@{=>}[dr] \ar@{->}[ur]  && \al_1+\al_2 \ar@{->}[ur] \\
        3& \al_3  \ar@{->}[ur] && 2\al_2+\al_3 \ar@{->}[ur]  && 2\al_1+2\al_2+\al_3 \ar@{->}[ur]  }}}
    \end{align}
    Under \eqref{eq: Uppsi}, each $\tD_{[\ii_0]}(\al,0)$ corresponds to $\underline{m^{(i)}[p,\xi_i]}$ with $\phi_Q(i,p)=(\al,0)$ as follows:
    \begin{align} \label{eq:nUQM to KRmono in C3}
        \raisebox{3.2em}{ \scalebox{0.7}{\xymatrix@!C=2ex@R=2ex{
        (i\setminus p) & -3&  -2 & -1 & 0 & 1 & 2  & 3\\
        1&&& \tD(\be_7,0) \ar@{->}[dr]&& \tD(\be_4,0) \ar@{->}[dr] && \tD(\be_1,0)   \\
        2&& \tD(\be_8,0) \ar@{=>}[dr]\ar@{->}[ur] && \tD(\be_5,0) \ar@{=>}[dr] \ar@{->}[ur]  && \tD(\be_2,0) \ar@{->}[ur] \\
        3& \tD(\be_9,0)  \ar@{->}[ur] && \tD(\be_6,0) \ar@{->}[ur]  && \tD(\be_3,0) \ar@{->}[ur]  }}}
        \quad \overset{\Uppsi_Q}{\longmapsto} \quad
        \raisebox{3.2em}{ \scalebox{0.7}{\xymatrix@!C=2ex@R=2ex{
        (i\setminus p) & -3&  -2 & -1 & 0 & 1 & 2  & 3\\
        1&&& \underline{m^{(1)}[-1,3]} \ar@{->}[dr]&& \underline{m^{(1)}[1,3]} \ar@{->}[dr] && \underline{m^{(1)}[3,3]}   \\
        2&& \underline{m^{(2)}[-2,2]} \ar@{=>}[dr]\ar@{->}[ur] && \underline{m^{(2)}[0,2]} \ar@{=>}[dr] \ar@{->}[ur]  && \underline{m^{(2)}[2,2]} \ar@{->}[ur] \\
        3& \underline{m^{(3)}[-3,1]} \ar@{->}[ur] && \underline{m^{(3)}[-1,1]} \ar@{->}[ur]  && \underline{m^{(3)}[1,1]} \ar@{->}[ur]  }}}
    \end{align}
    Here $\tD(\be_i,0)$ $(1 \le i \le 3)$ are the normalized dual PBW-vectors with weights $-\be_i$, while
    $\tD(\be_i,0)$ $(4 \le i \le 9)$ are not.

On the other hand, for $\be_k \in \Phi_+$ with $\res^{[\ii_\circ]}(\be_k)=  i$, we have
    \smallskip
    \begin{enumerate}[(1)]
        \item $i=1:$ $
            \la_{\be_1}^{\ii_\circ} = \varpi_1 - \al_1, \ \
            \la_{\be_4}^{\ii_\circ} = \varpi_1 - \al_1 - \al_2, \ \
            \la_{\be_7}^{\ii_\circ} = \varpi_1 - 2\al_1 - 2\al_2 - \al_3,
        $

        \item $i=2:$ $
            \la_{\be_2}^{\ii_\circ} = \varpi_2 - \al_1 - \al_2, \ \
            \la_{\be_5}^{\ii_\circ} = \varpi_2 - 2\al_1 - 3\al_2 - \al_3, \ \
            \la_{\be_8}^{\ii_\circ} = \varpi_2 - 2\al_1 - 4\al_2 - 2\al_3,
        $

        \item $i=3: $
        $
            \la_{\be_3}^{\ii_\circ} = \varpi_3 - 2\al_1 - 2\al_2 - \al_3, \ \
            \la_{\be_6}^{\ii_\circ} = \varpi_3 - 2\al_1 - 4\al_2 - 2\al_3, \ \
            \la_{\be_9}^{\ii_\circ} = \varpi_3 - 2\al_1 - 4\al_2 - 3\al_3,
        $
    \end{enumerate}
    \smallskip
    \noindent
    and each $\tD(i;p,p+2)$ is the normalized dual PBW-vector (see the convention in \eqref{eq:quantum minors to Q}):
    \begin{align*}
        \raisebox{3.2em}{ \scalebox{0.7}{\xymatrix@!C=2ex@R=2ex{
        (i\setminus p) & -3&  -2 & -1 & 0 & 1 & 2  & 3\\
        1\,\,\,&&& \tD(1;-1,1) \ar@{->}[dr]&& \tD(1;1,3) \ar@{->}[dr] && \tD(1;3,5)   \\
        2\,\,\,&& \tD(2;-2,0) \ar@{=>}[dr]\ar@{->}[ur] && \tD(2;0,2) \ar@{=>}[dr] \ar@{->}[ur]  && \tD(2;2,4) \ar@{->}[ur] \\
        3\,\,\,& \tD(3;-3,-1)  \ar@{->}[ur] && \tD(3;-1,1) \ar@{->}[ur]  && \tD(3;1,3) \ar@{->}[ur]  }}}
        \quad = \quad
        \raisebox{3.2em}{ \scalebox{0.7}{\xymatrix@!C=2ex@R=2ex{
        (i\setminus p) & -3&  -2 & -1 & 0 & 1 & 2  & 3\\
        1&&& \tF_{[\ii_\circ]}^{\rm up}(\be_7) \ar@{->}[dr]&&  \tF_{[\ii_\circ]}^{\rm up}(\be_4) \ar@{->}[dr] && \tF_{[\ii_\circ]}^{\rm up}(\be_1)   \\
        2&& \tF_{[\ii_\circ]}^{\rm up}(\be_8) \ar@{=>}[dr]\ar@{->}[ur] && \tF_{[\ii_\circ]}^{\rm up}(\be_5) \ar@{=>}[dr] \ar@{->}[ur]  && \tF_{[\ii_\circ]}^{\rm up}(\be_2) \ar@{->}[ur] \\
        3& \tF_{[\ii_\circ]}^{\rm up}(\be_9)  \ar@{->}[ur] && \tF_{[\ii_\circ]}^{\rm up}(\be_6) \ar@{->}[ur]  && \tF_{[\ii_\circ]}^{\rm up}(\be_3) \ar@{->}[ur]  }}}
    \end{align*}
    In fact, \eqref{eq: D to F} asserts that each generator $F_q(X_{i,p})_{\le \xi}$ of $\frakK_{q,Q}^T$ for $(i,p) \in \Gamma_0^Q$ corresponds to the normalized dual PBW-vector $\tF_{[\ii_\circ]}^{\rm up}(\al)$ with $\phi_Q(i,p)=(\al,0)$ for $\al \in \Phi_+$ under \eqref{eq:nUQM to KRmono in C3}, which implies that $\Uppsi_Q(\calA_\bbA(\n)) \simeq \frakK_{q,Q}^T$ because $\calA_\bbA(\n)$ is generated by the (normalized) dual PBW-vectors.
    \medskip

    Now we explain the computations behind the proof of \eqref{eq: D to F}.
    First, if $u=\xi_i+2$, then $\tD(i;p,u) = \tD_{[Q]}(\al,0)$ for $\al \in \Phi_+$ satisfying $\phi_Q(i,p)=(\al,0)$, so $\Uppsi_Q(\tD(i;p,u)) = F_q(m^{(i)}[p,\xi_i])_{\le \xi}$ by \eqref{eq:nUQM to KRmono in C3}.
    Second, let us determine $\Uppsi_Q(\tD(1;1,3))$, $\Uppsi_Q(\tD(2;0,2))$, and $\Uppsi_Q(\tD(3;-1,1))$.
    By \eqref{eq: tD ips det}, we have
    \begin{align} \label{eq:det rel of D113 and D135 in C3}
        \tD(1;1,3) \tD(1;3,5)
        =
        q^{\tka_1} \tD(1;1,5) + q^{\tze_1} \tD(2;2,4),
    \end{align}
    where $\tD(1;1,5) \mapsto \underline{m^{(1)}[1,3]}$, $\tD(2;2,4) \mapsto \underline{m^{(2)}[2,2]}$, and $\tD(1;3,5) \mapsto \underline{m^{(1)}[3,3]}$ under \eqref{eq:nUQM to KRmono in C3}.
     We remark that the half-integers $\tka_1$ and $\tze_1$ could be computed explicitly by \cite[Proposition 5.4]{GLS13}.
    This yields
    \begin{equation} \label{eq:image of tD113 in C3}
        \Uppsi_Q(\tD(1;1,3)) = \frac{q^{\tka_1}\underline{X_{1,1}X_{1,3}} + q^{\tze_1} \underline{X_{2,2}} }{\underline{X_{1,3}}} = q^{\frac{1}{2}}\tX_{1,1} + q^{\frac{1}{2}} \tX_{2,2}*\tX_{1,3}^{-1} =
        F_q(X_{1,1})_{\le \xi}.
    \end{equation}
    where $\tka_1 = -\frac{1}{2}$ and $\tze_1 = \frac{1}{2}$ (cf.~\cite[Appendix A]{JLO1} for the formula of $F_q(X_{1,1})$).
    Note that $\tka_1$ and $\tze_1$ are not only determined by \eqref{eq:det rel of D113 and D135 in C3}, but also by the bar-invariance of $\tD(1;1,3)$ and $F_q(X_{1,1})_{\le \xi}$.
    A similar computation gives
    \begin{align}
        \tD(2;0,2)\tD(2;2,4) &= q^{\tka_2} \tD(2;0,4) + q^{\tze_2} \tD(1;1,3) \tD(3;1,3), \label{eq:det rel of D202 and D224 in C3} \\
        \tD(3;-1,1)\tD(3;1,3) &= q^{\tka_3} \tD(3;-1,3) + q^{\tze_3} \tD(2;0,2)^2, \label{eq:det rel of D3m11 and D313 in C3}
    \end{align}
    where $\tD(2;2,4) \mapsto \underline{m^{(2)}[2,2]}$, $\tD(2;0,4) \mapsto \underline{m^{(2)}[0,2]}$, $\tD(3;1,3) \mapsto \underline{m^{(3)}[1,1]}$, $\tD(3;-1,3) \mapsto \underline{m^{(3)}[-1,1]}$, $\tka_2 = \frac{1}{2},\, \tze_2=\frac{3}{2}$, and $\tka_3 = 0,\, \tze_3 = 2$.
    Note that \eqref{eq:det rel of D113 and D135 in C3}, \eqref{eq:det rel of D202 and D224 in C3}, and \eqref{eq:det rel of D3m11 and D313 in C3} correspond to the quantum folded $T$-system \eqref{eq: T-system via C(t)} by $\Uppsi_Q$.
    By \eqref{eq:det rel of D113 and D135 in C3}, \eqref{eq:image of tD113 in C3}, \eqref{eq:det rel of D202 and D224 in C3}, and \eqref{eq:det rel of D3m11 and D313 in C3},  we obtain
    \begin{align}
        & \Uppsi_Q(\tD(2;0,2)) = q^{\frac{1}{2}} \tX_{2,0} + q^{\frac{5}{2}} \tX_{1,1} * \tX_{3,1} * \tX_{2,2}^{-1} + q^{\frac{3}{2}} \tX_{3,1}*\tX_{1,3}^{-1} = F_q(X_{2,0})_{\le \xi}, \label{eq:image of tD202 in C3} \\
        \begin{split} \label{eq:image of tD3m11 in C3}
        & \Uppsi_Q(\tD(3;-1,1)) = q\tX_{3,-1} + q^2 \tX_{2,0}^2 * \tX_{3,1}^{-1} + (1+q^2) \tX_{2,1}*\tX_{1,2}*\tX_{2,3}^{-1} + (q^{-1} + q) \tX_{2,0}*\tX_{1,3}^{-1} \\
        & \qquad \qquad \qquad \quad + q^5 \tX_{1,1}^2 * \tX_{3,1} * \tX_{2,2}^{-1} + (q+q^3) \tX_{1,1}*\tX_{3,1}*\tX_{2,2}^{-1}*\tX_{1,3}^{-1} + q^2 \tX_{3,1}*\tX_{1,3}^{-2} \\
        & \qquad \qquad \qquad \quad = F_q(X_{3,-1})_{\le \xi}.
        \end{split}
    \end{align}
    Here \eqref{eq:image of tD113 in C3} (resp.~\eqref{eq:image of tD202 in C3}) is used to compute \eqref{eq:image of tD202 in C3} (resp.~\eqref{eq:image of tD3m11 in C3}).
    Finally, one can see that
    \begin{align*}
        \Uppsi_Q(\tD(1;-1,1)) = F_q(X_{1,-1})_{\le \xi}, \quad
        \Uppsi_Q(\tD(2;-2,0)) = F_q(X_{2,-2})_{\le \xi}, \quad
        \Uppsi_Q(\tD(3;-3,-1)) = F_q(X_{3,-3})_{\le \xi}
    \end{align*}
    by the same procedure as in the second case.
    We leave the details to the reader, and
    refer to \cite[Appendix A]{JLO1} for the explicit formulas of $F_q(X_{1,p})$, $F_q(X_{2,p})$, and $F_q(X_{3,p})$ for $p \in \Z$.
\end{example}

\section{Categorification and multiplicative relations} \label{Sec: quiver Hecke}
 The aim of this section is to investigate the multiplicative relations for pairs of $L_q(X_{i,p})$'s, which is a preparation to prove the remaining main results.
Our main ingredient is the categorification of $\calA_\bbA(\n)$ by graded $R$-modules \cite{KL1,KL2,R08,KKKO18}, where $R$ is the quiver Hecke algebra.
Combined with Theorem \ref{thm:main1}, this yields the relationship \eqref{eq:important diagram} between the Grothendieck ring of graded $R$-modules and the heart subring $\frakK_{q,Q}$ of the quantum virtual Grothendieck ring $\frakK_q$.
In particular, it allows us to utilize the crucial invariant (Definition \ref{def:invariants}) closely related to the simplicity for the convolution product of certain simple $R$-modules \cite{KKKO15,KKOP21}, which is determined by the inverse of $\usfB(t)$ \cite{Fuj22a, KO22A}.
When \eqref{eq:important diagram} is taken into account,
these results enable us to obtain the multiplicative relations for certain pairs of $L_q(X_{i,p})$'s.

\subsection{Quiver Hecke algebras and quantum virtual Grothendieck rings}
Let $\bfk$ be a field. For $i,j \in I$, we choose polynomials $\calQ_{i,j}(u,v) \in \bfk[u,v]$ such that  $\calQ_{i,j}(u,v)=\calQ_{j,i}(v,u)$, which is of the form
$$
\calQ_{i,j}(u,v) = \delta(i \ne j)
\displaystyle\sum_{pd_i+qd_j=-(\al_i,\al_j)} t_{i,j;p,q}u^pv^p \quad \text{ where }  t_{i,j;-\sfc_{i,j},0} \in \bfk^\times.$$

For $\be \in \rl^+$ with $\het(\be)=m$, we set
$$ I^\be \seteq \Bigl\{ \eta =( \eta_1,\ldots, \eta_m) \in I^m \ \bigm| \ \sum_{k=1}^m \al_{ \eta_k} =\be   \Bigr\}.$$
The \emph{quiver Hecke algebra} $R(\be)$ associated with $\cm$ and $(\calQ_{i,j}(u,v))_{i,j \in I}$ is the $\Z$-graded $\bfk$-algebra generated by
$$  \{ e( \eta) \ | \  \eta \in I^\be \}, \ \{ x_k \ | \ 1 \le  k \le m \}, \ \{ \tau_l \ | \ 1 \le l <m \} $$
subject to certain relations (see \cite{KL1,KL2,R08}). Here $1 = \sum_{\eta \in I^\be}e(\eta)$, the elements
$e( \eta) $'s are idempotents and
the $\Z$-grading of $R(\be)$ is defined by
$$  \deg(e( \eta))=0, \quad  \deg(x_ke( \eta))=(\al_{ \eta_k},\al_{ \eta_k}), \quad \deg(\tau_le( \eta))=-(\al_{ \eta_l},\al_{ \eta_{l+1}}).$$

We denote by $R(\beta)\gmod$ the category of  finite-dimensional graded $R(\beta)$-modules with degree-preserving homomorphisms. For $M,N \in R(\beta)\gmod$, we denote by $\Hom_{R(\beta)}(M,N)$ the space of degree-preserving module homomorphisms.
We set  $R\gmod \seteq \bigoplus_{\beta \in \rl^+} R(\beta)\gmod$.
The trivial $R(0)$-module of degree 0 is denoted by $\mathbf{1}$.
We define the grading shift functor $q$
by $(q M)_k = M_{k-1}$ for a $ \Z$-graded module $M = \bigoplus_{k \in \Z} M_k $.

For $M, N \in R(\beta)\gmod$, we define
\[
\HOM_{R(\beta)}( M,N ) \seteq \bigoplus_{k \in \Z} \Hom_{R(\beta)}(q^{k}M, N).
\]
For $M \in R(\beta)\gmod$, we set $M^\star \seteq  \HOM_{R(\be)}(M, R(\be))$ with the $R(\beta)$-action given by
$$
(r \cdot f) (u) \seteq  f(\psi(r)u), \quad \text{for  $f\in M^\star$, $r \in R(\beta)$ and $u\in M$,}
$$
where $\psi$ is the antiautomorphism of $R(\beta)$ which fixes the generators.
We say that $M$ is \emph{self-dual} if $M \simeq M^\star$. For $M \in R(\beta)\Mod$, we set $\wt(M) \seteq -\be$.

For $\be,\be' \in \rl^+$, we set
$
e(\beta, \beta') \seteq \sum_{\eta \in I^\beta, \eta' \in I^{\beta'}} e(\eta, \eta'),
$
where $e(\eta, \eta')$ is the idempotent corresponding to the concatenation
$\eta\ast\eta'$ of
$\eta$ and $\eta'$.
It induces an injective ring homomorphism
$$R(\beta)\otimes R(\beta')\to e(\beta,\beta')R(\beta+\beta')e(\beta,\beta')$$
defined by
\begin{align*}
& e(\eta)\otimes e(\eta')\mapsto e(\eta,\eta'),  \quad x_ke(\beta)\otimes 1\mapsto x_ke(\beta,\beta'), \quad
1\otimes x_ke(\beta')\mapsto x_{k+\het(\beta)}e(\beta,\beta'), \\
& \tau_ke(\beta)\otimes 1\mapsto \tau_ke(\beta,\beta') \quad \text{ and }\quad
1\otimes \tau_ke(\beta')\mapsto \tau_{k+\het(\beta)}e(\beta,\beta').
\end{align*}
For $M \in R(\beta)\Mod$ and $N \in R(\beta')\Mod$, we set
$$
M \conv N \seteq R(\beta+\beta') e(\beta, \beta') \otimes_{R(\beta) \otimes R(\beta')} (M \otimes N).
$$

We say that simple $R$-modules $M$ and $N$ \emph{strongly commute} if $M \conv N$ is simple.  A simple $R$-module
$L$ is said to be \emph{real} if $L$ strongly commutes with itself. Note that if $M$ and $N$ strongly commute, then $M\conv N \simeq q^{k}N \conv M$ for some $k \in \Z$.
We denote by $K(R\gmod)$ the Grothendieck ring of $R\gmod$ induced by $\conv$. Note that
$K(R\gmod)$ has a natural $\Z[q^{\pm 1}]$-algebra structure given by grading shifts.

\begin{theorem}[\cite{KL1,KL2,R08}]\label{thm:KLR categorification}
There exists a $\Z[q^{\pm 1/2}]$-algebra isomorphism
$$  \Upomega :    \calK_{\bbA}(R\gmod) \seteq  \Z[q^{\pm 1/2}] \otimes_{\Z[q^{\pm 1}]} K(R\gmod) \isoto \calA_{\bbA}(\n) $$
which preserves weights.
\end{theorem}

\begin{proposition}[{\cite[Proposition 4.1]{KKOP18}}]\label{prop:categorification of quantum minors}
For $\varpi \in \wl^+$ and $\mu,\zeta \in \weyl \varpi$ with $\mu \preceq \zeta$, there exists a self-dual real simple $R( \zeta-\mu)$-module $\sfM(\mu,\zeta)$ such that
$$\Upomega ([\sfM(\mu,\zeta)]) = D(\mu,\zeta).$$
Here, $[\sfM(\mu,\zeta)]$ denotes the isomorphism class of $\sfM(\mu,\zeta)$.
\end{proposition}
We call $[\sfM(\mu,\zeta)]$ the \emph{determinantal module} associated to $D(\mu,\zeta)$.
For $[\ii_\circ]$ of $w_\circ$ and $\al \in \Phi_+$, we write $S_{[\ii_\circ]}(\al)$ the determinantal $R(\al)$-module such that
$$  \Upomega ([S_{[\ii_\circ]}(\al)]) = D_{[\ii_\circ]}(\al,\al^-).$$
We call $S_{[\ii_\circ]}(\al)$  the \emph{cuspidal module} associated with $[\ii_\circ]$ and $\al$. Note that $\wt(S_{[\ii_\circ]}(\al)) =-\al$.
For a Dynkin quiver $Q=(\Dynkin,\xi)$ and $(i,p) \in \Gamma^Q_0$, we set
$$   S_Q(i;p) \seteq S_{[Q]}(\be)  \quad \text{ where } \phi_Q(i,p) = (\be,0).  $$

\begin{theorem} \label{thm:important relationship}
    For each Dynkin quiver $Q=(\Dynkin,\xi)$,
    we have the $\Z[q^{\pm 1/2}]$-algebra isomorphism $\Upxi_Q$ given by
    \begin{equation*}
        \Upxi_Q :
        \xymatrix@R=0.5ex@C=6ex{
        \frakK_{q,Q} \ar@{->}[r] & \calK_{\bbA}(R\gmod) \\
        L_q^Q(\be,0) \ar@{|->}[r] & q^{(\be,\be)/4 + (\be ,\rho)/2 }[S_Q(i;p)]
        }
    \end{equation*}
    where $(i,p) \in \Gamma^Q_0$ with $\phi_Q(i,p)=(\be,0)$.
\end{theorem}

\begin{proof}
It follows from Theorem \ref{thm:main1}, Corollary \ref{cor: HLO iso}, and Theorem \ref{thm:KLR categorification} that
we have the $\Z[q^{\pm 1/2}]$-algebra isomorphism $\Upxi_Q$ given as follows:
\begin{equation} \label{eq:important diagram}
\begin{split}
 \raisebox{4.7em}{ \xymatrix@R=3ex@C=6ex{   & \calA_\bbA(\n)\ar[dr]^{\Upomega}\ar[dl]_{\Psi_Q}   \\ \qquad \qquad  \frakK_{q,Q}  \ar[rr]_{ \Upxi_Q \seteq  \Omega \circ \Psi_Q^{-1}} && \calK_{\bbA}(R\gmod)  \\
L_q(X_{i,p}) = L_q^Q(\be,0) \ar@{|->}[rr] && q^{(\be,\be)/4 + (\be,\rho)/2 }[S_Q(i;p)],
}}
\end{split}
\end{equation}
for $(i,p) \in \Gamma^Q_0$ with $\phi_Q(i,p)=(\be,0)$.
Note that the last correspondence follows from
$$ \Psi_Q^{-1}(L_q^Q(\be,0)) = \tF^{\up}_{[Q]}(\be) =  \Omega^{-1}(q^{(\be,\be)/4 +  (\be,\rho)/2}[S_{[Q]}(\be)]) $$
due to Theorem \ref{thm:main1} and Proposition \ref{prop:categorification of quantum minors}.
\end{proof}

 \subsection{Affreal modules, invariants, and multiplicative relations}
For $\be \in \rl^+$ with $\het(\be)=m$ and $i \in I$, let
\begin{align} \label{eq: pibe}
\mathfrak{p}_{i,\be} = \sum_{ \eta\,  \in I^\be} \left( \prod_{a \in [1,m]; \  \eta_a  = i} x_a\right)e(\eta) \in Z(R(\be)),
\end{align}
where $Z(R(\be))$ denotes the center of $R(\be)$.

\begin{definition}
Let $M$ be an $R(\be)$-module. We say that $M$ \emph{admits an affinization} if  there exists $R(\be)$-module $\hsfM$
satisfying the followings: There exists an endomorphism $z_{\hsfM}$ of degree $t_{\hsfM} \in \Z_{>0}$ such that $\hsfM/z_{\hsfM}\hsfM \simeq M$ and
\bnum
\item $\hsfM$ is a finitely generated free modules over the polynomial ring $\bfk[z_{\hsfM}]$,
\item $\mathfrak{p}_{i,\be}  \hsfM \ne 0$ for all $i \in I$.
\ee
We say that a simple $R(\be)$-module $M$ is \emph{affreal} if $M$ is real and admits an affinization.
\end{definition}

 \begin{remark}
It is proved in~\cite[Proposition 2.5]{KP18} that every affinization is \emph{essentially even}; that means $t_{\hsfM} \in 2\Z_{>0}$. Thus we assume that every affinization has an even degree.
\end{remark}

\begin{theorem} [{\cite[Theorem 3.26]{KKOP21}}] \label{thm:affinization of det}
For $\varpi \in\wl^+$ and $\mu,\zeta \in \weyl\varpi$ with $\mu\preceq \zeta$, the determinantal module $\sfM(\mu,\zeta)$
admits an affinization $\hsfM(\mu,\zeta)$. When $\varpi=\varpi_i$, the module $\sfM(\mu,\zeta)$ admits an affinization of degree $(\al_i,\al_i)=  2d_i$.
\end{theorem}

\begin{proposition}
[{\cite{KKKO18,KKOP21}}]\label{prop: l=r}
Let $M$ and $N$ be simple modules such that one of them is affreal.
Then there exists a unique $R$-module homomorphism $\Rr_{M,N} \in \HOM_R(M,N)$ satisfying
$$\HOM_R(M \conv N,N \conv M)=\bfk\, \Rr_{M,N}.$$
\end{proposition}

\begin{definition} \label{def:invariants}
Let $M$ and $N$ be simple modules such that one of them is affreal.
We define
\begin{align*}
\La(M,N) &\seteq  \deg (\Rr_{M,N}) , \\
\tLa(M,N) &\seteq   \frac{1}{2} \big( \La(M,N) + (\wt(M), \wt(N)) \big) , \\
\de(M,N) &\seteq  \frac{1}{2} \big( \La(M,N) + \La(N,M)\big).
\end{align*}
\end{definition}

For $M,N \in R\gmod$, let us
denote by $M \hconv N$  (resp.~$M \sconv N$) the \emph{head}  (resp.~\emph{socle}) of $M \conv N$.

\begin{proposition}
[{\cite[Theorem 3.2]{KKKO15}, \cite[Proposition 2.10]{KP18}, \cite[Proposition 3.13, Lemma 3.17]{KKOP21}}] \label{prop: simple head}
Let $M$ and $N$ be simple $R$-modules such that one of them is affreal.
Then we have
\bnum
\item   $M \conv N$ has simple socle and simple head,
\item ${\rm Im}(\Rr_{M,N})$ is isomorphic to $M \hconv N$ and $N\sconv M$,
\item $M \hconv N$ and $M \sconv N$ appear once in the composition series of $M \conv N$, respectively.
\ee
\end{proposition}

 Note that Proposition \ref{prop:categorification of quantum minors} and Theorem \ref{thm:affinization of det} tell us that $\left\{\,S_Q(i;p) \, \left| \, (i,p) \in \Gamma^Q_0\, \right.\right\}$ is a family of affreal modules.
Hence the results below become the main ingredients for our purpose from the viewpoint of Theorem \ref{thm:important relationship}.

\begin{theorem}[\cite{KKKO15,KKOP21}]\label{thm: commute1}
Let $M$ and $N$ be simple modules such that one of them is affreal. Then the following three conditions are equivalent:
\bnum
\item $\de(M,N)=0$,
\item \label{it:com rel} $M \conv N \simeq q^{-\La(M,N)} N \conv M$,
\item $M \conv N$ is simple.
\ee
In particular, we have $\de(M,M)=0$ whenever $M$ is affreal.
\end{theorem}

Let us recall \eqref{eq:essential part of Laurent expansion of tQCM} and \eqref{eq:def of tde ij k}.
Then the following theorem was proved in \cite{Fuj22a} for simply-laced types, and it is generalized recently in \cite{KO22A}.

\begin{theorem}[\cite{Fuj22a,KO22A}]\label{thm: commute2}
Let $Q=(\Dynkin,\xi)$ be a Dynkin quiver. Then we have
$$    \de(S_Q(i;p),S_Q(j;s))  =   \tde_{i,j}[|p-s|]
\quad \text{ for any } (i,p), (j,s) \in \Gamma^Q_0.$$
where $\tde_{i,j}(t)$ is defined in \eqref{eq:essential part of Laurent expansion of tQCM}.
In particular, we have
$$\de(S_{[Q]}(\al_i),S_{[Q]}(\al_j)) =\max(d_i,d_j)\delta(i\sim j). $$
\end{theorem}

\begin{proposition}[\cite{FHOO,KO22A}]\label{prop: dual zero}
Let $a \ne b \in I$ and $Q=(\Dynkin,\xi)$ be a Dynkin quiver with $\phi_Q(i,p)=(\al_a,0)$ and $\phi_Q(j,s)=(\al_b,0)$. Then we have
$$  \tde_{i,j^*}[\sfh-|p-s|] = 0.  $$
\end{proposition}
\begin{proof}
The proof can be found in \cite[Corollary 8.9]{KO22A} (cf.~\cite[Lemma 9.12]{FHOO}).
\end{proof}

 Now we are ready to present the multiplicative relations for pairs of $L_q(X_{i,p})$'s.

\begin{corollary} \label{cor: commute1}
For $(i,p), (j,s) \in \hDynkin_0$, assume that there exists a Dynkin quiver $Q$ such that $(i,p), (j,s) \in \Gamma_0^Q$ and  $\tde_{i,j}[|p-s|]=0$.
Then we have
$$   L_q( X_{i,p})* L_q(X_{j,s})   = q^{\ucalN(i,p;j,s)}  L_q(X_{j,s}) *L_q(X_{i,p})$$
and hence $\La(S_Q(i;p), S_Q(j;s))= -\ucalN(i,p;j,s).$
\end{corollary}

\begin{proof}
By Theorem~\ref{thm: commute1} and Theorem~\ref{thm: commute2},
$S_{Q}(i;p) \conv S_{Q}(j;s) \simeq q^k S_{Q}(j;s) \conv S_{Q}(j;s)$ for some $k \in \Z$.
By Theorem \ref{thm:important relationship}, we have
$$   L_q(X_{i,p})* L_q(X_{j,s})   = q^{k}  L_q(X_{j,s}) *L_q(X_{i,p}).$$
Since an element of $\frakK_q$ (and so $\frakK_{q,Q}$) is determined by its dominant monomials (cf.~\cite[Lemma 5.20 and Lemma 5.40]{JLO1}),
the integer $k$ must be $\ucalN(i,p;j,s)$ by comparing the dominant monomials on both sides (recall Definition \ref{def:quantum torus Xqg}),
as we desired.
Note that the last statement follows directly from Theorem \ref{thm:important relationship} and Theorem \ref{thm: commute1} (\ref{it:com rel}).
\end{proof}

\begin{corollary} \label{cor: order works well}
Let $Q=(\Dynkin,\xi)$ be a Dynkin quiver. Let $(i,p),(j,s) \in \hDynkin_0$ such that $\xi_i \ge p \ge s > \xi_j$. Then we have
$$   L_q(X_{i,p}) * L_q(X_{j,s})  =   L_q(X_{j,s}) *  L_q(X_{i,p})$$
\end{corollary}

\begin{proof}
Take a Dynkin quiver $Q'$ such that $(i,p),(j,s) \in \Gamma^{Q'}_0$.  Take $a \in \Z_{\ge 0}$ and $b \in \Z_{\ge 1}$
such that $p = \xi_i-2a$ and $s = \xi_j+2b$.
Then we have
$$  \tde_{i,j}[p-s]  =  \tusfb_{i,j}(p-s-1) = \tusfb_{i,j}(\xi_i-\xi_j-1-2(a+b)) .$$
Since $\xi_i-\xi_j-1 < d(i,j)$,  Lemma~\ref{lem: non-zero b}~\eqref{it: non-zero b} tells us  that $\tde_{i,j}(p-s) =0$.
By Corollary \ref{cor: commute1}, we have
$$L_q(X_{i,p}) * L_q(X_{j,s})  = q^{\ucalN(i,p;j,s)}  L_q(X_{j,s}) *  L_q(X_{i,p})  $$
Now we claim that $\ucalN(i,p;j,s)=0$. Note that
\begin{align}
\ucalN(i,p;j,s)
& = \tusfb_{i,j}(\xi_i-\xi_j-2a-2b-1)  +  \tusfb_{i,j}(\xi_i-\xi_j -2a -2b+1). \label{eq: two term vanish}
\end{align}
Since the two terms in~\eqref{eq: two term vanish} vanish by Lemma~\ref{lem: non-zero b}~\eqref{it: non-zero b}, the claim holds.
\end{proof}

\begin{corollary} \label{cor: dual commute}
Let $a \ne b \in I$ and $Q=(\Dynkin,\xi)$ be a Dynkin quiver with $\phi_Q(i,p)=(\al_a,0)$ and $\phi_Q(j,s)=(\al_b,0)$. Then we have
\begin{equation}\label{eq: dual commute}
 L_q(X_{i,p})* \frakD_q^{\pm 1} (L_q(X_{j,s})) = q^{\mp \ucalN(i,p;j,s)}
 \frakD^{\pm 1} (L_q(X_{j,s})) * L_q(X_{i,p}),
\end{equation}
where $\frakD_q$ is in \eqref{eq:dual map Dq} and we understand $\frakD_q^{+1} = \frakD_q$.
\end{corollary}

\begin{proof}
By Lemma~\ref{lem: D_q}, $\frakD_q^{\pm 1} (L_q(X_{j,s})) = L_q(X_{j^*,s\pm\sfh})$. If $|s\pm \sfh-p | > \sfh$, our assertion follows from Proposition~\ref{prop: far enough}. Thus it suffices to prove the assertion for
$$ \text{{\rm (i)} $s \le p$ and $\frakD_q$}, \quad \text{ or } \quad \text{{\rm (ii)} $s \ge p$ and $\frakD_q^{-1}$. }$$

\noindent
In case {\rm (i)}, the right hand side of~\eqref{eq: dual commute}  becomes
$ L_q(X_{i,p})*  L_q(X_{j^*,s+\sfh}).$
 By taking Dynkin quiver $Q'$ such that
$(i,p),(j^*,s+\sfh) \in \Gamma_0^{Q'}$,
it follows from Proposition~\ref{prop: dual zero}, Theorem~\ref{thm: commute1} and Theorem~\ref{thm: commute2} that
$S_{Q'}(i,p) \conv S_{Q'}(j^*,s+\sfh) \simeq q^k S_{Q'}(j^*,s+\sfh) \conv S_{Q'}(i,p) $ for some $k \in \Z$.
By Theorem \ref{thm:important relationship} with respect to $Q'$, we have
$$ L_q(X_{i,p})*  L_q(X_{j^*,s+\sfh}) \simeq   q^{\ucalN(i,p;j^*,\sfh)} L_q(X_{j^*,s+\sfh}) * L_q(X_{i,p}).$$
Since $\phi_{Q}(j^*,s+\sfh) =(\al_b,-1)$, the assertion for {\rm (i)} follows from Theorem~\ref{thm: N and wt}.  Case {\rm (ii)} is also proved in a similar way.
\end{proof}

\begin{example}
    We continue Example \ref{ex:for main1 in C3}.
    Let us recall that
    \begin{align*}
        \raisebox{3.2em}{ \scalebox{0.7}{\xymatrix@!C=2ex@R=2ex{
        (i\setminus p) & -3&  -2 & -1 & 0 & 1 & 2  & 3\\
        1\,\,\,&&& L_q(X_{1,-1}) \ar@{->}[dr]&& L_q(X_{1,1}) \ar@{->}[dr] && L_q(X_{1,3})   \\
        2\,\,\,&& L_q(X_{2,-2}) \ar@{=>}[dr]\ar@{->}[ur] && L_q(X_{2,0}) \ar@{=>}[dr] \ar@{->}[ur]  && L_q(X_{2,2}) \ar@{->}[ur] \\
        3\,\,\,& L_q(X_{3,-3})  \ar@{->}[ur] && L_q(X_{3,-1}) \ar@{->}[ur]  && L_q(X_{3,1}) \ar@{->}[ur]  }}}
        \quad \overset{\Upxi_Q}{\longmapsto} \quad
        \raisebox{3.2em}{ \scalebox{0.7}{\xymatrix@!C=2ex@R=2ex{
        (i\setminus p) & -3&  -2 & -1 & 0 & 1 & 2  & 3\\
        1&&& [\tS(1;-1)] \ar@{->}[dr]&&  [\tS(1;1)] \ar@{->}[dr] && [\tS(1;3)]   \\
        2&& [\tS(2;-2)] \ar@{=>}[dr]\ar@{->}[ur] && [\tS(2;1)] \ar@{=>}[dr] \ar@{->}[ur]  && [\tS(2;2)] \ar@{->}[ur] \\
        3& [\tS(3;-3)]  \ar@{->}[ur] && [\tS(3;-1)] \ar@{->}[ur]  && [\tS(3;1)] \ar@{->}[ur]  }}}
    \end{align*}
    where $\tS(i;p) = q^{(\be,\be)/4 + (\be,\rho)/2 } S_Q(i;p)$ with $\phi_Q(i,p) = (\be,0)$. As shown already, this relationship enables us to obtain the multiplicative relations of $L_q(X_{i,p})$'s from those of $[\tS(i;p)]$'s.
    By Proposition \ref{prop:formulas for inverses of tQCMs}, we have
    \begin{align} \label{eq:essential part of inv qQCM in C3}
    \begin{split}
        & \tde_{1,1}(t) = t + t^5, \quad
        \tde_{1,2}(t) = t^2 + t^4, \quad
        \tde_{1,3}(t) = 2t^3, \\
        & \tde_{2,2}(t) = t + 2t^3 + t^5, \quad
        \tde_{2,3}(t) = 2t^2 + 2t^4, \quad
        \tde_{3,3}(t) = 2t + 2t^3 + 2t^5.
    \end{split}
    \end{align}

Note that the following computations can also be deduced directly from the explicit formulas of $L_q(X_{i,p})$'s in \cite[(A.1)]{JLO1}.
    \bnum
    \item $(i,p) = (1,3)$ and $(j,s) = (2,2)$. By \eqref{eq:essential part of inv qQCM in C3}, we have
    $\tde_{1,2}[1] = 0$ and $\ucalN(1,3;2,2) = -1$, so it follows from Corollary \ref{cor: commute1} that
        \begin{equation*}
            L_q(X_{1,3})*L_q(X_{2,2}) = q^{-1}L_q(X_{2,2})*L_q(X_{1,3}),
        \end{equation*}
    where the products on both sides have unique dominant monomials $\tX_{1,3}*\tX_{2,2}$ and $q\tX_{1,3}*\tX_{2,2}$, respectively.
    \smallskip

    \item $(i,p) = (2,0)$ and $(j,s) = (3,-1)$. By \eqref{eq:essential part of inv qQCM in C3}, we have $\tde_{2,3}[1] = 0$ and $\ucalN(2,0;3,-1) = -2$. Thus we have
    \begin{equation*}
        L_q(X_{2,0})*L_q(X_{3,-1}) = q^{-2} L_q(X_{3,-1})*L_q(X_{2,0}),
    \end{equation*}
    where the products on both sides have unique dominant monomials $q^{-\frac{1}{2}}\tX_{3,-1}*\tX_{2,0}$ and $q^{\frac{3}{2}}\tX_{3,-1}*\tX_{2,0}$, respectively.
    \smallskip

    \item $(i,p) = (1,3)$ and $(j,s) = (3,-3)$. By \eqref{eq:essential part of inv qQCM in C3}, we have $\tde_{1,3}[6] = 0$, so it follows from Corollary \ref{cor: commute1} that
    \begin{equation*}
        L_q(X_{1,3})*L_q(X_{3,-3}) = L_q(X_{3,-3})*L_q(X_{1,3}).
    \end{equation*}
    Note that $\ucalN(1,3;3,-3) = 0$.
    \ee
    \smallskip
    \noindent
    We remark that $L_q(X_{i,p})$ and $L_q(X_{j,s})$ do not $q$-commute in the case of  $\tde_{i,j}[|p-s|] \neq 0$.
    For example, for $(i,p) = (2,2)$ and $(j,s) = (2,-2)$, we have  $\tde_{2,2}[3] = 2 \neq 0$.
    On the other hand, $L_q(X_{2,2})*L_q(X_{2,-2})$ (resp.~$L_q(X_{2,-2})*L_q(X_{2,2})$) has dominant monomials
    \begin{align*}
        q^2 \tX_{2,-2}*\tX_{2,2}, \quad
        q\tX_{1,-1}*\tX_{1,1}, \quad
        (q^{-1}+q) \tX_{2,0}, \\
        \text{(resp.~$q \tX_{2,-2}*\tX_{2,2}, \quad q^2\tX_{1,-1}*\tX_{1,1}, \quad (1+q^2) \tX_{2,0}$).}
    \end{align*}
    Hence $L_q(X_{2,2})*L_q(X_{2,-2}) \neq q^k L_q(X_{2,-2})*L_q(X_{2,2})$ for any $k \in \Z$.
\end{example}

\section{Presentations and automorphisms} \label{Sec: Presentation}
In this section, we introduce automorphisms $\sigma_{i,Q}^{\pm 1}$ of $\frakK_q$ with respect to a Dynkin quiver $Q$ with a source $i$.
For this purpose, we establish {\rm BCFG}-analogues of \cite[Section 10]{FHOO} by means of our previous results.
The analogues concern the presentation and automorphisms of $\frakK_q$.
In Section \ref{Sec: Braid}, we will show that such automorphisms induce a braid group action on $\frakK_q$.
\medskip

By~\eqref{eq: heart}, the commutative monoid $\calM_+$ can be understood as a product of $\frakD^k \calM_+^Q$ $(k \in \Z)$ for each Dynkin quiver $Q =(\Dynkin,\xi)$.
Thus for each $m \in \calM_+$, ~\eqref{eq: heart} tells us that $m$ has a unique factorization
\begin{align}\label{eq: decomposition m}
 m =  \prod_{k \in \Z} \frakD^k (m_k)
\end{align}
such that $m_k \in \calM^Q_+$, where $m_k=1$ for all but finitely many $k$.
Using Corollary~\ref{cor: order works well}, the proposition below holds in the same way as \cite[Proposition 10.1]{FHOO}.

\begin{proposition}
For $m \in \calM_+$, the element $E_q(m)$ can be written as
$$
E_q(m) = q^{\kappa(m,Q)} \prod^{\to}_{k \in \Z} \frakD_q^k(E_q(m_k)) \in \frakK_q,
$$
where
$$
\kappa(m,Q) \seteq -\dfrac{1}{2} \sum_{k<l} (-1)^{k+l} (\wt_Q(m_k),\wt_Q(m_l)).
$$
\end{proposition}

\subsection{The presentation}
In this subsection, we exhibit a presentation of the \emph{localized quantum virtual Grothendieck ring} defined by
$$ \bbK_{q} \seteq \Q(q^{1/2}) \otimes_{\Z[q^{\pm 1/2}]} \frakK_{q}.$$
Note that we can extend $\overline{( \cdot )}$ and $\frakD_q$ on $\bbK_{q}$.
For a Dynkin quiver $Q=(\Dynkin,\xi)$, we also set
$$\bbK_{q,Q} \seteq \Q(q^{1/2}) \otimes_{\Z[q^{\pm 1/2}]} \frakK_{q,Q}.$$
\smallskip

From now on, let us fix a Dynkin quiver $Q=(\Dynkin,\xi)$ and set, for each $(i,k) \in I \times \Z$,
\begin{align} \label{eq:sfxQik}
\sfx^Q_{i,k} \seteq  L_q(X_{\iota,p}) \qquad \text{ where\,\,\, $\phi_Q(\iota,p)=(\al_i,k)$. }
\end{align}
Note that $\frakD_q^l(\sfx^Q_{i,k}) = \sfx^Q_{i,k+l}$ for $l \in \Z$.
From \eqref{eq:An=Uq- and An=CN at q=1} and Corollary~\ref{cor: HLO iso}, we have an algebra isomorphism
\begin{align} \label{eq:tPsiQ}
\begin{split}
    \widetilde{\Psi}_Q:
    \xymatrix@!C=7ex@R=0.3ex{
       \calA_q(\n) \ar@{->}[r] & \bbK_{q,Q},\\
        f_i \ar@{|->}[r] & \sfx^Q_{i,0}.
    }
\end{split}
\end{align}
This implies that the quantum Serre relations holds for $i \ne j \in I$ and $m \in \Z$:
\begin{equation} \label{eq: q-serre}
\sum_{s=0}^{1-\sfc_{i,j}} (-1)^s \left[ \begin{matrix} 1-\sfc_{i,j} \\ s  \end{matrix}\right]_{q_i} (\sfx^Q_{i,m})^{1-\sfc_{i,j}-s}  \sfx^Q_{j,m} (\sfx^Q_{i,m})^s =0.
\end{equation}
Furthermore, it follows from Proposition~\ref{prop: quantum Boson1}, Proposition~\ref{prop: far enough}, and Corollary~\ref{cor: dual commute} that the quantum Boson relations also hold:
\begin{align}
\sfx^Q_{i,k}\sfx^Q_{j,k+1} & = q^{-(\al_i,\al_j)} \sfx^Q_{j,k+1}\sfx^Q_{i,k} + (1-q^{-(\al_i,\al_i)})\delta_{i,j}, \label{eq: q-Boson1}\\
\sfx^Q_{i,k}\sfx^Q_{j,l} &= q^{(-1)^{k+l}(\al_i,\al_j)} \sfx^Q_{j,l}\sfx^Q_{i,k}, \label{eq: q-Boson2}
\end{align}
for $i,j \in I$ and $k,l \in \Z$ with $l>k+1$.
\smallskip

Let us denote by $\hA_{q}(\n)$ the $\Q(q^{1/2})$-algebra generated by $\{ \sfy_{i,k} \ | \ (i,k) \in I \times \Z\}$ subject to the following relations:
\begin{align}
&\sum_{s=0}^{1-\sfc_{i,j}} (-1)^s \left[ \begin{matrix} 1-\sfc_{i,j} \\ s  \end{matrix}\right]_{q_i} \sfy_{i,k}^{1-\sfc_{i,j}-s}  \sfy_{j,k} \sfy_{i,k}^s =0,  \label{eq: q-serre y} \\
&\sfy_{i,k}\sfy_{j,k+1}  = q^{-(\al_i,\al_j)} \sfy_{j,k+1}\sfy_{i,k} + (1-q^{-(\al_i,\al_i)})\delta_{i,j}, \label{eq: q-Boson1 y}\\
&\sfy_{i,k}\sfy_{j,l} = q^{(-1)^{k+l}(\al_i,\al_j)}\sfy_{j,l}\sfy_{i,k}, \label{eq: q-Boson2 y}
\end{align}
for $i,j \in I$ and $k,l \in \Z$ with $l>k+1$.

\begin{theorem} \label{thm: presentation}
For each Dynkin quiver $Q=(\Dynkin,\xi)$, there exists an
algebra isomorphism
$$
\tUptheta_Q :  \hA_{q}(\n) \to \bbK_{q}
$$
such that $\tUptheta_Q\left(\sfy_{i,k}\right) = \sfx^Q_{i,k}$ for $(i,k) \in I \times \Z$.
\end{theorem}

\begin{proof}
The proof is almost identical with that of \cite[Theroem 10.4]{FHOO} (cf.~\cite[Theorem 7.3]{HL15}), where the above result is proved for simply-laced types.
\end{proof}

\subsection{Automorphisms of $\frakK_q$} For a Dynkin diagram $\Dynkin$, let $Q=(\Dynkin,\xi)$ and  $Q'=(\Dynkin,\xi')$ be Dynkin quivers whose underlying graphs
are the same as $\Dynkin$. Then we have an automorphism
$\tUptheta(Q',Q)$ of $ \bbK_{q}$ given by
\begin{align} \label{eq:tUptheta}
\tUptheta(Q',Q) \seteq   \tUptheta_{Q'} \circ \tUptheta_Q^{-1},
\end{align}
which is compatible with $\overline{( \cdot )}$ and $\frakD_q$.
Using the isomorphisms $\tPsi_Q$'s in \eqref{eq:tPsiQ} preserving $\tbfB^\up$, we define a bijection
$\psi_0: \calM_+^Q\to \calM_+^{Q'}$
such that
$$
\tPsi_{Q',Q}(L_q(m)) \seteq \tPsi_{Q'} \circ \tPsi_{Q}^{-1}(L_q(m)) = L_q(\psi_0(m))
$$
for $m \in \calM_+^Q$,
where $\tPsi_{Q',Q}$ coincides with the restriction of $\tUptheta(Q',Q)$ to $\bbK_{q,Q}$.
The bijection $\psi_0$ can be extended naturally to $\calM_+$ by the decomposition \eqref{eq: decomposition m} of $m \in \calM_+$, that is,
$$\psi(m)  \seteq \prod_{k \in \Z} \frakD^k( \psi_0(m_k)) \qquad \text{ for all } m \in \calM_+.$$
In particular, when
\begin{itemize}
    \item[(a)] $Q'=s_iQ$ for a Dynkin quiver $Q$ with a source $i$, and
    \item[(b)] $\phi_Q(k,p)=(\al_j,m)$,
\end{itemize}
we have
$$
\tUptheta(s_iQ,Q)(L_q(X_{k,p})) = \tUptheta(s_iQ,Q)(L^Q_q(\al_j,m)) \overset{\star}{=} L^{s_iQ}_q(\al_j,m)= L^{Q}_q(s_i(\al_j),m),
$$
where $\overset{\star}{=}$ holds due to Theorem \ref{thm: presentation} (recall \eqref{eq:LqQ al m}).

\begin{theorem} \label{thm: indeed frakKq}
The automorphism $\tUptheta(Q',Q)$ sends the canonical basis $\sfL_q$ to itself. Hence it induces an automorphism
on $\frakK_q$.
\end{theorem}

\begin{proof}
By using Theorem~\ref{thm: L_q},
one may prove our assertion following the proof of \cite[Theroem 10.6]{FHOO}, where the above is proved for simply-laced types.
\end{proof}

\subsection{Description of the automorphisms}
Now let us focus on the case when $Q' = s_iQ$ for a source $i$ of $Q$ and write
$$  \sigma_{i,Q} \seteq \tUptheta(s_iQ,Q)|_{\frakK_q}  \quad \text{ and } \quad  \sigma_{i,Q}^{-1} \seteq  \tUptheta(Q,s_iQ)|_{\frakK_q} .$$
To describe $\sigma_{i,Q}$ (resp. $\sigma^{-1}_{i,Q}$), it suffices to consider the images of $\sfx^Q_{j,m}$ under $\sigma_{i,Q}$ (resp. $\sfx^{s_iQ}_{j,m}$ under $\sigma^{-1}_{i,Q}$).

\begin{proposition} \label{prop: ext of Lustig} Let $Q$ be a Dynkin quiver with a source $i$, and hence $s_iQ$ is a Dynkin quiver with a sink $i$.
Then the automorphisms $\sigma_{i,Q}^{\pm 1}$ can be described as follows$:$
\begin{equation} \label{eq:brd1p}
\begin{aligned}
\sigma_{i,Q}(\sfx^Q_{j,m}) =
  \bc
\sfx^Q_{j,m+\delta_{i,j}}  &  \text{if} \  \sfc_{i,j} \ge 0, \\[1ex]
\dfrac{ \left( \displaystyle \sum_{k=0}^{- \sfc_{i,j}} (-1)^{k} q_i^{ - \sfc_{i,j}/2 -k  } (\sfx^Q_{i,m})^{(k)}\sfx^Q_{j,m}  (\sfx^Q_{i,m})^{(- \sfc_{i,j}-k)}  \right)}{ (q_i -q_i^{-1})^{- \sfc_{i,j}}  }  &\text{if} \  \sfc_{i,j} < 0,  \ec
 \end{aligned}
\end{equation}
and
\begin{equation}  \label{eq:invbrd1p}
\begin{aligned}
\sigma_{i,Q}^{-1}(\sfx^{s_iQ}_{j,m}) =
  \bc
\sfx^{s_iQ}_{j,m-\delta_{i,j}}  &  \text{if} \  \sfc_{i,j} \ge 0, \\[1ex]
\dfrac{ \left( \displaystyle \sum_{k=0}^{-\sfc_{i,j}} (-1)^{k} q_i^{ -\sfc_{i,j}/2 -k  } (\sfx^{s_iQ}_{i,m})^{(-\sfc_{i,j}-k)} \sfx^{s_iQ}_{j,m}  (\sfx^{s_iQ}_{i,m})^{(k)}  \right)}{ (q_i -q_i^{-1})^{-\sfc_{i,j} }  }  &\text{if} \  \sfc_{i,j} < 0,
\ec
 \end{aligned}
\end{equation}
where
$$
(\sfx^{Q'}_{i,m})^{(k)}  \seteq \dfrac{(\sfx^{Q'}_{i,m})^{k} }{[k]_{i}!} \quad \text{ for any Dynkin quiver $Q'$}.
$$
\end{proposition}

Before we see the proof, let us recall the $\Q(q^{1/2})$-algebra automorphisms $T''_{i,-1}, T'_{i,1} : \calU_q(\g) \to\calU_q(\g) $ which are inverse to each other (see \cite[Section 37.1.3]{LusztigBook} and also~\cite{Saito94}):
\begin{align*}
&T''_{i,-1}(q^h)= q^{s_i(h)},\qquad\qquad\qquad \   \qquad\qquad \ T'_{i,1}(q^h) = q^{s_i(h)}, \\
&  T''_{i,-1}(e_i)= -f_i K_i^{-1},  \  T''_{i,-1}(f_i)= -K_i e_i, \quad   T'_{i,1}(e_i)= - K_i f_i, \ T'_{i,1}(f_i)= -e_iK_i^{-1}, \\
& T''_{i,-1}(e_j) = \sum_{r+s = -\sfc_{i,j}} (-1)^r q_i^r e_i^{(s)}e_j e_i^{(r)}, \quad  \quad T'_{i,1}(e_j) = \sum_{r+s = -\sfc_{i,j}} (-1)^r q_i^r e_i^{(r)}e_j e_i^{(s)}     \ \  \quad \text{ for } i \ne j, \\
& T''_{i,-1}(f_j) = \sum_{r+s = -\sfc_{i,j}} (-1)^r q_i^{-r} f_i^{(r)}f_j f_i^{(s)}, \quad T'_{i,1}(f_j) = \sum_{r+s = -\sfc_{i,j}} (-1)^r q_i^{-r} f_i^{(s)}f_j f_i^{(r)}   \quad \text{ for } i \ne j.
\end{align*}

\begin{remark} \label{rmk: up ro constant}
Note that when $i \ne j$, the action
$\sigma_{i,Q}(\sfx^Q_{j,m})$ (resp. $\sigma^{-1}_{i,Q}(\sfx^{s_iQ}_{j,m})$) coincides with the $T''_{i,-1}$-operation (resp. $T'_{i,1}$-operation) on $f_j$ $(i \ne j)$ (\cite{L902,Saito94,LusztigBook}) up to a constant in $\Q[[q^{\pm 1/2}]] \setminus \{ 0 \}$.
Note also that
the formulas~\eqref{eq:brd1p} and~\eqref{eq:invbrd1p} for {\rm ADE}-cases and their proofs are announced in~\cite{KKOP21A} where a different approach is taken.
From now on, we write
$$T_i \seteq T^{''}_{i,-1} \quad \text{ and } \quad T_i^{-1} \seteq T^{'}_{i,1}$$ for notational simplicity.
\end{remark}

\begin{proof} [Proof of {\rm Proposition~\ref{prop: ext of Lustig}}]
By Theorem \ref{thm: presentation} and \eqref{eq:tUptheta}, we have
\begin{align} \label{eq:expression of tUptheta sfxjmQ}
    \tUptheta(s_iQ,Q)(\sfx_{j,m}^Q)
    =
    \tUptheta_{s_iQ} \circ \tUptheta_Q^{-1} ( \sfx_{j,m}^Q )
    =
    \tUptheta_{s_iQ} ( \sfy_{j,m} )
    =
    \sfx_{j,m}^{s_iQ},
\end{align}
which is equivalent to
\begin{align} \label{eq:another expression}
    \tUptheta(s_iQ,Q)(L_q^Q(\al_j,m)) = L_q^{s_iQ}(\al_j,m) = L_q^{Q}(s_i(\al_j),m).
\end{align}
(cf.~\eqref{eq:LqQ al m} and \eqref{eq:sfxQik}).
In what follows, we shall investigate how $\sfx_{j,m}^{s_iQ}$ can written in terms of $\sfx_{j,n}^{Q}$ for $n \in \Z$, and obtain \eqref{eq:brd1p} and \eqref{eq:invbrd1p}.
Since $\tUptheta(s_iQ,Q)$ is compatible with $\frakD_q$, it is enough to consider the case of $\sigma_{i,Q}( \sfx_{j,0}^Q) = \sfx_{j,0}^{s_iQ}$ for any $j \in I$.

In the case of \eqref{eq:brd1p},
the $i$-th simple root $\al_i$ is a minimal element of $\Phi_+$ with respect to $\prec_{[Q]}$.
The assertion for $i=j$ follows from the rule \eqref{eq: Qd properties}~\eqref{it: al_i poisition}.
Thus let us consider three cases when $i \ne j$ as follows:
\smallskip

\noindent
{\it Case 1}. Assume $\sfc_{i,j}=-1$. In this case $\pair{\al_i,\al_j}$ is a $[Q]$-minimal $[Q]$-pair for $\al_i+\al_j$ and $p_{\al_j,\al_i}=0$.
From the viewpoint of Corollary \ref{cor: HLO iso}, we interpret the commutation relation \eqref{eq: BKMcp} as follows:
\begin{align*}
&\dfrac{ q_i^{-\sfc_{i,j}/2}L_q^Q(\al_j,0)*L_q^Q(\al_i,0) -q_i^{\sfc_{i,j}/2 }L_q^Q(\al_i,0)*L_q^Q(\al_j,0)}{(q_i^{-\sfc_{i,j}}-q_i^{\sfc_{i,j}} )}
=   L_q^Q(\al_i+\al_j,0) = L_q^{s_iQ}(\al_j,0)
\end{align*}
By \eqref{eq:expression of tUptheta sfxjmQ} and \eqref{eq:another expression}, we conclude that $\sigma_{i,Q}(\sfx^Q_{j,0}) =   \sfx^{s_iQ}_{j,0}$ is written as \eqref{eq:brd1p} in this case.
\smallskip

\noindent
{\it Case 2}. Assume $\sfc_{i,j}=-2$. Note that,  when $\g$ is of type $B_n$ (resp. $C_n$),   $\pair{\al_n,\al_{n-1}+\al_{n}}$  (resp. $\pair{\al_{n-1},\al_{n-1}+\al_{n}}$) is a $[Q]$-minimal [Q]-pair,
since it is the only pair for $\al_{n-1}+2\al_n$ (resp. $2\al_{n-1}+\al_n$).
When $\g$ is of type $F_4$, one can check that  $\pair{\al_3,\al_{2}+\al_{3}}$ is also  a $[Q]$-minimal [Q]-pair for any Dynkin quiver with source $3$.   Since the proofs are similar, we assume that $\g=B_n$. In this case, we have
$$p_{\al_{n-1}+\al_{n},\al_n}=1 \qquad \text{ and } \qquad   (\al_{n-1}+\al_{n},\al_n) = 0.$$
From the viewpoint of Corollary \ref{cor: HLO iso},
the commutation relation \eqref{eq: BKMc2} tells us that
$$
\dfrac{L_q^Q(\al_n+\al_{n-1},0)*L_q^Q(\al_n,0) - L_q^Q(\al_n,0)*L_q^Q(\al_n+\al_{n-1},0)}{q-q^{-1}} =L_q^Q(2\al_n+\al_{n-1},0)=L_q^{s_nQ}(\al_{n-1},0).
$$
Since
$$
L_q^Q(\al_n+\al_{n-1},0) = \dfrac{ qL_q^Q(\al_{n-1},0)*L_q^Q(\al_n,0) -q^{-1}L_q^Q(\al_{n-1},0)*L_q^Q(\al_j,0)}{(q^{2}-q^{-2} )},
$$
we have
\begin{align*}
& L_q^{s_nQ}(\al_{n-1}) =  L_q^Q(2\al_n+\al_{n-1}) = \\
&\dfrac{qL_q^Q(\al_{n-1})*L_q^Q(\al_n)^{*2} - (q+q^{-1})L_q^Q(\al_n)*L_q^Q(\al_{n-1})*L_q^Q(\al_n) +q^{-1}L_q^Q(\al_n)^{*2}*L_q^Q(\al_{n-1})  }{(q-q^{-1})(q^2-q^{-2})},
\end{align*}
where $L_q^{s_nQ}(\be) \seteq L_q^{s_nQ}(\be,0)$.
Hence $\sigma_{n,Q}(\sfx^Q_{n-1,0}) =   \sfx^{s_nQ}_{n-1,0}$ gives \eqref{eq:brd1p} in this case.
\smallskip

\noindent
{\it Case 3}. The case $\sfc_{i,j}=-3$ is taken care of in Appendix~\ref{app: G}.
\medskip

In the case of \eqref{eq:invbrd1p}, the $i$-simple root $\al_i$  is a maximal element of $\Phi_+$ with respect to $\prec_{[s_iQ]}$. Then we apply the similar arguments for showing \eqref{eq:brd1p} to \eqref{eq:invbrd1p}, which implies that
\[
\sigma_{i,Q}^{-1}(\sfx^{s_iQ}_{j,0}) =  L_q^{s_iQ}(s_i(\al_j))  =  \sfx^{Q}_{j,0}.
\]
Thus $\sigma_{i,Q}^{-1}$ gives a description of $\tUptheta(Q,s_iQ)$, which is the inverse of $\sigma_{i,Q}$.
\end{proof}

\begin{corollary}
For $i \in I$ and $\be \in \Phi_+$ with $i \not\in {\rm supp}(\be)$, we have
$$\sigma_{i,Q}\big( L_q^Q(\be,m) \big)= L_q^Q(s_i(\be),m) \quad \text{ and } \quad \sigma_{i,Q}^{-1}\big(  L_q^{s_iQ}(\be,m) \big)= L_q^{s_iQ}(s_i(\be),m),$$
where $Q$ is a Dynkin quiver with source $i$.
\end{corollary}
\begin{proof}
    Recall that ${\rm supp}(\be) = \{  i \in I \ | \  k_i \ne 0 \}$ for $\be = \sum_{i \in I} k_i\al_i \in \rl^+$.
    Our assertion follows directly from Theorem~\ref{thm: minimal pair dual pbw} and Proposition~\ref{prop: ext of Lustig} (see {\it Case 1}--{\it Case 3} in the proof of Proposition~\ref{prop: ext of Lustig}).
\end{proof}

\section{Braid group action} \label{Sec: Braid}
For each $\Dynkin$, the braid group $B(\Dynkin)$ (or \emph{Artin--Tits} group) is generated by  $T^{\pm}_i$ $(i \in \Dynkin_0)$ satisfying
\begin{align} \label{eq:braid relations}
\begin{split}
\underbrace{T^{\pm}_iT^{\pm}_j \cdots}_{\text{ $h_{i,j}$-times}} & =\underbrace{T^{\pm}_jT^{\pm}_i \cdots}_{\text{ $h_{i,j}$-times}} \qquad \text{for all $i \ne j \in \Dynkin_0$,}
\end{split}
\end{align}
where $h_{i,j}$ is defined in \eqref{eq:def of hij}.
In this section, we shall prove that the automorphisms $\sigma_{i,Q}^{\pm 1}$ induce a braid group action on $\hA_{q}(\n)$ and hence on $\frakK_q$ by Theorem~\ref{thm: indeed frakKq}.

\subsection{Braid group symmetry on $\hA_{q}(\n)$}
By Theorem \ref{thm: presentation} and Proposition~\ref{prop: ext of Lustig}, we define automorphisms $\TT_i^{\pm1}$ on $\hA_{q}(\n)$
by the following:
\begin{equation} \label{eq:brd1pY}
\begin{aligned}
\TT_{i}(\sfy_{j,m}) =
  \bc
\sfy_{j,m+\delta_{i,j}}  &  \text{if} \  \sfc_{i,j} \ge 0, \\[1ex]
\dfrac{ \left( \displaystyle \sum_{k=0}^{- \sfc_{i,j}} (-1)^{k} q_i^{ - \sfc_{i,j}/2 -k  } \sfy_{i,m}^{(k)}\sfy_{j,m}  \sfy_{i,m}^{(- \sfc_{i,j}-k)}  \right)}{ (q_i -q_i^{-1})^{- \sfc_{i,j}}  }  &\text{if} \  \sfc_{i,j} < 0,  \ec
 \end{aligned}
\end{equation}
and
\begin{equation}  \label{eq:invbrd1pY}
\begin{aligned}
\TT_{i}^{-1}(\sfy_{j,m}) =
  \bc
\sfy_{j,m-\delta_{i,j}}  &  \text{if} \  \sfc_{i,j} \ge 0, \\[1ex]
\dfrac{ \left( \displaystyle \sum_{k=0}^{-\sfc_{i,j}} (-1)^{k} q_i^{ -\sfc_{i,j}/2 -k  } \sfy_{i,m}^{(-\sfc_{i,j}-k)} \sfy_{j,m}  \sfy_{i,m}^{(k)}  \right)}{ (q_i -q_i^{-1})^{-\sfc_{i,j} }  }  &\text{if} \  \sfc_{i,j} < 0,
\ec
 \end{aligned}
\end{equation}
where
$
\sfy_{i,m}^{(k)}  \seteq \dfrac{\sfy_{i,m}^{k} }{[k]_{i}!}.
$
Now we are in a position to state our final main result of this paper.

\begin{theorem} \label{thm: braid main}
The automorphisms $\{ \TT^{\pm 1}_i \ | \ i \in I\}$ give a braid group $B(\Dynkin)$ action on $\hA_{q}(\n)$ and hence on $\frakK_q$.
\end{theorem}
\begin{proof}
We give the proof of Theorem \ref{thm: braid main} in Section \ref{subsec:proof of final main result}.
\end{proof}
\noindent
 The rest of this subsection will be devoted to present consequences of Theorem \ref{thm: braid main} and its proof.
 \medskip

For $m \in \Z$, we set
$$\hA_{q}(\n)_m \seteq\lan \sfy_{i,m} \ | \ i \in \Dynkin_0 \ran \subset \hA_{q}(\n).$$
Note that $\hA_{q}(\n)_m$ is isomorphic to $\calU_q^-(\g)$ via the $\Q(q^{1/2})$-algebra isomorphism $\varsigma_m$ such that  $\varsigma_m(\sfy_{i,m})=f_i$.
From the proofs of the propositions in Section \ref{subsec:proof of final main result}, we obtain the following corollary:

\begin{corollary}  \label{cor: EiEjspan}
For $i,j \in I$ and $0 \le p < h_{i,j}$, we have
$$
\underbrace{\cdots \TT_{i}\TT_{j}}_{\text{$p$-times}}(\sfy_{i}), \ \underbrace{\cdots \TT_{j}\TT_{i}}_{\text{$p$-times}}(\sfy_{j}) \in   \lan \sfy_{i,m},\sfy_{j,m}  \ran \subseteq
 \hA_{q}(\n)_m.
$$
\end{corollary}

 For $w \in \weyl$ with $(i_1, \dots, i_r) \in I(w)$, set
\begin{align*}
    \TT^{\pm 1}_w \seteq \TT^{\pm 1}_{i_1} \dots \TT^{\pm 1}_{i_r}.
\end{align*}
Note that $\TT^{\pm 1}_w$ is well-defined by Theorem \ref{thm: braid main}.

\begin{proposition} \label{prop: stable and up to constant}
For $w \in \weyl$ and $\ell(ws_{i})=\ell(w)+1$. Then we have the following:
\ben
\item \label{it: stable} $  \TT_{w}( \sfy_{i,m} ) \in \hA_{q}(\n)_m.$
\item \label{it: up to constant} There exists $\bfc \in \Q[[q^{1/2}]] \setminus \{0\}$ such that
$\varsigma_m(\TT^{\pm}_{w}( \sfy_{i,m} )) = \bfc T^{\pm}_w(f_i).$
\ee
\end{proposition}

\begin{proof}
By the standard argument, there exists $j \in I$ such that $\ell(ws_j)=\ell(w)-1$ and a unique $w' \in \weyl$ such that
$$\ell(w's_j)=\ell(w')+1,\quad \ell(w's_i)=\ell(w')+1 \quad \text{ and } \quad w = w'y$$
whether either
\bna
\item $y=\underbrace{s_is_j \cdots s_j}_{\text{$(h_{i,j}> p \ge 2)$-factors}}$ and $\ell(w)=\ell(w')+p$, or
\item $y=\underbrace{s_js_i \cdots s_j}_{\text{$(h_{i,j}> p \ge 1)$-factors}}$ and $\ell(w)=\ell(w')+p$.
\ee

With this observation, the first assertion now follows from an induction argument on $\ell(w)$ and Corollary~\ref{cor: EiEjspan} (see \cite[40.1.2]{LusztigBook}).
Next, Remark~\ref{rmk: up ro constant} implies that the second assertion holds for $w$ with $\ell(w)=1$. In Propositions~\ref{prop: 3 move in simply-laced},~\ref{prop: 4-move dobly-laced}
and,~\ref{prop: G2 barid}, we will see that
$$
\varsigma_m(\TT^{\pm}_{y}( \sfy_{i,m} )) = \bfc T^{\pm}_y(f_i).
$$
Then the second assertion also follows from an induction argument on the length and~\eqref{it: stable}.
\end{proof}

\begin{corollary} \label{cor: roots}\hfill
\bna
\item \label{it: a-red}
For a reduced expression $\uw=s_{i_1}\cdots s_{i_{r-1}}s_{i_r}$ of $w$, if $s_{i_1}\cdots s_{i_{r-1}}(\al_{i_r})=\al_j$, then
$$  \TT^\pm_{i_1}\cdots \TT^\pm_{i_{r-1}}( \sfy_{i_r,m} ) = \sfy_{j,m}.$$
\item \label{it: b-longest} For $w_\circ \in \weyl$, we have
$$
\TT^\pm_{w_\circ}(\sfy_{i,m}) =\sfy_{i^*,m\pm1} \in \hA_{q}(\n)_{m\pm1}.
$$
\ee
\end{corollary}

\begin{proof} In this proof, we only consider $\TT_i$-case, since the proof for  $\TT^{-}_i$-case is similar.

\smallskip

\noindent
\eqref{it: a-red} Assigning $(-1)^{m}\al_i$ for the weight of $\sfy_{i,m}$, the algebra $\hA_{q}(\n)$ is a $\Phi$-graded algebra. With this $\Phi$-grading, the automorphism $\TT_i$ sends a $\be$-homogeneous element to a $s_i(\be)$-homogeneous element. Since $(-1)^{m}\al_j$-graded space in $ \hA_{q}(\n)_m$ is $1$-dimensional and spanned by $\sfy_{j,m}$, our assertion follows.

\smallskip

\noindent
\eqref{it: b-longest} Note that there exists $\ii_\circ = (i_1,\ldots, i_\ell)\in I(w_\circ)$ such that $i_\ell=i$. Then we have $s_{i_1} \cdots s_{i_{\ell-1}}(\al_i)=\al_{i^*}$ and hence
$$  \TT_{i_1} \cdots \TT_{i_{\ell-1}}(\sfy_{i,m}) = \sfy_{i^*,m}  $$
by~\eqref{it: a-red}. Thus
$$  \TT_{i_1} \cdots \TT_{i_{\ell-1}}\TT_{i_\ell}(\sfy_{i,m})=\TT_{i_1} \cdots \TT_{i_{\ell-1}}(\sfy_{i,m+1}) =\sfy_{i^*,m+1},$$
as desired.
\end{proof}

Note that, for $\ii_\circ =(i_1,i_2,\ldots,i_\ell) \in I(w_\circ)$, the set
$$ \bbP_{\ii_\circ,m} \seteq \{ \TT_{i_1}\cdots  \TT_{i_{k-1}}(\sfy_{i_k,m}) \ | \ 1 \le k \le \ell \} \text{ is linearly independent and
$|\bbP_{\ii_\circ,m}|=\ell$.}$$
For $\ii_\circ=(j_1,j_2,\ldots,j_\ell) \in I(w_\circ)$, we define $\widetilde{\ii}_0 = (i_k)_{k \in \Z_{>0}}$ as follows:
$$   i_k = j_k  \text{ for } 1\le k \le \ell \quad  \text{ and } \quad j_{k+\ell}=j_k^*.$$
Then Corollary~\ref{cor: roots} implies that
$$ \{ \TT_{i_1}\cdots  \TT_{i_{k-1}}(\sfy_{i_k,m}) \ | \  k \in \Z_{>0}  \}  \sqcup  \{ \TT^-_{i_1}\cdots  \TT^-_{i_{k-1}}(\sfy_{i_k,m-1}) \ | \  k \in \Z_{>0}  \}    \text{ is linearly independent}$$
for any $m\in\Z$.

\subsection{Proof of Theorem \ref{thm: braid main}} \label{subsec:proof of final main result}
We investigate the braid relations \eqref{eq:braid relations} on $\left\{  \TT_i \, | \, i \in I \,\right\}$ in the next five propositions.

\begin{proposition} \label{prop: commutation}
Assume $\sfc_{i,j}=0$. Then we have
$$\TT_i\TT_j   = \TT_j\TT_i.$$
\end{proposition}

\begin{proof}
We verify $\TT_i\TT_j(\sfy_{k,m}) = \TT_j\TT_i(\sfy_{k,m})$ for any pair $(k,m) \in I \times \Z$.
In the case of $\sfc_{i,k}\ge 0$ and $\sfc_{j,k} \ge 0$,
it is straightforward to check it.
Let us consider the other case, that is, $\sfc_{i,k}<0$ or $\sfc_{j,k}<0$.
Since $\sfc_{i,j}=0$, it allows us to take a Dynkin quiver $Q$ with sources $i$ and $j$.
By Theorem \ref{thm: presentation} and Proposition~\ref{prop: ext of Lustig}, we have
\begin{align*}
\tUptheta_Q \circ \TT_i\TT_j (\sfy_{k,m}) = L^Q_q(s_is_j(\al_k),m) =  \tUptheta_Q \circ \TT_j\TT_i(\sfy_{k,m}).
\end{align*}
Since $\tUptheta_Q$ is invertible, we conclude $\TT_i\TT_j (\sfy_{k,m}) = \TT_j\TT_i(\sfy_{k,m})$.
This completes the proof.
\end{proof}

\begin{proposition} \label{prop: 3 move in simply-laced}
For a simply-laced $\g$, assume $\sfc_{i,j}=-1$. Then we have
$$\TT_i\TT_j\TT_i   = \TT_j\TT_i\TT_j.$$
\end{proposition}

\begin{proof}
We claim that $\TT_i\TT_j\TT_i(\sfy_{k,m})   = \TT_j\TT_i\TT_j(\sfy_{k,m})$ for any pair $(k,m) \in I \times \Z$.
If $i \not\sim k$ and $j \not\sim k$, that is, $k$ is not adjacent to both $i$ and $j$ on the Dynkin diagram, then we have $\TT_i\TT_j\TT_i(\sfy_{k,m})   = \TT_j\TT_i\TT_j(\sfy_{k,m})$ by definition.
Let us consider the other two cases:
\smallskip

\noindent
{\it Case 1}. Assume $k=i$ or $k=j$.
We prove the case of $k=i$ here, since the proof for the other case is quite similar.
Let $Q$ be a Dynkin quiver with  source $i$.
By similar computation to the proof of Proposition~\ref{prop: ext of Lustig}, we have
\begin{equation*}
\begin{aligned}
\tUptheta_Q \circ \TT_i\TT_j (\sfy_{i,m}) & =  \tUptheta_Q \circ  \TT_i  \left(   \dfrac{ q^{1/2}\sfy_{i,m}\sfy_{j,m} -q^{-1/2} \sfy_{j,m} \sfy_{i,m} }{q-q^{-1}}  \right) \\
& \hspace{-9ex}= \left(   \dfrac{ q^{1/2}L_q^{Q}(\al_i,m+1)*L_q^{Q}(\al_i+\al_j,m) -q^{-1/2} L_q^{Q}(\al_i+\al_j,m)* L_q^{Q}(\al_i,m+1) }{q-q^{-1}}  \right) \\
& \hspace{-9ex} = \left(   \dfrac{ q^{1/2}L_q^{s_i Q}(\al_i,m)*L_q^{s_iQ}(\al_j,m) -q^{-1/2} L_q^{s_i Q}(\al_j,m) *L_q^{s_i Q}(\al_i,m) }{q-q^{-1}}  \right)  \\
&\hspace{-9ex}   = L_q^{s_iQ}(\al_i+\al_j,m) = L_q^{Q}(\al_j,m),
\end{aligned}
\end{equation*}
where $\tUptheta_Q \circ  \TT_i (\sfy_{i,m}) = L_q^Q(\al_i, m+1) = L_q^{s_i Q}(\al_i,m)$ and $\tUptheta_Q \circ  \TT_i (\sfy_{j,m}) = L_q^Q(\al_i+\al_j,m) = L_q^{s_iQ}(\al_j,m)$.
This implies
\begin{align} \label{eq: ij(i,m)=(j,m)}
\TT_i\TT_j (\sfy_{i,m}) = \sfy_{j,m},
\end{align}
which prove the current case because \eqref{eq: ij(i,m)=(j,m)} yields
\begin{align*}
& \TT_i\TT_j\TT_i (\sfy_{i,m}) = \TT_i\TT_j(\sfy_{i,m+1}) = \sfy_{j,m+1}, \qquad  \TT_j\TT_i\TT_j (\sfy_{i,m}) = \TT_j(\sfy_{j,m})= \sfy_{j,m+1}.
\end{align*}
\smallskip

\noindent
{\it Case 2}.
Assume $i \sim k$, \, $j \not\sim k$ or $i \not\sim k$, $j \sim k$.
We only consider the former case, since the other case is proved in a similar way.
Under the assumption of $i \sim k$ and  $j \not\sim k$,
we have
\begin{align}
\TT_i\TT_j\TT_i (\sfy_{k,m})
& =
\left( \dfrac{q^{1/2}\TT_i(\sfy_{k,m}) \sfy_{j,m} - q^{-1/2}\sfy_{j,m} \TT_i(\sfy_{k,m})}{q-q^{-1}} \right),  \label{eq: st1} \allowdisplaybreaks\\
\TT_j\TT_i\TT_j (\sfy_{k,m}) &
= \left( \dfrac{ q^{1/2} \sfy_{k,m}\TT_j(\sfy_{i,m}) - q^{1/2}\TT_j(\sfy_{i,m})\sfy_{k,m} }{q-q^{-1}} \right), \label{eq: st2}
\end{align}
where \eqref{eq: st1} follows from \eqref{eq: ij(i,m)=(j,m)}.
To see the equality between ~\eqref {eq: st1} and~\eqref{eq: st2},
it is enough to consider the case of $\g=A_3$, $k=1$, $i=2$, $j=3$, and $m=0$.
Set $y_t\seteq \sfy_{t,0}$ for $t=1,2,3$.
Under this setting,
~\eqref{eq: st1} and \eqref{eq: st2} become
\begin{align*}
& \dfrac{q^{1/2}\TT_2(y_{1}) y_{3} - q^{-1/2}y_{3} \TT_2(y_{1})}{q-q^{-1}}
 =  \left(\dfrac{ qy_{1}y_{2}y_{3}- y_{2}y_{1}y_{3}- y_{3}y_{1}y_{2}+ q^{-1}y_{3}y_{2}y_{1} }{(q-q^{-1})^2} \right), \\
& \dfrac{q^{1/2} y_{1}\TT_3(y_{2}) - q^{1/2}\TT_3(y_{2})y_{1} }{q-q^{-1}}
= \left(       \dfrac{qy_{1}y_{2}y_{3}-y_{1}y_{3}y_{2} -y_{2}y_{3}y_{1}+q^{-1}y_{3}y_{2}y_{1} }{(q-q^{-1})^2} \right),
\end{align*}
respectively.
Since $y_1y_3=y_3y_1$ by the defining relation of $\hA_{q}(\n)_m$,
we prove our claim.
\end{proof}

\begin{proposition} \label{prop: 3 move in C}
For $\g$  of type $B_n$, $C_n$ or $F_4$, assume $\sfc_{j,i}\sfc_{i,j}=1$. Then we have
$$\TT_{i}\TT_{j}\TT_{i}   = \TT_{j}\TT_{i}\TT_{j}.$$
\end{proposition}

\begin{proof} Since the proofs for $B_n$, $C_n$ and $F_4$ are all similar, we only give the proof when $\g=C_n$.
The case of $i, j \in I \setminus \{ n-2, n-1 \}$ with $\sfc_{j,i}\sfc_{i,j}=1$ follows from Proposition \ref{prop: 3 move in simply-laced}, so
it is enough to prove that
$$\TT_{n-2}\TT_{n-1}\TT_{n-2}(\sfy_{n,m})=\TT_{n-1}\TT_{n-2}\TT_{n-1}(\sfy_{n,m}).$$
We see that
\fontsize{10}{10}
\begin{equation} \label{eq: stepp0}
\begin{aligned}
& \TT_{n-2}\TT_{n-1}\TT_{n-2}(\sfy_{n,m}) =  \TT_{n-2}\TT_{n-1}(\sfy_{n,m}) \\
& =
\dfrac{ q\sfy_{n,m}  \TT_{n-2}(\sfy_{n-1,m})^2 - (q+q^{-1})\TT_{n-2}(\sfy_{n-1,m})\sfy_{n,m} \TT_{n-2}(\sfy_{n-1,m}) + q^{-1} \TT_{n-2}(\sfy_{n-1,m})^2\sfy_{n,m}  }
{(q-q^{-1})(q^2-q^{-2})},
\end{aligned}
\end{equation}
\fontsize{11}{11}
while
\begin{equation} \label{eq: stepp1}
\begin{aligned}
&\TT_{n-1}\TT_{n-2}\TT_{n-1}(\sfy_{n,m}) = \TT_{n-1}\TT_{n-2} \left( \dfrac{q\sfy_{n,m}\sfy_{n-1,m}^{(2)} - \sfy_{n-1,m}\sfy_{n,m}\sfy_{n-1,m} + q^{-1}\sfy_{n-1,m}^{(2)}\sfy_{n,m}}{(q-q^{-1})^2}\right)\\
& \overset{\dagger}{=}
\dfrac{q\TT_{n-1}(\sfy_{n,m})\sfy_{n-2,m}^{ 2} -(q+q^{-1}) \sfy_{n-2,m}\TT_{n-1}(\sfy_{n,m})\sfy_{n-2,m} + q^{-1}\sfy_{n-2,m}^{2}\TT_{n-1}(\sfy_{n,m})}
{(q-q^{-1})(q^2-q^{-2}) },
\end{aligned}
\end{equation}
where $\overset{\dagger}{=}$ follows from~\eqref{eq: ij(i,m)=(j,m)}.
To see the equality between ~\eqref{eq: stepp0} and~\eqref{eq: stepp1}, it suffices to assume that $n=3$ and $m=0$. Set $y_i\seteq \sfy_{i,0}$ for $i=1,2,3$ and recall that  $q_1=q_2=q$, and  $q_3=q^2$.
By multiplying the numerators of \eqref{eq: stepp0} and \eqref{eq: stepp1} by $(q-q^{-1})(q^2-q^{-2})$, the first one becomes
\begin{align} \label{eq: stepp00}
\begin{split}
&  \big(qy_{3}  \TT_{1}(y_{2})^2 - (q+q^{-1})\TT_{1}(y_{2})y_{3} \TT_{1}(y_{2}) + q^{-1} \TT_{1}(y_{2})^2y_{3}  \big) \times (q-q^{-1})(q^2-q^{-2})  \\
& = \left\{   q^2y_{3}y_{2}y_{1}y_{2}y_{1}- qy_{3}y_{1}y_{2}^2y_{1} - qy_{3}y_{2}y_{1}^2y_{2}  +y_{3}y_{1}y_{2}y_{1}y_{2}       \right.   \allowdisplaybreaks \\
&\hspace{5ex}   -(q+q^{-1})         (q  y_{2}y_{3}y_{1}y_{2}y_{1} -  y_{1}y_{2}y_3y_{2}y_{1}    -  y_{2}y_{1}^2y_3y_{2}+q^{-1}y_{1}y_{2}y_{1}y_{3}y_{2})    \allowdisplaybreaks \\
&\hspace{15ex} \left. +   y_{2}y_{1}y_{2}y_{1}y_{3} -q^{-1}  y_{1}y_{2}^2y_{1}y_{3} - q^{-1}y_{2}y_{1}^2y_{2}y_{3} + q^{-2}y_{1}y_{2}y_{1}y_{2}y_{3}  \right\}  \times   (q+q^{-1})
  \allowdisplaybreaks \\
\end{split}
\end{align}
and the second one becomes
\begin{align} \label{eq: stepp11}
\begin{split}
& \big( q\TT_{2}(y_{3})y_{1}^{2} -(q+q^{-1}) y_{1}\TT_{2}(y_{3})y_{1} + q^{-1}y_{1}^{2}\TT_{2}(y_{3}) \big) \times  (q-q^{-1})(q^2-q^{-2})  \\
& = \left\{ q^2y_{3}y_{2}^{2}y_{1}^{2} - (q^2+1)y_{2}y_{3}y_{2}y_{1}^{2} + y_{2}^{2}y_{3}y_{1}^{2} \right.  \allowdisplaybreaks\\
&\hspace{5ex}   - (q^2+1)y_{1}y_{3}y_{2}^{2}y_{1} +(q^2+2+q^{-2})y_{1}y_{2}y_{3}y_{2}y_{1} - (1+q^{-2})y_{1}y_{2}^{2}y_{3}y_{1}    \allowdisplaybreaks\\
&\hspace{15ex} \left.  +y_{1}^2y_{3}y_{2}^{2} - (1+q^{-2})y_{1}^2y_{2}y_{3}y_{2} + q^{-2}y_{1}^2y_{2}^{2}y_{3} \right\}.  \allowdisplaybreaks\\
\end{split}
\end{align}
Now, we claim that ~\eqref{eq: stepp00} and ~\eqref{eq: stepp11} coincide with each other.
By the quantum Serre relations~\eqref{eq: q-serre y}, we have
\begin{eqnarray} &&
\parbox{90ex}{
\bna
\item  \label{it: C12}  $y_1^2y_2 - (q+q^{-1})y_1y_2y_1 + y_2y_1^2 = 0 \iff (q+q^{-1})y_1y_2y_1 =  y_1^2y_2+ y_2y_1^2 $,
\item \label{it: C21}  $y_2^2y_1 - (q+q^{-1})y_2y_1y_2 + y_1y_2^2 = 0 \iff (q+q^{-1})y_2y_1y_2 = y_2^2y_1+ y_1y_2^2$.
\ee
}\label{eq: C3 q-serre}
\end{eqnarray}
Using~\eqref{eq: C3 q-serre} ~\eqref{it: C12}, ~\eqref{it: C21} and the relation $y_1y_2^2y_1 = y_2y_1^2y_2$,
the equality of \eqref{eq: stepp00} and \eqref{eq: stepp11} follows from the computation as shown below:
\begin{align*}
&\left\{   q^2y_{3}y_{2} \underline{y_{1}y_{2}y_{1}}- qy_{3}y_{1}y_{2}^2y_{1} - qy_{3}y_{2}y_{1}^2y_{2}  +y_{3}y_{1} \underline{y_{2}y_{1}y_{2}}       \right.   \nonumber\allowdisplaybreaks\\
&\hspace{5ex}   -(q+q^{-1})         (q  y_{2}y_{3} \underline{y_{1}y_{2}y_{1}} -  y_{1}y_{2}y_3y_{2}y_{1}    -  y_{2}y_{1}^2y_3y_{2}+q^{-1} \underline{y_{1}y_{2}y_{1}}y_{3}y_{2})    \nonumber\allowdisplaybreaks \\
&\hspace{15ex} \left. +   y_{2} \underline{y_{1}y_{2}y_{1}}y_{3} -q^{-1}  y_{1}y_{2}^2y_{1}y_{3} - q^{-1}y_{2}y_{1}^2y_{2}y_{3} + q^{-2} \underline{y_{1}y_{2}y_{1}}y_{2}y_{3}  \right\}  \times  \underline{(q+q^{-1})}  \\
&=    q^2y_{3}y_{2}^2y_{1}^2 - (q^2+1) y_{2}y_{3}y_{2}y_{1}^2 +     y_{2}^2y_{3}y_{1}^2            \nonumber\allowdisplaybreaks\\
&\hspace{5ex}  -  (q^2 +1)y_{1}y_{3}y_{2}^2y_{1}  + (q^2+2+q^{-2})y_{1}y_{2}y_{3}y_{2}y_{1}  -(1+q^{-2})  y_{1}y_{2}^2y_{3}y_{1}         \nonumber\allowdisplaybreaks \\
& \hspace{15ex}  +y_{1}^2y_{3}y_{2}^2 - (1+q^{-2}) y_{1}^2y_{2}y_{3}y_{2} +  q^{-2}y_{1}^2y_{2}^2y_{3}+y_{3} (y_{1}y_{2}^2y_{1}   - y_{2}y_{1}^2y_{2})     \nonumber\allowdisplaybreaks
\end{align*}
where all underlined monomials in $y_i$ $(1 \le i \le 3)$ with the factor $(q+q^{-1})$ are replaced by \eqref{eq: C3 q-serre}, and then one can see that the last expression is equal to \eqref{eq: stepp11}.
This completes the proof.
\end{proof}

\begin{proposition} \label{prop: 4-move dobly-laced}
 For $\g$ of type $B_n$ or $C_n$ $($resp.~$F_4$$)$, we have
$$\TT_{n}\TT_{n-1}\TT_{n}\TT_{n-1}   = \TT_{n-1}\TT_{n}\TT_{n-1}\TT_{n}  \qquad \text{$($resp. } \TT_2\TT_3\TT_2\TT_3=\TT_3\TT_2\TT_3\TT_2)$$
\end{proposition}

\begin{proof} Since the proofs for $B_n$, $C_n$ and $F_4$ are all similar, we only give the proof when $\g=B_n$.
We claim that
$$
\TT_{n} \TT_{n-1}\TT_{n}\TT_{n-1}(\sfy_{k,m}) =  \TT_{n-1}\TT_{n}\TT_{n-1}\TT_{n}(\sfy_{k,m})
$$
for all $1 \le k \le n$. The case when $k<n-2$ is obvious by definition.
Let us verify three non-trivial cases as follows, which prove our claim.
\smallskip

\noindent
{\it Case 1}.
Let $k=n-1$.  For the time being, we omit $m+1$ for simplicity of notation to write $\sfy_i = \sfy_{i,m+1}$ and $L_q^{\widetilde{Q}}(\be) = L_q^{\widetilde{Q}}(\be,m+1)$ for a Dynkin quiver $\widetilde{Q}$.
By taking a Dynkin quiver $Q$ with source at $n-1$,
\fontsize{10}{10}
\begin{align}
 & (q-q^{-1})  \tUptheta_{Q} \circ  \TT_{n-1}\TT_{n}\TT_{n-1}(\sfy_{n-1,m})=   (q-q^{-1}) \tUptheta_{Q} \circ \TT_{n-1}\TT_{n}(\sfy_{n-1})  \nonumber\allowdisplaybreaks \\
& = (q-q^{-1}) \tUptheta_{Q} \circ \TT_{n-1} \left(\dfrac{  q \sfy_{n-1}  \sfy_{n}^{(2)} - \sfy_{n}  \sfy_{n-1} \sfy_{n} +  q^{-1}\sfy_{n}^{(2)} \sfy_{n-1} }{(q-q^{-1})^{2}} \right)  \nonumber \allowdisplaybreaks\\
&=
\left(\dfrac{q  L_q^Q(\al_{n-1},m+2) *L_q^Q( \al_{n-1}+\al_n)  - q^{-1}L_q^Q( \al_{n-1}+\al_n) * L_q^Q(\al_{n-1},m+2)}{q^2-q^{-2}} \right) * L_q^Q( \al_{n-1}+\al_n)
\nonumber \allowdisplaybreaks\\
& \,\, -   L_q^Q( \al_{n-1}+\al_n) *
\left(
\dfrac{  q  L_q^Q(\al_{n-1},m+2) *L_q^Q( \al_{n-1}+\al_n)  - q^{-1}L_q^Q( \al_{n-1}+\al_n) * L_q^Q(\al_{n-1},m+2) }{(q^2-q^{-2})}
\right) \nonumber \allowdisplaybreaks \\
& =
\left(\dfrac{q  L_q^{s_{n-1}Q}(\al_{n-1}) *L_q^{s_{n-1}Q}( \al_n)  - q^{-1}L_q^{s_{n-1}Q}(\al_n) * L_q^{s_{n-1}Q}(\al_{n-1})}{q^2-q^{-2}} \right) * L_q^{s_{n-1}Q}( \al_n)
\nonumber \allowdisplaybreaks\\
& \,\, -   L_q^{s_{n-1}Q}( \al_n)
* \left(
\dfrac{  q  L_q^{s_{n-1}Q}(\al_{n-1}) *L_q^{s_{n-1}Q}(\al_n)  - q^{-1}L_q^{s_{n-1}Q}(\al_n) * L_q^{s_{n-1}Q}(\al_{n-1}) }{(q^2-q^{-2})}
\right)
\nonumber \allowdisplaybreaks \\
&  \overset{\ddagger}{=}
L_q^{s_{n-1}Q}(\al_{n-1}+\al_{n})*L_q^{s_{n-1}Q}(\al_{n}) - L_q^{s_{n-1}Q}(\al_{n})*L_q^{s_{n-1}Q}(\al_{n-1}+\al_{n})
\nonumber \allowdisplaybreaks \\
&  \overset{\star}{=}  (q-q^{-1}) L_q^{s_{n-1}Q}(  \al_{n-1}+2\al_n) = (q-q^{-1}) L_q^{Q}(  \al_{n-1}+2\al_n)
 \nonumber \allowdisplaybreaks\\
& \overset{\dagger}{=}
L_q^{Q}(   \al_n)*L_q^{Q}(   \al_{n-1}+ \al_n)  -L_q^{Q}(   \al_{n-1}+ \al_n)* L_q^{Q}(   \al_n),\label{eq: rstep1}
\end{align}
\fontsize{11}{11}

\noindent
 where $ \overset{\ddagger}{=}$, $\overset{\star}{=}$, and
$ \overset{\dagger}{=}$ follow from ~\eqref{eq: BKMc2}.
Since~\eqref{eq: BKMcp} tells us that
\begin{equation}\label{eq: rep}
\begin{aligned}
& L_q^{Q}(   \al_{n-1}+ \al_n,m+1) \\
& \hspace{5ex}=  \dfrac{q L_q^{Q}(   \al_{n},m+1)*L_q^{Q}(   \al_{n-1},m+1) - q^{-1}L_q^{Q}(   \al_{n-1},m+1)* L_q^{Q}(   \al_{n},m+1)}{q^2-q^{-2}},
\end{aligned}
\end{equation}
we have
\begin{equation}\label{eq: 4move n-1 0}
\begin{aligned}
\TT_{n-1}\TT_{n}(\sfy_{n-1,m+1})  & =
\dfrac{q \sfy_{n,m+1}^{(2)} \sfy_{n-1,m+1} - \sfy_{n,m+1} \sfy_{n-1,m+1} \sfy_{n,m+1} +q^{-1}  \sfy_{n-1,m+1} \sfy_{n,m+1}^{(2)} }{(q-q^{-1})^2} \\
& = \TT_{n}^{-1}(\sfy_{n-1,m+1}),
\end{aligned}
\end{equation}
by replacing $L_q^{Q}(   \al_{n-1}+ \al_n )$ in~\eqref{eq: rstep1} with~\eqref{eq: rep}.
By \eqref{eq: 4move n-1 0}, we have
\begin{equation}\label{eq: 4move n-1}
\begin{aligned}
& \TT_{n}\TT_{n-1}\TT_{n}\TT_{n-1}(\sfy_{n-1,m})= \TT_{n}\TT_{n-1}\TT_{n}(\sfy_{n-1,m+1}) = \TT_{n}\TT_{n}^{-1}(\sfy_{n-1,m+1}) = \sfy_{n-1,m+1}.
\end{aligned}
\end{equation}
and
\begin{align*}
  \TT_{n-1}\TT_{n}\TT_{n-1}\TT_{n}(\sfy_{n-1,m})
&=   \TT_{n-1}\TT_{n}\TT_{n}^{-1}(\sfy_{n-1,m}) =   \TT_{n-1}(\sfy_{n-1,m}) =\sfy_{n-1,m+1}.
\end{align*}

\noindent
{\it Case 2}.
Let us consider $k=n$. Take a Dynkin quiver $Q$ with source $n$.
By a similar argument as in {\it Case 1} with \eqref{eq: BKMc2}, we have
\begin{align*}
& \tUptheta_{Q} \circ\TT_{n}\TT_{n-1}\TT_{n}(\sfy_{n,m})   = \tUptheta_{Q} \circ\TT_{n}\TT_{n-1}(\sfy_{n,m+1})   \allowdisplaybreaks\\
& =\tUptheta_{Q} \circ\TT_{n} \left( \dfrac{q \sfy_{n,m+1}\sfy_{n-1,m+1} - q^{-1} \sfy_{n-1,m+1} \sfy_{n,m+1} }{q^2-q^{-2}} \right) \allowdisplaybreaks \\
& =\dfrac{ q  L_q^Q(\al_{n},m+2) *L_q^Q(\al_{n-1}+2\al_n,m+1) - q^{-1}  L_q^Q(\al_{n-1}+2\al_n,m+1)* L_q^Q(\al_{n},m+2) }{q^2-q^{-2}}\allowdisplaybreaks \\
& =\dfrac{ q  L_q^{s_nQ}(\al_{n},m+1) * L_q^{s_nQ}(\al_{n-1},m+1) - q^{-1}  L_q^{s_nQ}(\al_{n-1},m+1)* L_q^{s_nQ}(\al_{n},m+1) }{q^2-q^{-2}}\allowdisplaybreaks \\
&  = L_q^{s_nQ}(\al_{n-1}+\al_{n},m+1)   = L_q^{Q}(\al_{n-1}+\al_{n},m+1)  \\
&=   \dfrac{q L_q^Q(\al_{n-1},m+1)*L_q^Q(\al_{n},m+1) - q^{-1}L_q^Q(\al_{n},m+1)* L_q^Q(\al_{n-1},m+1) }{q^2-q^{-2}} \\
&= \tUptheta_{Q} \circ \TT_{n-1}^{-1} (\sfy_{n,m+1}).
\end{align*}
Hence, it yields
$$
\TT_{n-1}\TT_{n}\TT_{n-1}\TT_{n}(\sfy_{n,m}) = \sfy_{n,m+1},
$$
and
\begin{align*}
  \TT_{n}\TT_{n-1}\TT_{n}\TT_{n-1}(\sfy_{n,m}) &=   \TT_{n}\TT_{n-1}\TT_{n}\TT_{n-1}\TT_{n}(\sfy_{n,m-1}) = \TT_{n}(\sfy_{n,m}) = \sfy_{n,m+1}.
\end{align*}

\noindent
{\it Case 3}.
Let us consider $k=n-2$.
 We have
\begin{align}
& \TT_{n}\TT_{n-1}  \TT_{n} \TT_{n-1}(\sfy_{n-2,m})
=  \TT_{n} \TT_{n-1}  \TT_{n} \left( \dfrac{q\sfy_{n-2,m}\sfy_{n-1,m} - q^{-1}\sfy_{n-1,m}\sfy_{n-2,m} }{q^2-q^{-2}} \right) \allowdisplaybreaks \nonumber \\
&\overset{\star}{=}    \left( \dfrac{q\TT_{n} \TT_{n-1}  (\sfy_{n-2,m}) \sfy_{n-1,m} - q^{-1}\sfy_{n-1,m}\TT_{n} \TT_{n-1}  (\sfy_{n-2,m}) }{q^2-q^{-2}} \right). \label{eq: goal LHS}
\end{align}
Here $ \overset{\star}{=}$ follows from ~\eqref{eq: 4move n-1}.
On the other hand,
\begin{align}
&\TT_{n-1} \TT_{n}\TT_{n-1}  \TT_{n} (\sfy_{n-2,m})    =   \TT_{n-1}  \TT_{n} \TT_{n-1} (\sfy_{n-2,m})  \allowdisplaybreaks \nonumber \\
&    =   \TT_{n-1}  \TT_{n} \left( \dfrac{q\sfy_{n-2,m}\sfy_{n-1,m} - q^{-1}\sfy_{n-1,m}\sfy_{n-2,m} }{q^2-q^{-2}} \right) \allowdisplaybreaks \nonumber \\
&\overset{\dagger}{=}  \left( \dfrac{q \TT_{n-1} (\sfy_{n-2,m}) \TT_{n}^{-1} (\sfy_{n-1,m}) - q^{-1}\TT_{n}^{-1}(\sfy_{n-1,m} )\TT_{n-1} (\sfy_{n-2,m}) }{q^2-q^{-2}} \right).  \label{eq: goal RHS}
\end{align}
Here $ \overset{\dagger}{=}$ follows from ~\eqref{eq: 4move n-1 0}.
To see the equality between ~\eqref {eq: goal LHS} and~\eqref{eq: goal RHS}, it suffices to assume that $n=3$ and $m=0$. Set $y_i\seteq \sfy_{i,0}$ for $i=1,2,3$ and recall that  $q_1=q_2=q^2$, and $q_3=q$.
By direct computation, we have
\begin{equation} \label{eq: should be zero}
\begin{aligned}
& \dfrac{
(q\TT_{3} \TT_{2}  (y_{1})y_2  - q^{-1}y_2\TT_{3} \TT_{2}  (y_{1})) - (q\TT_{2} (y_{1}) \TT_{3}^{-1} (y_{2})-q^{-1}\TT_{3}^{-1} (y_{2})\TT_{2} (y_{1}))}{(q^2-q^{-2})} \\
& =   \left\{ (q^3+q)(y_{1}y_{2}y_{3}y_{2}y_{3} -  y_{1}y_{3}y_{2}y_{3}y_{2} )   + (q- q^{-1})( y_{3}^2y_{2}y_{1}y_{2}   - y_{2}y_{1}y_{2}y_{3}^2 )  \right.\\
&  \left. \hspace{4ex}  +(q^{-1}+q^{-3})(y_{3}y_{2}y_{3}y_{2}y_{1}   - y_{2}y_{3}y_{2}y_{3}y_{1})  +  qy_{3}^2y_{1}y_{2}^2 +  q^{-1}   y_{2}^2y_{1}y_{3}^2     -qy_{1}y_{2}^2y_{3}^2      -q^{-1}y_{3}^2y_{2}^2y_{1} \right\} \\
&   \hspace{12ex} \times (q-q^{-1})^{-1}(q^2-q^{-2})^{-2}.
\end{aligned}
\end{equation}
Now we claim that~\eqref{eq: should be zero} vanishes.
Note that
\begin{eqnarray} &&
\parbox{90ex}{
\bna
\item  \label{it: 21}  $y_2^2y_1 - (q^2+q^{-2})y_2y_1y_2 + y_1y_2^2 = 0 \iff (q^2+q^{-2})y_2y_1y_2 =  y_2^2y_1+ y_1y_2^2 $,
\item \label{it: 23}  $y_2^2y_3 - (q^2+q^{-2})y_2y_3y_2 + y_3y_2^2 = 0 \iff (q^2+q^{-2})y_2y_3y_2 = y_2^2y_3+ y_3y_2^2$.
\ee
}\label{eq: B3 q-serre}
\end{eqnarray}
Then we have
\begin{align*}
& \underline{(q^2+q^{-2})}  \times \left\{ (q^3+q)(y_{1}\underline{y_{2}y_{3}y_{2}}y_{3} -  y_{1}y_{3}\underline{y_{2}y_{3}y_{2}} )
+ (q- q^{-1})( y_{3}^2\underline{y_{2}y_{1}y_{2}}   - \underline{y_{2}y_{1}y_{2}}y_{3}^2 ) \right.\\
& \hspace{4ex} \left. +(q^{-1}+q^{-3})(y_{3}\underline{y_{2}y_{3}y_{2}}y_{1}   - \underline{y_{2}y_{3}y_{2}}y_{3}y_{1})
+  qy_{3}^2y_{1}y_{2}^2 +  q^{-1}   y_{2}^2y_{1}y_{3}^2     -qy_{1}y_{2}^2y_{3}^2      -q^{-1}y_{3}^2y_{2}^2y_{1} \right\} \\
& =
 (q^3+q) \left( y_{1}y_2^2y_3^2 - y_{1}y_{3}^2y_2^2   \right)  \quad (\because ~\eqref{eq: B3 q-serre}~\eqref{it: 23})\\
& \hspace{3ex}+ (q-q^{-1})  \left(  y_{3}^2y_2^2y_1+ y_{3}^2y_1y_2^2    -y_2^2y_1y_{3}^2- y_1y_2^2y_{3}^2  \right)
\quad (\because ~\eqref{eq: B3 q-serre}~\eqref{it: 21})\\
& \hspace{6ex}+ (q^{-1}+q^{-3})  \left(  y_3^2y_2^2y_{1}    -   y_2^2y_3^2y_{1}   \right) \quad (\because ~\eqref{eq: B3 q-serre}~\eqref{it: 23}) \\
& \hspace{9ex} +  (q^3+q^{-1})y_{3}^2y_{1}y_{2}^2 +  (q+q^{-3})   y_{2}^2y_{1}y_{3}^2     - (q^3+q^{-1})y_{1}y_{2}^2y_{3}^2      - (q+q^{-3})y_{3}^2y_{2}^2y_{1} \\
& =0,
\end{align*}
as desired,
where all underlined monomials in $y_i$ $(1 \le i \le 3)$ with the factor $(q^2+q^{-2})$ are replaced by \eqref{eq: B3 q-serre}.
\end{proof}

\begin{proposition} \label{prop: G2 barid}
For $\g=G_2$, we have
$$\TT_1\TT_2\TT_1\TT_2\TT_1\TT_2  = \TT_2\TT_1\TT_2\TT_1\TT_2\TT_1.$$
\end{proposition}

\begin{proof} It suffices to show that
$$\TT_1\TT_2\TT_1\TT_2\TT_1\TT_2(\sfy_{i,m}) = \TT_2\TT_1\TT_2\TT_1\TT_2\TT_1 (\sfy_{i,m}) \qquad (i=1,2).
$$
In this proof, we omit $*$-operation and $m+1$ frequently for simplicity of notation. Note that $q_1=q$ and $q_2=q^3$.
Let us take a Dynkin quiver $Q'$ with source $2$ and $\xi_2=5$ whose $\Gamma^{Q'}$ is described as
$$
\Gamma^{Q'} = \raisebox{3em}{ \scalebox{0.9}{\xymatrix@!C=4ex@R=2ex{
(i\setminus p)  & 0 & 1 & 2&3&4&5\\
1&   \al_1  \ar@{=>}[dr]  && 2\al_1+\al_2  \ar@{=>}[dr]     &&\al_1+\al_2 \ar@{=>}[dr] \\
2&&  3\al_1+\al_2  \ar@{-}[ul]\ar@{->}[ur]  && 3\al_1+2\al_2 \ar@{-}[ul]\ar@{->}[ur]  &&\al_2 \ar@{-}[ul]
}}},
$$
and set $Q\seteq s_2Q'$.
We remark that the height function $\xi$ of $Q$ is given by $\xi_1 = 4$ and $\xi_2 = 3$
(cf.~\eqref{eq: G2 Q} below for $\Gamma^Q$).
Then we have
\fontsize{10}{10}
\begin{align} \label{eq:Q' T2T1T2y2 in G2}
\begin{split}
&\tUptheta_{Q'} \circ\TT_2\TT_1\TT_2(\sfy_{2,m}) =\tUptheta_{Q'} \circ\TT_2\TT_1 (\sfy_{2})  \allowdisplaybreaks\\
&=\tUptheta_{Q'} \circ \TT_2 \left(  \dfrac{   q^{3/2} \sfy_{2}\sfy_{1}^3 -q^{1/2}[3]_1\sfy_{1}\sfy_{2}\sfy_{1}^2 + q^{-1/2}[3]_1\sfy_{1}^2\sfy_{2}\sfy_{1} - q^{-3/2} \sfy_{1}^3\sfy_{2} }{(q-q^{-1})(q^2-q^{-2})(q^3-q^{-3})}  \right) \allowdisplaybreaks\\
& = \left\{ q^{3/2}L^{Q'}_q(\al_2,m+2)L^{Q'}_q( \al_1+\al_2)^3 - q^{1/2}[3]_1L^{Q'}_q( \al_1+\al_2)L^{Q'}_q(\al_2,m+2)L^{Q'}_q( \al_1+\al_2)^2 \right. \allowdisplaybreaks\\
& \hspace{6ex} \left. - q^{-1/2}[3]_1L^{Q'}_q(\al_1+\al_2)^2L^{Q'}_q(\al_2,m+2)L^{Q'}_q( \al_1+\al_2) +q^{-3/2} L^{Q'}_q( \al_1+\al_2)^3L^{Q'}_q(\al_2,m+2) \right\} \allowdisplaybreaks\\
& \hspace{54ex}  \times (q-q^{-1})^{-1}(q^2-q^{-2})^{-1}(q^3-q^{-3})^{-1} \allowdisplaybreaks\\
& = \left\{ q^{3/2}L^{Q}_q(\al_2)L^{Q}_q(  \al_1)^3 - q^{1/2}[3]_1L^{Q}_q( \al_1)L^{Q}_q(\al_2)L^{Q}_q( \al_1)^2 \right. \allowdisplaybreaks\\
& \hspace{32ex} \left. - q^{-1/2}[3]_1L^{Q}_q(\al_1)^2L^{Q}_q(\al_2)L^{Q}_q( \al_1) +q^{-3/2} L^{Q}_q( \al_1)^3L^{Q}_q(\al_2) \right\} \allowdisplaybreaks\\
& \hspace{54ex}  \times (q-q^{-1})^{-1}(q^2-q^{-2})^{-1}(q^3-q^{-3})^{-1} \allowdisplaybreaks\\
& = L^{Q}_q(3\al_1+\al_2) = L^{Q'}_q(3\al_1+2\al_2)  \allowdisplaybreaks\\
& =  \dfrac{ q^{-1/2}  L^{Q'}_q(2\al_1+\al_2)  L^{Q'}_q( \al_1+\al_2) - q^{1/2} L^{Q'}_q( \al_1+\al_2) L^{Q'}_q(2\al_1+\al_2)   }{q-q^{-1}},	
\end{split}
\end{align}
\fontsize{11}{11}
by~\eqref{eq: BKMc2}.
Since
\begin{align*}
&  L^{Q'}_q( \al_1+\al_2) =  \dfrac{q^{3/2} L^{Q'}_q( \al_1) L^{Q'}_q( \al_2) - q^{-3/2}  L^{Q'}_q( \al_2) L^{Q'}_q( \al_1)}{q^3-q^{-3}} \allowdisplaybreaks\\
& \hspace{18ex}  \iff    \tUptheta_{Q'}^{-1}(L^{Q'}_q( \al_1+\al_2)) = \dfrac{q^{3/2} \sfy_{1}\sfy_{2}- q^{-3/2}  \sfy_{2}\sfy_{1}}{q^3-q^{-3}}, \allowdisplaybreaks\\
& L^{Q'}_q( 2\al_1+\al_2)  =\dfrac{q^{1/2} L_q^{Q'}( \al_1) L_q^{Q'}(\al_1+\al_2)    -  q^{-1/2}L_q^{Q'}(\al_1+\al_2)  L_q^{Q'}( \al_1) }{q^2-q^{-2}}\allowdisplaybreaks \\
& \hspace{18ex}  \iff   \tUptheta_{Q'}^{-1}(L^{Q'}_q( 2\al_1+\al_2)) = \left( \dfrac{q^{2} \sfy_{1}^2\sfy_{2}- (q+q^{-1}) \sfy_{1}  \sfy_{2}\sfy_{1} + q^{-2}  \sfy_{2}\sfy_{1}^2}{(q^3-q^{-3})(q^2-q^{-2})} \right),
\end{align*}
the equality \eqref{eq:Q' T2T1T2y2 in G2} implies that
\fontsize{10}{10}
\begin{align} \label{eq:T2T1y2}
\begin{split}
&\TT_2\TT_1\TT_2(\sfy_{2,m}) =\TT_2\TT_1 (\sfy_{2}) \allowdisplaybreaks\\
& =\left( q^{-1/2} (q^2\sfy_{1}^2\sfy_{2}-(q+q^{-1})\sfy_{1}\sfy_{2}\sfy_{1} +q^{-2}\sfy_{2}\sfy_{1}^2  ) (q^{3/2}\sfy_{1}\sfy_{2}-q^{-3/2}\sfy_{2}\sfy_{1})  \right. \allowdisplaybreaks\\
& \hspace{5ex} \left.  - q^{1/2} (q^{3/2}\sfy_{1}\sfy_{2}-q^{-3/2}\sfy_{2}\sfy_{1}) (q^2\sfy_{1}^2\sfy_{2}-(q+q^{-1})\sfy_{1}\sfy_{2}\sfy_{1} +q^{-2}\sfy_{2}\sfy_{1}^2  )  \right) \allowdisplaybreaks\\
& \hspace{45ex} \times (q-q^{-1})^{-1} (q^2-q^{-2})^{-1}(q^3-q^{-3})^{-2} \allowdisplaybreaks\\
& =\left\{   q^3\sfy_{1}^2\sfy_{2}\sfy_{1}\sfy_{2}-q^2(q^2+1+q^{-2})\sfy_{1}\sfy_{2}\sfy_{1}^2\sfy_{2} +(q+q^{-1})\sfy_{2}\sfy_{1}^3\sfy_{2}     \right. \allowdisplaybreaks\\
& \hspace{5ex}
+ q^2(q+q^{-1})\sfy_{1}\sfy_{2}\sfy_{1}\sfy_{2}\sfy_{1}  -(1+q^{-2}+q^{-4})\sfy_{2}\sfy_{1}^2\sfy_{2}\sfy_{1} +q^{-3}\sfy_{2}\sfy_{1}\sfy_{2}\sfy_{1}^2  \allowdisplaybreaks\\
& \hspace{10ex} \left. -  \sfy_{1}( \sfy_{2}^2\sfy_{1}- (q^3+q^{-3})\sfy_{2}\sfy_{1}\sfy_{2} +\sfy_{1}\sfy_{2}^2 ) \sfy_{1}
  \right\} \times (q-q^{-1})^{-1} (q^2-q^{-2})^{-1}(q^3-q^{-3})^{-2}  \allowdisplaybreaks\\
& \overset{\dagger}{=}\left\{    q^3\sfy_{1}^2\sfy_{2}\sfy_{1}\sfy_{2}-q^2(q^2+1+q^{-2})\sfy_{1}\sfy_{2}\sfy_{1}^2\sfy_{2} +(q+q^{-1})\sfy_{2}\sfy_{1}^3\sfy_{2}     \right. \allowdisplaybreaks\\
& \hspace{5ex} \left.
+ q^2(q+q^{-1})\sfy_{1}\sfy_{2}\sfy_{1}\sfy_{2}\sfy_{1}  -(1+q^{-2}+q^{-4})\sfy_{2}\sfy_{1}^2\sfy_{2}\sfy_{1} +q^{-3}\sfy_{2}\sfy_{1}\sfy_{2}\sfy_{1}^2 \right\} \allowdisplaybreaks\\
& \hspace{45ex}  \times  (q-q^{-1})^{-1} (q^2-q^{-2})^{-1}(q^3-q^{-3})^{-2}  \allowdisplaybreaks \\
& = \dfrac{
q^3\sfy_{1}^2\sfy_{2}\sfy_{1}\sfy_{2}-q^2 [3]_1\sfy_{1}\sfy_{2}\sfy_{1}^2\sfy_{2} +[2]_1\sfy_{2}\sfy_{1}^3\sfy_{2}   -q^{-2}[3]_1 \sfy_{2}\sfy_{1}^2\sfy_{2}\sfy_{1}  + q^2[2]_1\sfy_{1}\sfy_{2}\sfy_{1}\sfy_{2}\sfy_{1}   +q^{-3}\sfy_{2}\sfy_{1}\sfy_{2}\sfy_{1}^2
}{ (q-q^{-1})(q^2-q^{-2})(q^3-q^{-3})^{2}},
\end{split}
\end{align}
\fontsize{11}{11}
where $\overset{\dagger}{=}$ follows from~\eqref{eq: q-serre y}.
Therefore we obtain
\begin{align*}
&\tUptheta_{Q'} \circ \TT_2\TT_1 (\sfy_{2}) = L^{Q'}_q(3\al_1+2\al_2)\\
&=
\dfrac{q^3 \calK^2\calR\calK\calR-q^2[3]_1\calK\calR\calK^2\calR +[2]_1\calR\calK^3\calR   -q^{-2}[3]_1 \calR\calK^2\calR\calK  + q^2[2]_1\calK\calR\calK\calR\calK   +q^{-3}\calR\calK\calR\calK^2
}{ (q-q^{-1})(q^2-q^{-2})(q^3-q^{-3})^{2}}
\end{align*}
where
$\calK \seteq  L_q^{Q'}(\al_1)$ and $\calR \seteq  L_q^{Q'}(\al_2)$.
Here one can observe that
\begin{equation} \label{eq:shift by -2 in G2}
\begin{split}
    L_q^Q(\al_1,m+2) &= L_q^{s_2Q'}(\al_1,m+2) = L_q^{Q'}(\al_1+\al_2,m+2) = \scrS_{-2}(L_q^{Q'}(\al_1)), \\
    L_q^Q(3\al_1+\al_2) &= L_q^{s_2Q'}(3\al_1+\al_2) = L_q^{Q'}(3\al_1+2\al_2) = \scrS_{-2}(L_q^{Q'}(\al_2)).
\end{split}
\end{equation}
where $\scrS_{-2}$ is defined in \eqref{eq: shift of spectral parameters}.
By \eqref{eq:T2T1y2} and \eqref{eq:shift by -2 in G2}, we have
\begin{align*}
&  \tUptheta_{Q} \circ \TT_1\TT_2\TT_1\TT_2(\sfy_{2,m}) =\tUptheta_{Q} \circ \TT_1(\TT_2\TT_1 (\sfy_{2}) )  \allowdisplaybreaks\\
& =
\dfrac{
q^3 K^2RKR-q^2[3]_1KRK^2R +[2]_1RK^3R   -q^{-2}[3]_1 RK^2RK  + q^2[2]_1KRKRK   +q^{-3}RKRK^2
}{ (q-q^{-1})(q^2-q^{-2})(q^3-q^{-3})^{2}},
\end{align*}
where
$K \seteq L_q^{Q}(\al_1,m+2)=\scrS_{-2}(L_q^{Q'}(\al_1))$ and $R \seteq L_q^{Q}(\al_2+3\al_1)=\scrS_{-2}(L_q^{Q'}(\al_2))$.
Thus we have
\begin{equation} \label{eq:Q T1T2T1y2 in G2}
\tUptheta_{Q} \circ \TT_1\TT_2\TT_1 (\sfy_{2})=\scrS_{-2}( L_q^{Q'}(3\al_1+2\al_2)) = L_q^{Q}(3\al_1+2\al_2).
\end{equation}
On the other hand, we have
\begin{align} \label{eq:Q T2inv T1inv y2 in G2}
\begin{split}
&\tUptheta_{Q}  \circ\TT_2^{-1}\TT_1^{-1}(\sfy_{2}) \allowdisplaybreaks \\
& = \tUptheta_{Q}  \circ \TT^{-1}_2 \left(  \dfrac{   q^{3/2} \sfy_{1}^3\sfy_{2} -q^{1/2}[3]_1\sfy_{1}^2\sfy_{2}\sfy_{1} + q^{-1/2}[3]_1\sfy_{1}\sfy_{2,m}\sfy_{1}^2 - q^{-3/2} \sfy_{2}\sfy_{1}^3 }{(q-q^{-1})(q^2-q^{-2})(q^3-q^{-3})}  \right) \allowdisplaybreaks \\
& = \left\{ q^{3/2}  L_q^{Q}(\al_1+\al_2)^3 L_q^{Q}(\al_2,m) - q^{1/2}[3]_1L_q^{Q}(\al_1+\al_2)^2 L_q^{Q}(\al_2,m) L_q^{Q}(\al_1+\al_2) \right.\allowdisplaybreaks\\
&   \hspace{6ex}\left. + q^{-1/2}[3]_1 L_q^{Q}(\al_1+\al_2) L_q^{Q}(\al_2,m) L_q^{Q}(\al_1+\al_2)^2 - q^{-3/2}   L_q^{Q}(\al_2,m) L_q^{Q}(\al_1+\al_2)^3 \right\} \allowdisplaybreaks \\
& \hspace{52ex}  \times (q-q^{-1})^{-1}(q^2-q^{-2})^{-1}(q^3-q^{-3})^{-1}  \allowdisplaybreaks\\
& = \left\{ q^{3/2}  L_q^{Q'}(\al_1)^3 L_q^{Q'}(\al_2) - q^{1/2}[3]_1L_q^{Q'}(\al_1)^2 L_q^{Q'}(\al_2) L_q^{Q'}(\al_1) \right. \allowdisplaybreaks \\
&  \hspace{16ex} \left. + q^{-1/2}[3]_1 L_q^{Q'}(\al_1) L_q^{Q'}(\al_2) L_q^{Q'}(\al_1)^2 - q^{-3/2}   L_q^{Q'}(\al_2) L_q^{Q'}(\al_1)^3 \right\} \allowdisplaybreaks \\
& \hspace{52ex}  \times (q-q^{-1})^{-1}(q^2-q^{-2})^{-1}(q^3-q^{-3})^{-1}  \allowdisplaybreaks\\
&  = L_q^{Q'}(3\al_1+\al_2) = L_q^{Q}(3\al_1+2\al_2).
\end{split}
\end{align}
By combining \eqref{eq:Q T1T2T1y2 in G2} with \eqref{eq:Q T2inv T1inv y2 in G2}, we obtain
\begin{align*}
\TT_1\TT_2\TT_1\TT_2\TT_1\TT_2(y_{2,m}) = y_{2,m+1}.
\end{align*}
This implies that
\begin{align*}
\TT_2\TT_1\TT_2\TT_1\TT_2\TT_1(y_{2,m}) =\TT_2\TT_1\TT_2\TT_1\TT_2\TT_1\TT_2(y_{2,m-1})=\TT_2(y_{2,m})=y_{2,m+1}.
\end{align*}
Similarly, one can see that
\begin{align*}
\TT_1\TT_2\TT_1\TT_2\TT_1\TT_2(y_{1,m}) = y_{1,m+1}=\TT_2\TT_1\TT_2\TT_1\TT_2\TT_1(y_{1,m}),
\end{align*}
as desired. We complete the proof.
\end{proof}

\begin{proof}[Proof of Theorem~\ref{thm: braid main}]
In Propositions \ref{prop: commutation}--\ref{prop: G2 barid}, we have shown that
\begin{align*}
\underbrace{\TT_i\TT_j \cdots}_{\text{ $h_{i,j}$-times}} & =\underbrace{\TT_j\TT_i \cdots}_{\text{ $h_{i,j}$-times}},
\end{align*}
which amounts to a proof of the theorem.
\end{proof}

\appendix

\section{Interpretation of $\TT_1(\sfy_{2,m})$ when $\g=G_2$} \label{app: G}
\noindent
Recall that $q_1=q$, $q_2=q^3$ and
$$
\TT_1(\sfy_{2,m}) = \dfrac{   q^{3/2} \sfy_{2,m}\sfy_{1,m}^3 -q^{1/2}[3]_1\sfy_{1,m}\sfy_{2,m}\sfy_{1,m}^2 + q^{-1/2}[3]_1\sfy_{1,m}^2\sfy_{2,m}\sfy_{1,m} - q^{-3/2} \sfy_{1,m}^3\sfy_{2,m} }{(q-q^{-1})(q^2-q^{-2})(q^3-q^{-3})}.
$$
Let us take the Dynkin quiver $Q$ with source $1$ and $\xi_1=3$, whose $\Gamma^Q$ is described as follows:
\begin{align}\label{eq: G2 Q}
\Gamma^Q= \raisebox{2.3em}{ \scalebox{0.8}{\xymatrix@!C=4ex@R=2ex{
(i\setminus p)  & -2 & -1 & 0 & 1 & 2&3  \\
1&&   \al_1+\al_2  \ar@{=>}[dr]  && 2\al_1+\al_2  \ar@{=>}[dr]     &&\al_1  \\
2&\al_2 \ar@{->}[ur] &&  3\al_1+2\al_2  \ar@{-}[ul]\ar@{->}[ur]  && 3\al_1+\al_2 \ar@{-}[ul]\ar@{->}[ur]
}}}
\end{align}

\noindent
From the viewpoint of Corollary \ref{cor: HLO iso},
the relation~\eqref{eq: BKMcp} tells us that
\bnum
\item \label{it:1} $\dfrac{q^{-1/2}L_q^Q(2\al_1+\al_2,m)  * L_q^Q( \al_1,m)   -  q^{1/2} L_q^Q( \al_1,m)*L_q^Q(2\al_1+\al_2,m)}{q-q^{-1}} =L_q^Q(3\al_1+\al_2,m)$,
\item \label{it:2} $\dfrac{q^{1/2}L_q^Q(\al_1+\al_2,m)  * L_q^Q( \al_1,m)   -  q^{-1/2} L_q^Q( \al_1,m)*L_q^Q(\al_1+\al_2,m)}{q^2-q^{-2}} =L_q^Q(2\al_1+\al_2,m)$,
\item \label{it:3} $\dfrac{q^{3/2}L_q^Q(\al_2,m)  * L_q^Q( \al_1,m)   -  q^{-3/2} L_q^Q( \al_1,m)*L_q^Q(\al_2,m)}{q^3-q^{-3}} =L_q^Q(\al_1+\al_2,m)$,
\ee
since
\bnump
\item $p_{2\al_1+\al_2,\al_1}=2$ and $(2\al_1+\al_2,\al_1)=1$,
\item $p_{\al_1+\al_2,\al_1}=1$ and $(\al_1+\al_2,\al_1)=-1$,
\item $p_{\al_2,\al_1}=0$ and $(\al_2,\al_1)=-3$,
\ee
respectively. Replacing $L_q^Q(\al_1+\al_2,m)$ in~\eqref{it:2} with the LHS of~\eqref{it:3} and then replacing $L_q^Q(2\al_1+\al_2,m)$ in~\eqref{it:1} with the already-replaced one, we have
\begin{align*}
&L_q^Q(3\al_1+\al_2,m) \\
&= \dfrac{   q^{3/2} \sfx^Q_{2,m}(\sfx^Q_{1,m})^3 -q^{1/2}[3]_1\sfx^Q_{1,m}\sfx^Q_{2,m}(\sfx^Q_{1,m})^2 + q^{-1/2}[3]_1(\sfx^Q_{1,m})^2\sfx^Q_{2,m}\sfx^Q_{1,m} - q^{-3/2} (\sfx^Q_{1,m})^3\sfx^Q_{2,m} }{(q-q^{-1})(q^2-q^{-2})(q^3-q^{-3})}
\end{align*}
and hence
$$  \tUptheta_{Q} \circ \TT_1(\sfy_{2,m})   = L_q^Q(3\al_1+\al_2,m).$$

\end{document}